\documentclass[11pt,letterpaper,twoside,reqno,nosumlimits]{amsart}

\usepackage[usenames,dvipsnames]{xcolor}
\usepackage{fancyhdr}
\usepackage{amsmath,amsfonts,amsbsy,amsgen,amscd,amssymb,amsthm}
\usepackage{url}

\usepackage{mathtools}

\usepackage[font=small,margin=0.25in,labelfont={bf},labelsep={colon}]{caption}

\usepackage{subcaption}

\usepackage{tikz}
\usepackage{microtype}
\usepackage{enumitem}

\definecolor{dark-gray}{gray}{0.3}
\definecolor{dkgray}{rgb}{.4,.4,.4}
\definecolor{dkblue}{rgb}{0,0,.5}
\definecolor{medblue}{rgb}{0,0,.75}
\definecolor{rust}{rgb}{0.5,0.1,0.1}

\usepackage[colorlinks=true]{hyperref}

\hypersetup{urlcolor=rust}
\hypersetup{citecolor=dkblue}
\hypersetup{linkcolor=dkblue}

\usepackage{setspace}

\usepackage{graphicx}
\usepackage{booktabs,longtable,tabu} 
\usepackage{multirow} 

\usepackage{float}

\usepackage[full]{textcomp}

\usepackage[scaled=.98,sups]{XCharter}
\usepackage[scaled=1.04,varqu,varl]{inconsolata}
\usepackage[type1]{cabin}
\usepackage[charter,vvarbb,scaled=1.07]{newtxmath}
\usepackage[cal=boondoxo]{mathalfa}
\usepackage[T1]{fontenc}

\usepackage{bm} 

\graphicspath{{figures/}}

\newtheorem{theorem}{Theorem}[section]
\newtheorem{lemma}[theorem]{Lemma}

\newtheorem{proposition}[theorem]{Proposition}

\newtheorem{conjecture}[theorem]{Conjecture}
\newtheorem{corollary}[theorem]{Corollary}

\theoremstyle{definition}

\numberwithin{equation}{section} 
\numberwithin{figure}{section}
\numberwithin{table}{section}

\floatstyle{plaintop}
\newfloat{recipe}{thp}{lor}
\floatname{recipe}{Recipe}
\numberwithin{recipe}{section}

\providecommand{\mathbold}[1]{\bm{#1}}  
                                
\renewcommand{\phi}{\varphi}

\newcommand{\suml}{\sum\limits}

\newcommand{\econst}{\mathrm{e}}
\newcommand{\iunit}{\mathrm{i}}

\providecommand{\mathbbm}{\mathbb} 

\newcommand{\R}{\mathbbm{R}}

\newcommand{\sgn}[1]{\operatorname{sgn}{#1}}

\newcommand{\diff}[1]{\mathrm{d}{#1}}
\newcommand{\idiff}[1]{\, \diff{#1}}

\newcommand{\mtx}[1]{\mathbold{#1}}

\newcommand{\triplenorm}[1]{{\left\vert\kern-0.25ex\left\vert\kern-0.25ex\left\vert #1
    \right\vert\kern-0.25ex\right\vert\kern-0.25ex\right\vert}}

\allowdisplaybreaks

\usepackage[numbers]{natbib}

\newcommand{\om}{\omega}

\mathtoolsset{showonlyrefs}

\begin{document}

\title[Self-similar finite-time blowups of the generalized CLM model]{Self-similar finite-time blowups with smooth profiles\\ of the generalized Constantin--Lax--Majda model}
\author[D. Huang, X. Qin, X. Wang, and D. Wei]{De Huang$^1$, Xiang Qin, Xiuyuan Wang, and Dongyi Wei$^2$}
\thanks{$^1$ dhuang@math.pku.edu.cn (corresponding)}
\thanks{$^2$ jnwdyi@pku.edu.cn}
\maketitle

\vspace{-6mm}
\begin{center}
\textit{\small School of Mathematical Sciences, Peking University}
\end{center} 

\begin{abstract}
We show that the $a$-parameterized family of the generalized Constantin--Lax--Majda model, also known as the Okamoto--Sakajo--Wunsch model, admits exact self-similar finite-time blowup solutions with interiorly smooth profiles for all $a\leq 1$. Depending on the value of $a$, these self-similar profiles are either smooth on the whole real line or compactly supported and smooth in the interior of their closed supports. The existence of these profiles is proved in a consistent way by considering the fixed-point problem of an $a$-dependent nonlinear map, based on which detailed characterizations of their regularity, monotonicity, and far-field decay rates are established. Our work unifies existing results for some discrete values of $a$ and also explains previous numerical observations for a wide range of $a$.
\end{abstract}

\section{Introduction}
We consider the $1$D generalized Constantin--Lax--Majda (gCLM) equation
\begin{equation}\label{eqt:gCLM}
\om_t + au\om_x = u_x \om, \quad u_x = \mtx{H}(\om),\quad u(0)=0,
\end{equation}
for $x\in \R$, where $\mtx{H}(\cdot)$ denotes the Hilbert transform on the real line. The normalization condition $u(0)=0$ is not essential; we impose it throughout the paper to remove the degree of freedom due to translation. This equation is a $1$D model for the vorticity formulation of the $3$D incompressible Euler equations, proposed to study the competitive relation between advection and vortex stretching. In particular, $\om$ models the vorticity, and the nonlinear terms $u\om_x$ and $u_x \om$ model the advection term and the vortex stretching term, respectively. The $3$D Biot-Savart law that recovers the velocity from the vorticity is modeled by $u_x= \mtx{H}(\om)$, which has the same scaling as the original Biot-Savart law.

The fundamental question on the global regularity of the $3$D Euler equations with smooth initial data of finite energy remains one of the most challenging open problems in fluid dynamics. It is widely believed that the vortex stretching effect has the potential to induce an infinite growth of the vorticity in finite time. The advection, on the contrary, has been found to have a smoothing effect that may weaken the local growth of the solution and destroy the potential singularity formation (e.g., see \cite{okamoto2005role,hou2006dynamic,hou2008dynamic,lei2009stabilizing}). The first construction of a (stable) self-similar finite-time blowup for the 3D incompressible Euler equations on $\R^3$ was established by Elgindi \cite{elgindi2021finite} in the axisymmetric setting from $C^{1,\alpha}$ initial velocity for sufficiently small $\alpha$ (with stability of the blowup discussed in \cite{elgindi2021stability}). Under some change of variables, the $C^{1,\alpha}$ regularity of the solution gives rise to a small coefficient $\alpha$ in the advection term that weakens its effect, which is similar to the role of the parameter $a$ in \eqref{eqt:gCLM}. More recently, Chen and Hou \cite{chen2022stable} used rigorous computer-assisted proof to show for the first time that the 3D vortex stretching can actually dominate the advection and lead to an asymptotically self-similar finite-time singularity from smooth initial data in the presence of a solid boundary. However, whether a finite-time singularity can happen for the $3$D incompressible Euler equations from smooth initial data in the free space $\mathbb{R}^3$ still remain open. Hence, it is still worthwhile to work on simplified models to acquire better understanding of the potential blowup mechanism.

The original version of \eqref{eqt:gCLM} with $a=0$ was proposed by Constantin, Lax and Majda \cite{constantin1985simple} to demonstrate that a non-local vortex stretching term can lead to finite-time blowup in the absence of advection. Later, De Gregorio \cite{de1990one} included an advection term $u\om_x$ in the equation (known as the De Gregorio model) and conjectured the occurrence of a finite-time singularity none the less. As a generalization, Okamoto, Sakajo and Wunsch \cite{okamoto2008generalization} introduced the real parameter $a$ to modify the effect of advection in the competition against vortex stretching. Hence, equation \eqref{eqt:gCLM} is also referred to as the Okamoto--Sakajo--Wunsch (OSW) model.

Motivated by the long standing problem on finite-time blowup of the $3$D incompressible Euler equations, finite-time singularity formation of the gCLM model for a wide range of $a$ has been studied extensively in the literature. In view of the scaling property of equation \eqref{eqt:gCLM}, we are particularly interested in self-similar finite-time blowups of the form 
\begin{equation}\label{eqt:self-similar_solution}
\om(x,t) = (T-t)^{c_\om}\cdot\Omega\left(\frac{x}{(T-t)^{c_l}}\right),
\end{equation}
where $\Omega$ is referred to as the self-similar profile, and $c_\om,c_l$ are the scaling factors. Plugging this ansatz into \eqref{eqt:gCLM} and taking $t\rightarrow T$ yields that the only possible non-zero value for $c_\om$ is $-1$. The value of $c_l$ determines the spatial feature of $\om$: The case $c_l>0$ corresponds to a focusing blowup at $x=0$, while a negative $c_l$ corresponds to an expanding blowup. 

We say that a profile $\Omega$ is interiorly smooth, if $\Omega$ is smooth on $\mathbb{R}$ or if $\Omega$ is compactly supported and smooth in the interior of its closed support. In this paper, we prove the existence of self-similar finite-time blowup with an interiorly smooth profile of the gCLM model for all $a\leq 1$:

\begin{theorem}\label{thm:main_informal}
For each $a\leq 1$, the generalized Constantin--Lax--Majda model \eqref{eqt:gCLM} admits a self-similar solution of the form \eqref{eqt:self-similar_solution} with $c_\om = -1$ and an odd self-similar profile $\Omega\in L^\infty(\mathbb{R})\cap \dot{H}^1(\mathbb{R})$. The profile satisfies that $\Omega'(0)<0$ and that $-\Omega(x)/x$ is decreasing in $x$ and convex in $x^2$ on $[0,+\infty)$. Depending on the sign of $c_l$, one of the following happens:
\begin{enumerate}
\item $c_l<0$: $\Omega$ is compactly supported on $[-L,L]$ for some $L>0$, strictly negative on $(0,L)$, and smooth in the interior of $(-L,L)$; 
\item $c_l=0$: $\Omega$ is strictly negative on $(0,+\infty)$ and smooth on $\mathbb{R}$, and $|\Omega(x)|\lesssim \econst^{-\delta x^2}$ for some $\delta>0$.
\item $c_l> 0$: $\Omega$ is strictly negative on $(0,+\infty)$ and smooth on $\mathbb{R}$, and $|\Omega(x)|\sim |x|^{-1/c_l}$ as $x\rightarrow \infty$.
\end{enumerate}
Moreover, there exists some $1/2<\underline{a}\leq \overline{a}<1$ such that a solution of type $(1)$ must exist for any $a>\overline{a}$, and a solution of type $(3)$ must exist for any $a<\underline{a}$.
\end{theorem}

To be clear, our result does not exclude the possibility of case $(1)$ happening when $a<\underline{a}$ or case $(3)$ happening when $a>\overline{a}$ for potential self-similar solutions constructed in an undiscovered way. A more detailed version of this theorem with finer characterizations (such as integrability and decay rate) of the self-similar profile $\Omega$ and accurate values of $\underline{a},\overline{a}$ will be given in the next section.\\

Before discussing our result, we first give a brief review on previous works in this line of research. In the regime $a<0$, the advection term works in favor of producing a singularity. Finite-time singularity from smooth initial data for the special case $a=-1$ was established by Cordoba, Cordoba and Fontelos \cite{cordoba2005formation}, followed by the improvement of Castro and Cordoba \cite{castro2010infinite} to all $a<0$ based on a Lyapunov functional argument. However, it was unknown whether these finite-time blowups were self-similar. 

For the original case $a=0$, finite-time singularity was established simultaneously with the proposal of the model in \cite{constantin1985simple} by solving the equation explicitly with suitable initial data. It was only recently that Elgindi and Jeong \cite{elgindi2020effects} discovered an exact self-similar finite-time blowup of the form \eqref{eqt:self-similar_solution} with $c_l = -c_\om = 1$. Based on this exact self-similar solution, they also proved in the same paper the existence of self-similar finite-time blowups from smooth initial data for $|a|$ small enough using a series expansion argument. Later, Elgindi, Ghoul, and Masmoudi \cite{elgindi2021stable} improved on this result by establishing the stability of those self-similar blowups for sufficiently small $|a|$. A similar result was also established independently in a work of Chen, Hou, and Huang \cite{chen2021finite}. In a similar spirit, Lushnikov, Silantyev, Siegel \cite{lushnikov2021collapse} and Chen \cite{chen2020singularity} independently found an exact self-similar solution for $a=1/2$ with $c_l = -c_\om/3 = 1/3$. Chen also proved stable self-similar finite-time singularities from smooth initial data for $a$ close to $1/2$ using the method developed in \cite{chen2021finite}. 

Finite-time singularity in the case $a=1$ was conjectured by De Gregorio \cite{de1990one} and was first rigorously established by Chen, Hou, and Huang \cite{chen2021finite} using a computer-assisted proof. They proved the existence of a self-similar solution $\om$ of the form \eqref{eqt:self-similar_solution} with $c_l=c_\om=-1$ and a compactly supported profile $\Omega_*\in H^1(\mathbb{R})$, and they showed that any solution that is initially close to $\om$ in some weighted $H^1$-norm shall develop an asymptotically self-similar singularity with the same scaling (so the initial data can be smooth on $\R$). Recently, Huang, Tong, and Wei \cite{huang2023self} further showed that the De Gregorio model actually admits infinitely many self-similar finite-time blowup solutions of the same scaling $c_l=c_\om = -1$ but with distinct profiles (under re-scaling) that are compactly supported and interiorly smooth, all corresponding to the eigen-functions of a self-adjoint, compact operator.  

Other than the settled cases of $a=0,1/2,1$ and $a$ close to these three values, it was previously an open question whether the gCLM model \eqref{eqt:gCLM} admits self-similar finite-time singularities of the form \eqref{eqt:self-similar_solution} with interiorly smooth $\Omega\in \dot{H}^1(\R)$ for a wide range of $a$. Nevertheless, the numerical studies by Lushnikov, Silantyev, and Siegel \cite{lushnikov2021collapse} suggested the existence of a family of self-similar solutions $\om_s^{(a)}$, with $\Omega^{(a)}$ and $c_l^{(a)}$ continuously depending on $a$. In particular they discovered a critical value $a_c \approx 0.6891$ such that $c_l^{(a)}<0$ for $a>a_c$ while $c_l^{(a)}>0$ for $a<a_c$. That is, $a_c$ is the transition threshold that separates focusing singularities from expanding ones.

Finally, we remark that self-similar solutions with $C^\alpha$ profiles have been constructed by Elgindi and Jeong \cite{elgindi2020effects} for all values of $a$ under the constraint $|a|\alpha<\epsilon$ for some small constant $\epsilon$. This constraint implies that the profile they constructed only has very low regularity for $a$ not close to $0$, making it unuseful in proving finite-time singularity form smooth initial data. Nevertheless, self-similar finite-time blowup from H\"older continuous initial data with finite energy for all $a$ was later proved in \cite{chen2021finite} based on the construction of Elgindi and Jeong.\\

Returning to our result, we see that Theorem \ref{thm:main_informal} answers affirmatively to the question on the existence of self-similar finite-time blowups of the gCLM model with smooth profiles for a large range of $a$. As will be elaborated, we prove this theorem for all $a\leq 1$ in a unified way by considering the fixed-point problem of an $a$-parameterized nonlinear map. More precisely, we first construct a nonlinear map $\mtx{R}_a$ over a suitable function set $\mathbb{D}$ such that if a function $f\in \mathbb{D}$ is a fixed-point of $\mtx{R}_a$, i.e. $f=\mtx{R}_a(f)$, then $\Omega(x) = -xf(x)$ is an exact self-similar profile of the gCLM model \eqref{eqt:gCLM} with $c_l,c_\om$ given explicitly in terms of integrals of $f$. We then prove the existence of fixed points of $\mtx{R}_a$ in $\mathbb{D}$ for all $a\leq 1$ using the Schauder fixed-point theorem. One key observation in our proof is that the map $\mtx{R}_a$ preserves the properties that $f(x)$ is non-increasing in $x$ for $x\geq0$ and $f(\sqrt{s})$ is convex in $s$, which will be frequently used in our analysis. Furthermore, making use of the fixed-point relation $f=\mtx{R}_a(f)$ in $\mathbb{D}$, we are able to prove some general properties of a fixed point $f$ such as its regularity and far-field decay rate, which then transfer to desired properties of the corresponding self-similar profile $\Omega(x) = -xf(x)$. As we will explain later, the previously found self-similar solutions for the three discrete cases $a=0,1/2,1$ all actually belong to our fixed-point family, that is, they can be all recovered as $\Omega(x)=-xf(x)$ from the fixed points of $\mtx{R}_0,\mtx{R}_{1/2}, \mtx{R}_1$, respectively. Therefore, our result unifies the existing results in a single framework. 

Regarding the numerical simulations of Lushnikov, Silantyev, and Siegel \cite{lushnikov2021collapse}, our result partially explains their numerical observations on the qualitatively behavior of the self-similar solutions. Firstly, for any $a$ tested, the self-similar profile $\Omega(x)$ they found numerically is odd in $x$ and non-negative on $[0,+\infty)$. Theorem \eqref{thm:main_informal} confirms the existence of such profiles for all $a\leq 1$. Secondly, they observed a critical value $a_c$ that divides profiles that are compactly supported from those that are not. Though we are not able to prove the existence of such a threshold, we provide a rigorous estimate such that $a_c\in (\underline{a},\overline{a})$ if $a_c$ exists, with $\underline{a}\approx 0.5269$ and $\overline{a} \approx 0.7342$. This at least explains the transition phenomenon of the two types of self-similar singularities: the focusing type and the expanding type. In particular, it is consistent with the previous results that the exact self-similar profiles for $a=0$ and $a=1/2$ are strictly negative on $(0,+\infty)$, while the one for $a=1$ is compactly supported.

We should note that our work only proves the existence of interiorly smooth self-similar profiles that do not change sign on $[0,+\infty)$. It does not exclude the possibility of sign-changing profiles. See \cite{huang2023self} for the finding of infinitely many interiorly smooth, sign-changing profiles in the case $a=1$.\\

It is worth mentioning that the gCLM model on the circle $\mathbb{S}^1$ has also been widely considered in parallel studies \cite{de1996partial,okamoto2008generalization,jia2019gregorio,lei2020constantin,chen2021finite,lushnikov2021collapse,chen2020singularity,lushnikov2021collapse,chen2021slightly,chen2023regularity}. In the mean while, singularity formation and global well-posedness for the gCLM equation with dissipation have also been extensively studied in the literature \cite{schochet1986explicit,kiselev2010regularity,li2008blow,silvestre2016transport,cordoba2005formation,dong2008well,wunsch2011generalized,chen2020singularity,ambrose2023global}.\\

The remaining of this paper is organized as follows. In Section \ref{sec:profile_equation}, we derive an equation for the self-similar profiles and present our main result with more details. Section \ref{sec:existence} is devoted to the proof of existence of self-similar profiles via a fixed-point method, and Section \ref{sec:general_properties} is devoted to the establishment of the desired properties. We review some existing results with more details in Section \ref{sec:review} and show how they relate to our result. Finally, we perform some numerical simulations based on the fixed-point method in Section \ref{sec:numerical} to verify and visualize our theoretical results.

\section{The self-similar profile equation}\label{sec:profile_equation}
Assuming that \eqref{eqt:self-similar_solution} is an exact self-similar solution of \eqref{eqt:gCLM}, we first derive a non-local ordinary differential equation for the self-similar profile $\Omega$ and the scaling factors $c_l,c_\om$. By imposing some natural conditions on $\Omega$, we then deduce a fixed-point formulation for the new variable $f=-\Omega/x$. We also state our main result in this section.

\subsection{Self-similar profiles} Substituting the ansatz \eqref{eqt:self-similar_solution} into the equation \eqref{eqt:gCLM} yields 
\[-c_\om(T-t)^{c_\om-1}\Omega + c_l(T-t)^{c_\om-1}X\Omega_X + a(T-t)^{2c_\om}U\Omega_X = (T-t)^{2c_\om}U_X\Omega,\]
where $X = x/(T-t)^{c_l}$ and $U_X(X) = \mtx{H}(\Omega)(X)$. Provided that $c_\om \neq 0$, balancing the above equation yields $c_\om = -1$ and an equation for the self-similar profile:
\[(c_lX+aU)\Omega_X  = (c_\om+U_X)\Omega,\quad U_X = \mtx{H}(\Omega).\]

For notation simplicity, we will still use $\om,u,x$ for $\Omega, U, X$, respectively, in the rest of this paper. Our goal is to study the existence of solutions $(\om, c_l, c_{\om})$ of the self-similar profile equation 
\begin{equation}\label{eqt:main_equation}
(c_lx+au)\om_x  = (c_\om+u_x)\om,\quad u_x = \mtx{H}(\om),
\end{equation}
for different values of $a$ in the range $(-\infty,1]$. The expressions of $u$ and $u_x$ in terms of $\om$ are, respectively, 
\[u(x) = -(-\Delta)^{-1/2}\om(x) = \frac{1}{\pi}\int_{\R}\om(y)\ln|x-y|\idiff y,\]
\[u_x(x) = \mtx{H}(\om)(x) = \frac{1}{\pi}P.V.\int_{\R}\frac{\om(y)}{x-y}\idiff y.\]
Here, $\mtx{H}$ is the Hilbert transform on the real line with $P.V.$ denoting the Cauchy principal value.\\

Our main result, a more detailed version of Theorem \ref{thm:main_informal}, is stated below.

\begin{theorem}\label{thm:main}
For each $a\leq 1$, the self-similar equation \eqref{eqt:main_equation} admits a solution $(\om, c_l, c_\om)$ with $c_\om = -1$ and an odd function $\om\in L^\infty(\mathbb{R})\cap\dot{H}^1(\mathbb{R})$ satisfying that $\om'(0)<0$ and that $-\om(x)/x$ is decreasing in $x$ and convex in $x^2$ on $[0,+\infty)$. There is some $a$-dependent number $\mu_a\in(0,\min\{1,|a|^{-1}\})$ such that 
\begin{equation}\label{eqt:cl_formula}
c_l = \frac{1-a(2-\mu_a)}{1-a\mu_a}
\begin{cases}
\displaystyle\ =-1, & a=1,  \\ 
\displaystyle\ \in (-1,1), & a\in(0,1),  \\ 
\displaystyle\ =1, & a= 0,  \\ 
\displaystyle\ \in (\max\{1,|a|\},1+2|a|), & a <0,  \\ 
\displaystyle\ \rightarrow +\infty, & a \rightarrow -\infty.
\end{cases}
\end{equation}
Depending on the sign of $c_l$, one of the following happens:
\begin{enumerate}
\item $c_l<0$: The is some $L_a>0$ such that $\om$ is compactly supported on $[-L_a,L_a]$, strictly negative on $(0,L_a)$, and smooth in the interior of $(-L_a,L_a)$, and $L_a$ satisfies
\[\bar{C}|c_l|^{-1/2} \leq  L_a \leq \tilde{C}|c_l|^{-1/2}\]
for some absolute constants $\bar{C},\tilde{C}>0$. There exist some finite numbers $C_a,p_a>0$ such that 
\[\lim_{x\rightarrow L_a-}\frac{\om(x)}{(L_a-x)^{p_a}} = -C_a,\]
and $p_a$ satisfies
\[p_a \geq \max\left\{\frac{3-a}{2a}\ , \frac{1}{a} + \frac{1-a}{a}\cdot \hat{C} L_a^2\right\}\]
for some absolute constant $\hat{C}>0$.
\item $c_l=0$: $\om$ is strictly negative on $(0,+\infty)$ and smooth on $\mathbb{R}$, and $\om \in H^{p}(\mathbb{R})$ for all $p\geq0$. There is some finite number $C_a>0$ such that 
\[\lim_{x\rightarrow +\infty}\frac{\ln|\om(x)|}{x^2} = -C_a.\] 
\item $c_l>0$: $\om$ is strictly negative on $(0,+\infty)$ and smooth on $\mathbb{R}$, and $\om' \in H^{p}(\mathbb{R})$ for all $p\geq0$. There is some finite number $C_a>0$ such that 
\[\lim_{x\rightarrow +\infty} x^{1/c_l}\om(x) = -C_a.\]
\end{enumerate}
Moreover, case $(1)$ must happen for $a>\overline{a}$, while case $(3)$ must happen for $a<\underline{a}$, where 
\[\underline{a} = \frac{400}{848 - 9\pi^2} \approx 0.5269,\quad \overline{a} = \frac{64}{176-9\pi^2}\approx0.7342.\]
\end{theorem}

\vspace{2mm}

Let us briefly comment on this result. Theorem \ref{thm:main} states that the self-similar profile equation \eqref{eqt:main_equation} admits interiorly smooth solutions with $c_\om=-1$ for all $a\leq 1$, implying that the gCLM model \eqref{eqt:gCLM} admits exactly self-similar finite-time blowup of the form \eqref{eqt:self-similar_solution} for all $a\leq 1$. Depending on the sign of $c_l$, these profiles fall in three categories. $(1)$ If $c_l<0$, the profile is compactly supported and smooth in the interior of its closed support, and it vanishes like $(L_a-|x|)^{p_a}$ as $|x|\rightarrow L_a-$ for some $p_a\gtrsim (1 + (1-a)L_a^2)/a$. $(2)$ If $c_l=0$, the profile is strictly negative on $(0,+\infty)$ and smooth on $\mathbb{R}$, and it decays as fast as $\econst^{-C_ax^2}$ for some positive constant $C_a$. $(3)$ If $c_l>0$, the profile is also strictly negative on $(0,+\infty)$ and smooth on $\mathbb{R}$, and it decays algebraically like $|x|^{-1/c_l}$ in the far field. We will prove this theorem through Section \ref{sec:existence} and Section \ref{sec:general_properties}. In view of \eqref{eqt:cl_formula}, the upper bound $\overline{a}$ and the lower bound $\underline{a}$ for the sign-changing point of $c_l$ are obtained by deriving a finer estimate of the number $\mu_a$ (see Theorem \ref{thm:k_bound}). 

Note that only when $c_l\geq 0$ can we immediately claim that the gCLM model develops finite-time singularity from smooth initial data, as the profile itself is smooth on the whole real line.

\begin{corollary}
For any $a\leq \underline{a}\approx 0.5269$, the generalized Constantin--Lax--Majda model \eqref{eqt:gCLM} can develop finite-time singularity from smooth initial data.
\end{corollary}

As for the case $c_l<0$, it requires some extra effort to show that the compactly supported profile $\Omega$ is nonlinearly stable in some suitable energy norm, so that any smooth solution that is initially close to $\Omega$ in this energy norm can develop a finite-time blowup asymptotically in the self-similar form \eqref{eqt:self-similar_solution}. See \cite{chen2021finite} for a practice of this stability argument in the case $a=1$ with $c_l=c_\om=-1$.

\subsection{Reformulation of the problem} Note that if $(\om(x), c_l, c_\om)$ is a solution to \eqref{eqt:main_equation}, then
\begin{equation}\label{eqt:scaling}
(\om_{\alpha,\beta}(x),\ c_{l,\alpha},\ c_{\om,\alpha})= (\alpha\omega(\beta x),\ \alpha c_l,\ \alpha c_\om)
\end{equation}
is also a solution for any $\alpha\in \R$ and $\beta>0$. Owing to this scaling property, we will release ourselves from the restriction that $c_\om=-1$. In fact, it is the ratio $c_l/c_\om$ that matters. Furthermore, we look for solutions that satisfy the following conditions: 

\begin{itemize}
\item Odd symmetry: $\om(x)$ is an odd function of $x$, i.e. $\om(-x) = -\om(x)$.
\item Regularity: $\om\in H^1_{loc}(\R)$.
\item Non-degeneracy: $\om'(0)\neq0$.
\end{itemize}

The odd symmetry is a common feature of all self-similar finite-time singularities of the generalized Constantin--Lax--Majda equation that have been found so far in the literature. In particular, it is proved in \cite{lei2020constantin} that the De Gregorio model ($a=1$) is globally well-posed for initial data that does not change sign on $\mathbb{R}$ (under some mild regularity assumption). Therefore, we only focus on odd solutions.

Assuming the condition $\om\in H^1_{loc}(\R)$ means to avoid solutions with relatively lower regularity. Elgindi and Jeong \cite{elgindi2020effects} have proved the existence of self-similar solutions to \eqref{eqt:gCLM} with $C^\alpha$ profiles for some small $\alpha=\alpha(a)\leq \epsilon/|a|$ for all values of $a$. Our goal is to prove the existence of self-similar profiles with higher regularity. 

In view of the scaling property \eqref{eqt:scaling}, the non-degeneracy condition $\om'(0)\neq 0$ is to make sure $\om$ is non-trivial. This condition leads to a relation between $c_\om,c_l$ and $u'(0)$:
\begin{equation}\label{eqt:cl_relation}
c_l = c_\om + (1-a)u'(0).
\end{equation}
To see this, we simply divide the first equation in \eqref{eqt:main_equation} by $x$ and take the limit $x\rightarrow 0$. Substituting \eqref{eqt:cl_relation} into \eqref{eqt:main_equation} yields that
\begin{equation}\label{eqt:main_equation_modified}
(c_lx + au)\om' = \big(c_l + u' - (1-a)u'(0)\big)\om, \quad u' = \mtx{H}(\om),\quad u(0)=0.
\end{equation}
Define $v = c_lx + au$. Then, for $x$ such that $\om(x)\neq 0$, we have
\[\frac{\om'}{\om} = \frac{v'}{v} + \frac{1-a}{a}\cdot \frac{v'-v'(0)}{v}.\]
Suppose that $\om$ is non-positive on $(0,+\infty)$ and that $\om(x)<0$ on $(0,L)$ for some $L>0$ ($L$ can be $+\infty$). Solving the ODE above on $(0,L)$ yields 
\begin{equation}\label{eqt:v_to_omega}
\om(x) = \frac{\om'(0)}{v'(0)}\cdot v(x)\exp\left(\frac{1-a}{a}\int_0^x\frac{v'(y)-v'(0)}{v(y)}\idiff y\right),\quad x\in [0,L].
\end{equation}
In view of \eqref{eqt:scaling}, we may further assume that $\om'(0)=-1$ without loss of generality. Since $\om$ and $v$ are both odd functions, we may consider the change of variables
\[f := \frac{\om}{x\om'(0)} = -\frac{\om}{x},\quad g := \left(\frac{v}{xv'(0)}\right)_+ = \left(\frac{c_lx - a(-\Delta)^{-1/2}(\om)}{x\big(c_l + a\mtx{H}(\om)(0)\big)}\right)_+.\]
Here and below, $(t)_+:=\max\{t\,,0\}$ for any number $t$. Note that $f$ and $g$ are both even functions of $x$, and $f(0)=g(0)=1$. We find that
\[\int_0^x\frac{v'(y)-v'(0)}{v(y)}\idiff y = \int_0^x\frac{yg'(y)+g(y)-g(0)}{yg(y)}\idiff y = \ln g(x) + \int_0^x\frac{g(y)-1}{yg(y)}\idiff y.\]
Therefore, $f$, $g$ and $c_l$ together satisfy
\begin{equation}\label{eqt:g_to_f}
\begin{split}
f(x) &= g(x)^{1/a}\exp\left(\frac{1-a}{a}\int_0^x\frac{g(y)-1}{yg(y)}\idiff y\right),\\
g(x) &= \left(\frac{c_lx + a(-\Delta)^{-1/2}(xf)(x)}{x\big(c_l - a\mtx{H}(xf)(0)\big)}\right)_+.
\end{split}
\end{equation}
This observation is the starting point of our fixed-point method for proving the existence of self-similar solutions.

\section{Existence of solutions by a fixed-point method}\label{sec:existence}
Our goal of this section is to show that the nonlinear system $\eqref{eqt:g_to_f}$ admits non-trivial solutions for each $a\leq 1$. We do so by converting the problem into a fixed-point problem of some nonlinear map and then using the Schauder fixed-point theorem to show existence of fixed points. To this end, we need to select an appropriate Banach function space in which we can establish continuity and compactness of our nonlinear map.

\subsection{A fixed-point problem}
Consider a Banach space of continuous even functions,
\[\mathbb{V} := \left\{f\in C(\R):\ f(x) = f(-x),\ \|\rho f\|_{L^\infty} < +\infty\right\},\]
endowed with a weighted $L^\infty$-norm $\|\rho f\|_{L^\infty}$, referred to as the $L_\rho^\infty$-norm, where $\rho(x)=(1+|x|)^{-1/2}$. Moreover, we consider a closed (in the $L_\rho^\infty$-norm) and convex subset of $\mathbb{V}$,
\begin{equation*}
\begin{split}
\mathbb{D} := \Big\{\ f\in \mathbb{V}:\quad &f(0) = 1,\ (1-x^2)_+\leq f(x)\leq 1\ \text{for all $x$}, \\
& \text{$f(x)$ is non-increasing on $[0,+\infty)$},\ \text{$f(\sqrt{s})$ is convex in $s$},\\
& f'_{-}(1/2)\leq -\eta,\ \eta = 1/(3\cdot 2^{20}\sqrt{2})\ \Big\},
\end{split}
\end{equation*}
which will act as the invariant set for our fixed-point method. Here and below, $(t)_+:=\max\{t,0\}$ for any number $t$; $f'_{-}$ and $f'_{+}$ denote the left and the right derivatives of a function $f$. The condition $f'_-(1/2)\leq -\eta$ with $\eta>0$ is to avoid the constant function $f\equiv 1$ being in $\mathbb{D}$. In fact, $f\in \mathbb{D}$ implies
\begin{equation}\label{eqt:f_range}
(1-x^2)_+ \leq f(x)\leq \max\{1-\eta x^2,\ 1-\eta/4\},\quad x\in \R,
\end{equation}
where the upper bound follows from the assumptions that $f(\sqrt{s})$ is convex in $s$, $f'_{-}(1/2)\leq -\eta$, and $f(x)$ is non-increasing on $[0,+\infty)$. On the other hand, for $\mathbb{D}$ to be an invariant set in our fixed-point method, $\eta$ needs to be sufficiently small. More precisely, we can choose any positive $\eta\leq 1/(3\cdot 2^{20}\sqrt{2})$, where the upper bound (not necessarily optimal) is determined through a bootstrap argument that will be clear later. Here we simply choose $\eta=  1/(3\cdot 2^{20}\sqrt{2})$. We will explain the design of $\mathbb{D}$ with more details at the end of this subsection.

We remark that, though a function $f\in\mathbb{D}$ is not required to be differentiable, the one-sided derivatives $f'_{-}(x)$ and $f'_{+}(x)$ are both well defined at every point $x$ by the convexity of $f(\sqrt{s})$ in $s$. In what follows, we will abuse notation and simply use $f'(x)$ for $f'_{-}(x)$ and $f'_{+}(x)$ in both weak sense and strong sense. For example, when we write $f'(x)\leq C$, we mean $f'_{-}(x)\leq C$ and $f'_{+}(x)\leq C$ at the same time. In this context, the non-increasing property of $f$ on $[0,+\infty)$ can be represented as $f'(x)\leq 0$ for $x\geq 0$.

Now, we construct an $a$-dependent nonlinear map whose potential fixed points correspond to solutions of \eqref{eqt:g_to_f}. We first define a linear map
\[\mtx{T}(f)(x) := \frac{1}{\pi}\int_0^{+\infty}f(y)\left(\frac{y}{x}\ln\left|\frac{x+y}{x-y}\right|-2\right)\idiff y,\quad f\in \mathbb{D}.\]
This definition only uses the integral on $[0,+\infty)$ since $f\in \mathbb{D}\subset \mathbb{V}$ is an even function. We will always employ this symmetry property in the sequel. One should note that $\mtx{T}(f)$ is well-defined for any $f\in \mathbb{D}$, since for each fixed $x$, the kernel of $\mtx{T}$ decays like $y^{-2}$ as $y\rightarrow +\infty$. It is also not hard to show that, for any $f\in \mathbb{D}$, 
\[\mtx{T}(f)(0) = 0, \quad \mtx{T}(f)(+\infty) = \lim_{x\rightarrow+\infty}\mtx{T}(f)(x) =: -b(f) \leq 0,\]
where 
\[b(f) := \frac{2}{\pi}\int_0^{+\infty}\big(f(y)-f(+\infty)\big)\idiff y,\]
with $f(+\infty) = \lim_{x\rightarrow +\infty}f(x)$. The limit $\mtx{T}(f)(+\infty) = -b(f)$ is valid even when $b(f)=+\infty$, which is possible for $f\in \mathbb{D}$. To see how $\mtx{T}$ is related to \eqref{eqt:g_to_f}, we note that, when $f\in \mathbb{D}$ decays sufficiently fast (in particular, when $f(+\infty) = 0$ and $b(f)<+\infty$), 
\begin{equation}\label{eqt:T_to_laplacian}
\mtx{T}(f)(x) = \frac{1}{x}(-\Delta)^{-1/2}(xf)(x) + \mtx{H}(xf)(0) = \frac{1}{x}(-\Delta)^{-1/2}(xf)(x) - b(f),
\end{equation}
which also relies on the even symmetry of $f$.

Next, we define 
\begin{equation}\label{eqt:T_a}
\mtx{T}_a(f) := \left(1 + \frac{2a\cdot \mtx{T}(f)}{(1-a/3)c(f)}\right)_+,\quad f\in \mathbb{D},
\end{equation}
where 
\[c(f) := -\frac{2}{\pi}\int_0^{+\infty}\frac{f'(y)}{y}\idiff y = \frac{2}{\pi}\int_0^{+\infty}\frac{f(0)-f(y)}{y^2}\idiff y.\]
Since $\mtx{T}(f)(0)=0$, $\mtx{T}_a(f)(0) = 1$ in all cases. The ratio $b(f)/c(f)$ will be an important value that occurs frequently in what follows, as it determines the asymptotic behavior of $\mtx{T}_a(f)$:
\begin{equation}\label{eqt:Ta_limit}
\mtx{T}_a(f)(+\infty) = \lim_{x\rightarrow+\infty}\mtx{T}_a(f)(x) = \left(1 - \frac{2a}{(1-a/3)}\cdot \frac{b(f)}{c(f)}\right)_+
\end{equation}
Note that $c(f)$ must be strictly positive and finite for any $f\in \mathbb{D}$. Actually, in view of \eqref{eqt:f_range}, we have
\begin{equation}\label{eqt:cf_bound}
c(f) = \frac{2}{\pi}\int_0^{+\infty}\frac{1-f(y)}{y^2}\idiff y \begin{cases}
\displaystyle\ \leq \frac{2}{\pi}\int_0^1 1\idiff y + \frac{2}{\pi}\int_1^{+\infty}\frac{1}{y^2}\idiff y = \frac{4}{\pi},\vspace{4mm} \\ 
\displaystyle\ \geq \frac{2}{\pi}\int_0^{1/2}\eta\idiff y + \frac{2}{\pi}\int_{1/2}^{+\infty}\frac{\eta}{4y^2}\idiff y = \frac{2\eta}{\pi}.
\end{cases}
\end{equation}
The lower bound of $c(f)$ explains why we need to impose the condition $f'_{-}(1/2)\leq -\eta$ on $\mathbb{D}$: to make sure $c(f)$ is strictly positive so that $\mtx{T}_a$ is well defined for all $f\in \mathbb{D}$.

Finally, for $a\leq 1$, we define an $a$-dependent nonlinear map:
\begin{equation}
\mtx{R}_a(f)(x) := \big(\mtx{T}_a(f)(x)\big)^{1/a}\cdot\exp\left(\frac{1-a}{a}\int_0^x\frac{\mtx{T}_a(f)(y)-\mtx{T}_a(f)(0)}{y\mtx{T}_a(f)(y)}\idiff y\right),\quad f\in\mathbb{D}.
\end{equation}
We note that the expression of $\mtx{R}_a$ has a formal singularity at $a=0$, which can actually be well defined by considering the limit $a\rightarrow 0$. Thus, we specially define
\begin{equation}\label{eqt:R_0}
\mtx{R}_0(f)(x) := \exp\left(\frac{2}{c(f)}\left(\mtx{T}(f)(x)+\int_0^x\frac{\mtx{T}(f)(y)}{y}\idiff y\right)\right),\quad f\in\mathbb{D}.
\end{equation}
We now aim to study the fixed-point problem
\[\mtx{R}_a(f)=f,\quad f\in \mathbb{D}.\]
As the core idea of this paper, the following proposition explains how a fixed point of $\mtx{R}_a$ relates to a solution of \eqref{eqt:main_equation}. 

\begin{proposition}\label{prop:fixed_point_solution}
For any $a\leq 1$, if $f\in \mathbb{D}$ is a fixed point of $\mtx{R}_a$, i.e. $\mtx{R}_a(f) = f$, then $f(+\infty)=0$, $b(f)<+\infty$, and $(f,c_l)$ is a solution to equations \eqref{eqt:g_to_f} with 
\begin{equation}\label{eqt:f_to_cl}
c_l = \frac{1-a/3}{2}c(f) -ab(f)  = \frac{1-a/3}{\pi}\int_0^{+\infty}\frac{f(0)-f(y)}{y^2}\idiff y - \frac{2a}{\pi}\int_0^{+\infty}f(y)\idiff y .
\end{equation}
As a consequence, $(\om,c_l,c_\om)$ is a solution to \eqref{eqt:main_equation} with $\om = -xf$ and 
\begin{equation}\label{eqt:f_to_cw}
c_\om = c_l - (1-a)b(f) = \frac{1-a/3}{\pi}\int_0^{+\infty}\frac{f(0)-f(y)}{y^2}\idiff y -\frac{2}{\pi}\int_0^{+\infty}f(y)\idiff y.
\end{equation}
Conversely, if $(f,c_l)$ is a solution to equations \eqref{eqt:g_to_f} such that $f(x)$ is an even function of $x$, $f(x)\geq 0$ and $f'(x)\leq 0$ for all $x\geq 0$, $f(0)=1$, $f(+\infty)=0$, and $\lim_{x\rightarrow0}f'(x)/2x=-1$ (by re-normalization), then $f$ is a fixed point of $\mtx{R}_a$, and $c_l$ is related to $f$ as in \eqref{eqt:f_to_cl}. 
\end{proposition}

\begin{proof}
The first statement follows directly from the construction of $\mtx{R}_a$. The claims that $f(+\infty) = 0$ and $b(f)<+\infty$ provided $f=\mtx{R}_a(f)$ will be proved in Lemma \ref{lem:f_infinity} and Lemma \ref{lem:finite_bf}, respectively, in the next section. The formulas of $c_l$ is obtained by comparing the definition of $\mtx{T}_a(f)$ and the expression of $g(x)$ in \eqref{eqt:g_to_f}, and the formula of $c_\om$ follows from \eqref{eqt:cl_relation}.

Conversely, if $(f,c_l)$ is a solution to equations \eqref{eqt:g_to_f} such that $f$ is an even function of $x$, then $f$ automatically satisfies $f(0)=1$ and $f(x)\geq 0$ for $x\geq 0$. Moreover, we can re-scale $f(x)$ as $f(\beta x)\rightarrow f(x)$ (in view of \eqref{eqt:scaling}) so that $\lim_{x\rightarrow0}f'(x)/2x=-1$. Now, if we rewrite $g(x)$ in \eqref{eqt:g_to_f} as
\[g(x) = \big(1 + \lambda \mtx{T}(f)(x)\big)_+,\quad \lambda = \frac{a}{c_l-a\mtx{H}(xf)(0)},\]
then we can show that 
\begin{equation}\label{eqt:f_to_cl_step}
\lim_{x\rightarrow 0}\frac{f'(x)}{2x}=-1\quad \Longleftrightarrow \quad \lambda = \frac{2a}{(1-a/3)c(f)},
\end{equation}
that is 
\[c_l = \frac{1-a/3}{2}c(f) + a\mtx{H}(xf)(0) = \frac{1-a/3}{2}c(f) - ab(f),\]
which is exactly \eqref{eqt:f_to_cl}. We have used that $b(f) = -\mtx{H}(xf)(0)$ when $f(+\infty)=0$. We delay the details of showing \eqref{eqt:f_to_cl_step} to the proof of Lemma \ref{lem:Ra_close} below.
\end{proof}

The remaining of this section is devoted to proving the existence of fixed points of $\mtx{R}_a$ in $\mathbb{D}$ for $a\leq 1$. Before getting into the details, let us briefly explain the design of the set $\mathbb{D}$ and the ideas behind the proof. In order to apply the Schauder fixed-point theorem, we want that (1) $\mathbb{D}$ is nonempty, convex, and closed in the underlying Banach topology, (2) $\mathbb{D}$ is compact in the same topology, and (3) $\mtx{R}_a$ maps $\mathbb{D}$ continuously into itself. Note that (1) is automatically satisfied by the design of $\mathbb{D}$. To establish (2) and (3), it is crucial to observe that the intermediate linear map $\mtx{T}: f\mapsto \mtx{T}(f)$ preserves monotonicity in $x$ and convexity in $x^2$ on $[0,+\infty)$, which is found out by applying integration by parts to the formula of $\mtx{T}$ (as will be presented in the proof of Lemma \ref{lem:T_property}). This monotonicity and convexity preserving property of $\mtx{T}$ then passes on to the non-linear map $\mtx{R}_a$ for all $a\leq 1$ through some straightforward calculations of derivatives, which provides powerful controls on $\mtx{R}_a(f)$. In particular, it implies the uniform estimates that $(1-x^2)_+\leq \mtx{R}_a(f)\leq 1$ and that $\mtx{R}_a(f)'_{-}(1/2)\leq -\eta$ for some constant $\eta>0$ chosen through a bootstrap argument (see \eqref{eqt:bootstrap}). Recall that we need $f'_{-}(1/2)=-\eta$ to ensure that $c(f)$ has a uniform lower bound above $0$ (such non-degeneracy condition of $f'_{-}$ can be imposed alternatively at any point in $(0,1)$; $1/2$ is just one suitable choice). Also, it is easy to see that $\mtx{R}_a(f)$ is even and $\mtx{R}_a(f)(0)=1$. Thus, the function set $\mathbb{D}$ is closed under the map $\mtx{R}_a$. Moreover, the monotonicity and convexity properties lead to the continuity of $\mtx{R}_a:\mathbb{D}\rightarrow \mathbb{D}$ and the compactness of $\mathbb{D}$ in the $L_\rho^\infty$-topology. Finally, with all these ingredients in hand, it comes to applying the Schauder fixed-point theorem on $\mtx{R}_a:\mathbb{D}\rightarrow \mathbb{D}$ to conclude the proof.

\subsection{Properties of $c(f)$}
We start with a finer estimate of $c(f)$ that will be useful later.

\begin{lemma}\label{lem:cf_bound}
For any $f\in \mathbb{D}$ and any $x>0$, 
\[\frac{2\eta}{\pi} \leq c(f) \leq \min\left\{\frac{4(x+1)}{\pi x}\big(1-f(x)\big)^{1/2}\ ,\ \frac{4}{\pi}\right\}.\]
\end{lemma}

\begin{proof}
The lower bound and the constant upper bound have already been proved in \eqref{eqt:cf_bound}.

Next, fix an $x>0$. For $0\leq y\leq x$, $f(y)\geq \max\{(1-y^2)_+, f(x)\}$, so $1-f(y)\leq \min\{y^2, 1-f(x)\}$. For $y>x$, the convexity of $f(\sqrt{s})$ implies that $(1-f(y))/y^2\leq (1-f(x))/x^2$, and so $1-f(y)\leq \min\{y^2(1-f(x))/x^2, 1\}$. Combining these estimates yields 
\[1-f(y)\leq \min\{y^2, 1-f(x)\} + \min\{y^2(1-f(x))/x^2, 1\}, \quad y\geq0.\]
We thus obtain that 
\begin{align*}
c(f) &\leq \frac{2}{\pi}\int_0^{+\infty}\frac{\min\{y^2, 1-f(x)\}}{y^2}\idiff y + \frac{2}{\pi}\int_0^{+\infty}\frac{\min\{y^2(1-f(x))/x^2, 1\}}{y^2}\idiff y\\
&\leq \frac{4(x+1)}{\pi x}\big(1-f(x)\big)^{1/2}, 
\end{align*}
which is the desired bound.
\end{proof}

We will need the continuity of $c(f)$ for proving the continuity of $\mtx{R}_a(f)$ in the $L_\rho^\infty$ topology.

\begin{lemma}\label{lem:cf_continuous}
$c(f):\mathbb{D}\to \mathbb{R}$ is H\"older continuous in the $L_\rho^\infty$-norm. In particular, 
\[|c(f_1) - c(f_2)|\lesssim \|\rho(f_1-f_2)\|_{L^\infty}^{1/2},\]
for any $f_1,f_2\in \mathbb{D}$.
\end{lemma}

\begin{proof}
Recall that $\rho = (1+|x|)^{-1/2}$. Denote $\delta = \|\rho(f_1-f_2)\|_{L^\infty}\leq 1$. Since $f_i\geq (1-x^2)_+, i=1,2$, we have 
\[|f_1(x) - f_2(x)| \leq \min\{\ x^2,\ \delta(1+|x|)^{1/2}\ \}.\]
Hence,  
\[|c(f_1)-c(f_2)| \lesssim \int_0^{+\infty}\frac{|f_1(y)-f_2(y)|}{y^2}\idiff y \leq  \int_0^{\sqrt{\delta}}1\idiff y + \int_{\sqrt{\delta}}^{+\infty}\frac{\delta(1+y)^{1/2}}{y^2}\idiff y \lesssim \sqrt{\delta}.\]
This proves the lemma.
\end{proof}

\subsection{Properties of $\mathbf{T}$ and $\mathbf{T}_a$} We now turn to study the intermediate maps $\mtx{T}$ and $\mtx{T}_a$. As an important observation in our fixed-point method, they preserve the monotonicity and convexity of functions in $\mathbb{D}$.

\begin{lemma}\label{lem:T_property}
For any $f\in \mathbb{D}$, $\mtx{T}(f)'(x)\leq 0$ on $[0,+\infty)$, and $\mtx{T}(f)(\sqrt{s})$ is convex in $s$. 
\end{lemma}

\begin{proof}
We first show that $\mtx{T}(f)'(x)\leq 0$ on $(0,+\infty)$. We can use integration by parts to compute that, for $x>0$,  
\begin{equation}\label{eqt:T_integrate_by_part}
\begin{split}
\mtx{T}(f)(x) &= \frac{1}{\pi}\int_0^{+\infty}f(y)\left(\frac{y}{x}\ln\left|\frac{x+y}{x-y}\right|-2\right)\idiff y\\
&= \frac{1}{\pi}\int_0^{+\infty}f(y)\cdot \partial_y\left(\frac{y^2-x^2}{2x}\ln\left|\frac{x+y}{x-y}\right|-y\right)\idiff y\\
&= \frac{1}{\pi}\int_0^{+\infty}f'(y)\cdot\left(\frac{x^2-y^2}{2x}\ln\left|\frac{x+y}{x-y}\right|+y\right)\idiff y\\
&= \frac{1}{\pi}\int_0^{+\infty}f'(y)\cdot yF(x/y)\idiff y,
\end{split}
\end{equation}
where the function $F$ is defined in \eqref{eqt:F_definition} in Appendix \ref{sec:special_functions}, and the integration by parts can be justified by the properties of $F$ proved in Lemma \ref{lem:F_property}. Therefore, we have
\begin{equation}\label{eqt:T_derivative}
\mtx{T}(f)'(x) = \frac{1}{\pi}\int_0^{+\infty}f'(y)\cdot y\partial_x F(x/y)\idiff y = \frac{1}{\pi}\int_0^{+\infty}f'(y)\cdot F'(x/y)\idiff y \leq 0,
\end{equation}
where the inequality follows from property (3) in Lemma \ref{lem:F_property}.

Next, we show that $\mtx{T}(f)(\sqrt{s})$ is convex in $s$. By approximation theory, we may assume that $f(\sqrt{s})$ is twice differentiable in $s$, so that the convexity of $f(\sqrt{s})$ in $s$ is equivalent to $(f'(x)/x)'\geq 0$ for $x>0$. Continuing the calculations above, we have
\begin{align*}
\frac{\mtx{T}(f)'(x)}{x} &= \frac{1}{\pi}\int_0^{+\infty}\frac{f'(y)}{y}\cdot \frac{y}{x}\left(\frac{y^2+x^2}{2x^2}\ln\left|\frac{x+y}{x-y}\right| - \frac{y}{x}\right)\idiff y\\
&= \frac{1}{\pi}\int_0^{+\infty}\frac{f'(y)}{y}\cdot \partial_y\left(\frac{y^4+2x^2y^2-3x^4}{8x^3}\ln\left|\frac{x+y}{x-y}\right| - \frac{y^3}{4x^2} - \frac{7y}{12}\right)\idiff y + \frac{4}{3\pi}\int_0^{+\infty}\frac{f'(y)}{y}\idiff y\\
&= \frac{1}{\pi}\int_0^{+\infty}\left(\frac{f'(y)}{y}\right)'\cdot\left(\frac{3x^4-2x^2y^2-y^4}{8x^3}\ln\left|\frac{x+y}{x-y}\right| + \frac{y^3}{4x^2} + \frac{7y}{12}\right)\idiff y + \frac{4}{3\pi}\int_0^{+\infty}\frac{f'(y)}{y}\idiff y\\
&= \frac{1}{\pi}\int_0^{+\infty}\left(\frac{f'(y)}{y}\right)'\cdot y G(x/y)\idiff y + \frac{4}{3\pi}\int_0^{+\infty}\frac{f'(y)}{y}\idiff y.
\end{align*}
where the function $G$ is defined in \eqref{eqt:G_definition} in Appendix \ref{sec:special_functions}, and the integration by parts can be justified by the properties of $G$ proved in Lemma \ref{lem:G_property}. Therefore, 
\[
\left(\frac{\mtx{T}(f)'(x)}{x}\right)' = \frac{1}{\pi}\int_0^{+\infty}\left(\frac{f'(y)}{y}\right)'\cdot y \partial_x G(x/y)\idiff y = \frac{1}{\pi}\int_0^{+\infty}\left(\frac{f'(y)}{y}\right)'\cdot G'(x/y)\idiff y\geq 0,
\]
where the inequality follows from property (3) in Lemma \ref{lem:G_property}. This implies the convexity of $\mtx{T}(f)(\sqrt{s})$ in $s$.
\end{proof}

By the definition of $\mtx{T}_a$, we immediately have the following.

\begin{corollary}\label{cor:Ta_property}
For any $a\leq 1$ and any $f\in\mathbb{D}$, $\sgn(a)\cdot \mtx{T}_a(f)'(x)\leq 0$ on $[0,+\infty)$, and $\sgn(a)\cdot \mtx{T}_a(f)(\sqrt{s})$ is convex in $s$. Moreover, $\mtx{T}_a(f)$ is compactly supported on $[-X,X]$ for some $X>0$ if and only if $a>0$ and $b(f)/c(f) > (1-a/3)/2a$, and $\mtx{T}_a(f)$ is strictly positive in the interior of its support.
\end{corollary}

\begin{proof}
The claims that $\sgn(a)\cdot \mtx{T}_a(f)'(x)\leq 0$ on $[0,+\infty)$ and $\sgn(a)\cdot \mtx{T}_a(f)(\sqrt{s})$ is convex in $s$ follow directly from Lemma \ref{lem:T_property} and the definition of $\mtx{T}_a(f)$. 

If $a\leq0$, $\sgn(a)\cdot \mtx{T}_a(f)'(x)\leq 0$ implies that $\mtx{T}_a(f)$ is non-decreasing on $[0,+\infty)$, and thus $\mtx{T}_a(f)(x)\geq\mtx{T}_a(f)(0)=1$ for all $x$. 

In the case $a>0$, $\mtx{T}_a(f)$ is non-increasing on $[0,+\infty)$. In particular, from the formula \eqref{eqt:T_derivative} we know that $\mtx{T}(f)'(x)<0$ for $x\in (0,+\infty)$ (unless $f$ is constant, which cannot happen for $f\in \mathbb{D}$). So $\mtx{T}_a(f)$ is also strictly decreasing on $(0,+\infty)$. It then follows from \eqref{eqt:Ta_limit} that 
\[\mtx{T}_a(f)(x) > \mtx{T}_a(f)(+\infty) = 1 - \frac{2a}{(1-a/3)}\cdot \frac{b(f)}{c(f)} \geq 0,\]
if $b(f)/c(f) \leq (1-a/3)/2a$. Otherwise, there must be some $X>0$ such that $\mtx{T}_a(f)(x)> 0$ for $0\leq x<X$ and $\mtx{T}_a(f)(x)=0$ for $x\geq X$. Therefore, $\mtx{T}_a(f)$ is compactly supported if and only if $a>0$ and $b(f)/c(f) > (1-a/3)/2a$.
\end{proof}

\subsection{Properties of $\mathbf{R}_a$} We will prove continuity and some decay properties of $\mtx{R}_a$ in this subsection. The continuity property is a crucial ingredient for establishing existence of fixed points of $\mtx{R}_a$. The decay properties will be useful for characterizing far-field behavior of the fixed points in the next section. 

We first show that the set $\mathbb{D}$ is closed under $\mtx{R}_a$.

\begin{lemma}\label{lem:Ra_close}
For any $a\leq 1$, $\mtx{R}_a$ maps $\mathbb{D}$ into itself.
\end{lemma}

\begin{proof}
Noticing the particularity of $\mtx{R}_0$, we first assume that $a\neq 0$. Given $f\in\mathbb{D}$, let $g = \mtx{T}_a(f)$ and $h= \mtx{R}_a(f)$. We prove this lemma through the following steps. \\

\noindent \textbf{Step $1$: Show that $h(0)=1$.} Denote
\[\phi(x) := \frac{1-a}{a}\int_0^x\frac{g(y)-g(0)}{yg(y)}\idiff y,\]
so that $h(x) = g(x)^{1/a}\econst^{\phi(x)} $. Since $g(0)=1$, we have $h(0) = g(0)^{1/a}\econst^{\phi(0)} = 1$.\\

\noindent \textbf{Step $2$: Show that $h'(x)\leq 0$ for $x>0$.} By Corollary \ref{cor:Ta_property}, $\sgn(a)\cdot g'(x)\leq 0$ on $[0,+\infty)$, which also means $\sgn(a)(g(x)-g(0))\leq 0$ for $x\geq 0$. This implies that $\phi(x)\leq 0$ for all $x$, and for $x>0$,
\[g(x)\cdot \phi'(x) = \frac{1-a}{a}\cdot\frac{g(x)-g(0)}{x} \leq 0.\]
Hence, we have for $x>0$,
\[h'(x) = \left(\frac{1}{a}\cdot g'(x) + g(x)\phi'(x)\right)\cdot \frac{h(x)}{g(x)} \leq 0.\]
Note that $h(x)/g(x)=g(x)^{1/a-1}\econst^{\phi(x)}$ is always nonnegative and bounded by $1$ in spite of the sign of $a$. Also note that, if $g$ is compactly supported on $[-L,L]$ for some $L>0$ (see Corollary \ref{cor:Ta_property}), then $h'(x) = 0$ for $x>L$, $h'_{-}(L)\leq 0$, and $h'_{+}(L)=0$.\\ 

\noindent \textbf{Step $3$: Show that $h(\sqrt{s})$ is convex in $s$.} We only need to prove $(h'(x)/x)'\geq 0$ for all $x>0$ such that $h(x)>0$. Continuing the calculations in step $2$, we reach
\begin{align*}
\left(\frac{h'(x)}{x}\right)' &= \left(\frac{1}{a}\cdot\left(\frac{g'(x)}{x}\right)' + \frac{1-a}{a}\cdot \left(\frac{g(x)-g(0)}{x^2}\right)'\right)\cdot \frac{h(x)}{g(x)} + \left(\frac{1-a}{a}\cdot\frac{g'(x)}{g(x)} + \phi'(x)\right)\cdot \frac{h'(x)}{x}\\
&\geq \frac{1-a}{a}\cdot \left(\frac{g(x)-g(0)}{x^2}\right)'\cdot \frac{h(x)}{g(x)}.
\end{align*}
We have used that $\sgn(a) (g'(x)/x)'\geq 0$ (from Corollary \ref{cor:Ta_property}) and that $h'(x),\phi'(x),\sgn(a) g'(x)\leq 0$ for $x>0$. Note that $\sgn(a)(g'(x)/x)'\geq 0$ also implies
\begin{equation}\label{eqt:g_ineq}
\sgn(a)\cdot \frac{g'(x)}{2}\geq \sgn(a)\cdot \frac{g(x)-g(0)}{x},\quad x>0,
\end{equation}
and thus,
\[\sgn(a)\left(\frac{g(x)-g(0)}{x^2}\right)' = \sgn(a)\left(\frac{g'(x)}{x^2} - \frac{2(g(x)-g(0))}{x^3}\right)\geq 0.\]
Therefore, we have $(h'(x)/x)'\geq 0$ for $x>0$.\\

\noindent \textbf{Step $4$: Show that $(1-x^2)_+\leq h(x)\leq 1$.} The fact that $h(x)\leq 1$ follows directly from step $1$ and $2$. To prove that $h(x)\geq (1-x^2)_+$, namely $h(\sqrt{s})\geq (1-s)_+$ for $s\geq0$, we only need to show that
\[\lim_{x\rightarrow0}\frac{h'(x)}{2x}= \frac{\diff \ }{\diff s} h(\sqrt{s})\Big|_{s=0} =  -1\]
and then use the fact that $h(\sqrt{s})$ is convex in $s$ (step $3$). Note that in the support of $g$, 
\[g'(x) = \frac{2a}{(1-a/3)c(f)}\cdot \mtx{T}(f)'(x).\] 
Then, from the proof of Lemma \ref{lem:T_property}, we find
\begin{equation}\label{eqt:Ra_close_midstep1}
\begin{split}
\lim_{x\rightarrow 0}\frac{g'(x)}{2x} &= \frac{2a}{(1-a/3)c(f)}\left(\lim_{x\rightarrow0} \frac{1}{2\pi}\int_0^{+\infty}\left(\frac{f'(y)}{y}\right)'\cdot y G(x/y)\idiff y + \frac{2}{3\pi}\int_0^{+\infty}\frac{f'(y)}{y}\idiff y\right) \\
&= \frac{2a}{(1-a/3)c(f)}\cdot \frac{2}{3\pi}\int_0^{+\infty}\frac{f'(y)}{y}\idiff y = -\frac{2a}{3-a}.
\end{split}
\end{equation}
Note that 
\[\frac{h'(x)}{x} = \left(\frac{1}{a}\cdot\frac{g'(x)}{x} + \frac{1-a}{a}\cdot \frac{g(x)-g(0)}{x^2}\right)\cdot g(x)^{(1-a)/a}\econst^{\phi(x)},\quad x>0.\]
Hence, we have
\[\lim_{x\rightarrow0}\frac{h'(x)}{2x} = g(0)^{(1-a)/a}\econst^{\phi(0)}\cdot\frac{3-a}{2a}\lim_{x\rightarrow0}\frac{g'(x)}{2x} = -1,\]
as desired.\\

\noindent \textbf{Step $5$: Show that $h'_-(1/2)\leq -\eta$.} Using \eqref{eqt:Ra_close_midstep1} and the fact that $\sgn(a)g(\sqrt{s})$ is convex in $s$, we have 
\[\sgn(a)g(x)\geq \sgn(a)g(0) - \sgn(a)\frac{2a}{3-a}x^2 = \sgn(a) - \frac{2|a|}{3-a}x^2.\]
Note that when $a\geq 0$, $g(x)\leq g(0)=1$. Thus, for all $a\leq 1$, we always have
\begin{equation}\label{eqt:Ra_close_midstep2}
\left(1-\frac{2|a|}{3-a}x^2\right)_+\leq g(x)\leq 1 + \frac{2|a|}{3-a}x^2\leq 1+2x^2.
\end{equation}
We then use \eqref{eqt:g_ineq} to obtain
\[\frac{h'(x)}{h(x)} = \frac{1}{a}\cdot\frac{g'(x)}{g(x)} + \frac{1-a}{a}\cdot \frac{g(x)-g(0)}{xg(x)}\leq \frac{3-a}{2a}\cdot\frac{g'(x)}{g(x)}\leq  \frac{3\mtx{T}(f)'(x)}{c(f)(1+2x^2)}.\]
Note that $h(x)\geq 1-x^2>0$ for $x\in[0,1)$, thus 
\[h'(x) \leq \frac{3\mtx{T}(f)'(x)}{c(f)(1+2x^2)}h(x) \leq \frac{3(1-x^2)\mtx{T}(f)'(x)}{c(f)(1+2x^2)},\quad x\in[0,1).\]

Next, we upper bound $\mtx{T}(f)'(x)$ in two ways. On the one hand, we can use the calculations in the proof of Lemma \ref{lem:T_property} to get that, for $x>0$, 
\begin{align*}
\mtx{T}(f)'(x) &= \frac{1}{\pi}\int_0^{+\infty}f'(y)\cdot F'(x/y)\idiff y \leq \frac{1}{\pi}\int_0^{x}f'(y)\cdot F'(x/y)\idiff y\\
&\leq \frac{1}{\pi}\int_0^{x}\frac{f'(x)}{x}\cdot yF'(x/y)\idiff y= xf'(x)\cdot\frac{1}{\pi}\int_0^1tF'(1/t)\idiff t = \frac{xf'(x)}{2\pi}.
\end{align*}
We have used the fact that $\int_0^1tF'(1/t)\idiff t = (4t/3-tG(1/t))\big|_0^1 = 1/2$ (Lemma \ref{lem:G_property}). Recall that the special functions $F$ and $G$ are defined in Appendixes \ref{sec:F} and \ref{sec:G}, respectively. On the other hand, for any $0<z<x$, we use $F'(1/t)\geq 4t^3/3$ for $t\in[0,1]$ to find that 
\begin{align*}
\mtx{T}(f)'(x) &\leq \frac{1}{\pi}\int_0^{x}f'(y)\cdot \frac{4y^3}{3x^3}\idiff y
\leq \frac{1}{\pi}\int_z^{x}f'(y)\cdot \frac{4z^3}{3x^3}\idiff y\\
&= \frac{4z^3}{3\pi x^3}(f(x)-f(z))\leq \frac{4z^3}{3\pi x^3}(f(x)-1 + z^2).
\end{align*}
We then choose $z = ((1-f(x))/2)^{1/2}$ to obtain 
\begin{equation}\label{eqt:T_derivative_bound}
\mtx{T}(f)'(x) \leq -\frac{1}{3\sqrt{2}\pi x^3}\cdot(1-f(x))^{5/2}. 
\end{equation}
Putting these together, we reach
\[|\mtx{T}(f)'(x)| \geq \left(\frac{1}{3\sqrt{2}\pi x^3}\right)^{1/5}(1-f(x))^{1/2}\cdot \left(\frac{x|f'(x)|}{2\pi}\right)^{4/5} = \frac{1}{\pi}\left(\frac{x}{48\sqrt{2}}\right)^{1/5}|f'(x)|^{4/5}(1-f(x))^{1/2}.\]
Finally, we find
\[h'(x)\leq \frac{3(1-x^2)\mtx{T}(f)'(x)}{c(f)(1+2x^2)}\leq -\frac{3x^{6/5}(1-x)}{4(48\sqrt{2})^{1/5}(1+2x^2)}|f'(x)|^{4/5},\quad x\in[0,1].\]
We have used the $x$-dependent upper bound of $c(f)$ in Lemma \ref{lem:cf_bound}. In particular, plugging in $x=1/2$ gives
\begin{equation}\label{eqt:bootstrap}
h'(1/2)\leq -\left(\frac{1}{3\cdot2^{20}\sqrt{2}}\right)^{1/5}|f'(1/2)|^{4/5}.
\end{equation}
This explains the choice of the constant $\eta = 1/(3\cdot2^{20}\sqrt{2})$ in the definition of $\mathbb{D}$. We then use $f'(1/2)\leq -\eta$ to obtain $h'(1/2)\leq -\eta^{1/5}|f'(1/2)|^{4/5}\leq -\eta$.\\

Combining these steps proves the lemma for $a\leq 1$ and $a\neq0$. As for $\mtx{R}_0$, we note that for any $f\in \mathbb{D}$ and any $x\in\mathbb{R}$, $\lim_{a\rightarrow0}\mtx{R}_a(f)(x) = \mtx{R}_0(f)(x)$. Hence, the lemma is also true for $a=0$. 
\end{proof}

Next, we show that $\mtx{R}_a$ is continuous on $\mathbb{D}$ in the $L_\rho^\infty$-topology.

\begin{theorem}\label{thm:Ra_continuous}
For $a\leq 1$, $\mtx{R}_a:\mathbb{D}\to\mathbb{D}$ is continuous with respect to the $L_\rho^\infty$-norm.
\end{theorem}

\begin{proof} 
Recall that $\rho(x) = (1+|x|)^{-1/2}$. Given any (fixed) $f_0\in \mathbb{D}$, we only need to prove that $\mtx{R}_a$ is $L_\rho^\infty$-continuous at $f_0$. Denote $g_0 := \mtx{T}_a(f_0)$. Let $\epsilon>0$ be an arbitrarily small number. Since $\mtx{R}_a(f_0)$ is bounded, continuous, and non-increasing on $[0,+\infty)$, there is some $X_0>1$ such that 
\[\rho(X_0)g_0(X_0)\geq \rho(X_0)\mtx{R}_a(f_0)(X_0) = \epsilon.\]
This also means $\rho(x)\mtx{R}_a(f_0)(x)\leq \epsilon$ for $x\geq X_0$.

Let $f\in \mathbb{D}$ be arbitrary, and denote similarly $g := \mtx{T}_a(f)$. Suppose that $\|\rho(f-f_0)\|_{L^\infty}\leq \delta$ for some sufficiently small $\delta>0$. For any $x\geq 0$, we have 
\begin{align*}
|\mtx{T}(f)(x)-\mtx{T}(f_0)(x)| &= \frac{1}{\pi}\left|\int_0^{+\infty}(f(y)-f_0(y))\left(\frac{y}{x}\ln\left|\frac{x+y}{x-y}\right|-2\right)\idiff y\right|\\
&\leq \frac{\delta}{\pi}\int_0^{+\infty}(1+y)^{1/2}\left|\frac{y}{x}\ln\left|\frac{x+y}{x-y}\right|-2\right|\idiff y\\
&=\frac{\delta}{\pi}\cdot x\int_0^{+\infty}(1+tx)^{1/2}\left|t\ln\left|\frac{t+1}{t-1}\right|-2\right|\idiff t\\
&\leq \frac{\delta}{\pi}\cdot (1+x)^{1/2}x\int_0^{+\infty}(1+t)^{1/2}\left|t\ln\left|\frac{t+1}{t-1}\right|-2\right|\idiff t\\
&\lesssim \delta x(1+x)^{1/2}.
\end{align*}
The last integral of $t$ above is finite since $t\ln|(t+1)/(t-1)|-2 = O(t^{-2})$ as $t\rightarrow +\infty$. A similar argument shows that $|T(f)(x)|,|T(f_0)(x)|\lesssim x$.
Combining these estimates with Lemma \ref{lem:cf_bound} and Lemma \ref{lem:cf_continuous} yields 
\[|g(x) - g_0(x)| \lesssim \tilde{a}\left(\delta x(1+x)^{1/2} + \delta^{1/2}x\right) \lesssim \tilde{a}\delta^{1/2}x(1+x)^{1/2},\]
where $\tilde{a}=|a|/(1+|a|)$. This means that, for any $a\in(-\infty,0)\cup(0,1]$ and for any $x\in[0,X_0]$,
\[\rho(x)|g(x)^{1/a} - g_0(x)^{1/a}|\leq \frac{1}{|a|}\rho(x)|g(x) - g_0(x)|(g(x)^{(1-a)/a} + g_0(x)^{(1-a)/a})\lesssim \delta^{1/2}X_0.\] 
Moreover, provided that $\delta$ is sufficiently small (depending on $X_0$ and $\epsilon$), we shall have 
\[g(x)\geq \min\{g(X_0)\ ,\ g(0)\}\geq \min\left\{g_0(X_0) - \frac{\epsilon}{2\rho(X_0)}\ ,\ 1\right\}\gtrsim \frac{\epsilon}{\rho(X_0)},\quad \text{for $0\leq x\leq X_0$}.\]
It then follows that, for any $x\in [0,X_0]$,  
\begin{align*}
\int_0^x\left|\frac{g(y)-g(0)}{yg(y)}- \frac{g_0(y)-g_0(0)}{yg_0(y)}\right|\idiff y &= \int_0^x\frac{|g(y)-g_0(y)|}{yg(y)g_0(y)}\idiff y\\
&\lesssim \frac{\tilde{a}\delta^{1/2}\rho(X_0)^2}{\epsilon^2}\int_0^x(1+y)^{1/2}\idiff y\\
&\lesssim \frac{\tilde{a}\delta^{1/2}(1+X_0)^{1/2}}{\epsilon^2}.
\end{align*}
Hence, for any $a\in(-\infty,0)\cup(0,1]$, we obtain that 
\begin{align*}
&\rho(x)|\mtx{R}_a(f)(x) - \mtx{R}_a(f_0)(x)| \\
&= \rho(x)\left|g(x)^{1/a}\exp\left(\frac{1-a}{a}\int_0^x\frac{g(y)-g(0)}{yg(y)}\idiff y\right) - g_0(x)^{1/a}\exp\left(\frac{1-a}{a}\int_0^x\frac{g_0(y)-g_0(0)}{yg_0(y)}\idiff y\right)\right|\\
&\lesssim \delta^{1/2}X_0 + \frac{\delta^{1/2}(1+X_0)^{1/2}}{\epsilon^2} \lesssim \frac{\delta^{1/2}X_0}{\epsilon^2},\quad x\in [0,X_0].
\end{align*}
Again, provided that $\delta$ is sufficiently small, we can have $\rho(X_0)\mtx{R}_a(f)(X_0)\leq 2\epsilon$. By the monotonicity of $\mtx{R}_a(f)$, we also have $\rho(x)\mtx{R}_a(f)(x)\leq 2\epsilon$ for $x\geq X_0$. Therefore, we can choose $\delta$ small enough ($\delta\lesssim \epsilon^{6}X_0^{-2}$) so that 
\[\|\rho(\mtx{R}_a(f) - \mtx{R}_a(f_0))\|_{L^\infty}\lesssim \epsilon\]
for all $f\in\mathbb{D}$ such that $\|\rho(f-f_0)\|_{L^\infty}\leq \delta$. For the case $a=0$, the same result can be proved by taking the limit $a\rightarrow 0$, since all the estimates above do not rely on the value of $a$ (i.e., the constants hidden in the symbol ``$\lesssim$'' do not depend on $a$). Or, we can simply use the formula \eqref{eqt:R_0} for $\mtx{R}_0$ and carry out a similar estimate as above. We have thus proved that $\mtx{R}_a$ is $L_\rho^\infty$-continuous at $f_0$ as $\epsilon$ is arbitrary.
\end{proof}

We now turn to study the far-field behavior of $\mtx{R}_a(f)$ for $f\in \mathbb{D}$. The next two lemmas are not needed in proving the existence of a fixed point of $\mtx{R}_a$. Nevertheless, they will be useful in the next section for studying the far-field decay of a fixed point. We present them here because the results hold for all $f\in \mathbb{D}$. We first control the decay rate of $\mtx{R}_a(f)$ for $a<1$ as follows.

\begin{lemma}\label{lem:Ra_decay}
For any $a< 1$ and any $f\in \mathbb{D}$,
\[
\lim_{x\rightarrow+\infty }x^{\delta}\cdot \mtx{R}_a(f)(x) = 0\quad\text{for all}\quad \delta < r_a(f),
\]
where
\[r_a(f) = \frac{2(1-a)b(f)}{\big((1-a/3)c(f)-2a b(f)\big)_+} > 0.\]
When $2ab(f)\geq (1-a/3)c(f)$, $r_a(f)$ is defined to be $+\infty$. As a corollary, $\lim_{x\rightarrow+\infty }\mtx{R}_a(f)(x)=0$ for any $a< 1$ and any $f\in \mathbb{D}$.
\end{lemma}

\begin{proof}
Given $f\in \mathbb{D}$, let $g=\mtx{T}_a(f)$ and $f_\infty = f(+\infty)$. Recall that $b(f)$ can be $+\infty$ for a $f\in \mathbb{D}$. Nevertheless, our argument below works in either case, $b(f)<+\infty$ or $b(f)=+\infty$.

We first give a lower bound of the ratio $b(f)/c(f)$. It follows from $f(x)\geq\max\{(1-x^2)_+,f_\infty\}$ that
\[
b(f) = \frac{2}{\pi}\int_0^{+\infty}\big(f(y)-f(+\infty)\big)\idiff y \geq \frac{2}{\pi}\int_0^{\sqrt{1-f_\infty}}\big(1-y^2-f(+\infty)\big)\idiff y = \frac{4}{3\pi}(1-f_\infty)^{3/2}
\]
and 
\[c(f) = \frac{2}{\pi}\int_0^{+\infty}\frac{1-f(y)}{y^2}\idiff y \leq \frac{2}{\pi}\int_0^{\sqrt{1-f_\infty}}1\idiff y + \frac{2}{\pi}\int_{\sqrt{1-f_\infty}}^{+\infty}\frac{1-f_\infty}{y^2}\idiff y=\frac{4}{\pi}(1-f_\infty)^{1/2}.\]
In view of \eqref{eqt:f_range}, we have $f_\infty\leq 1-\eta/4$ with $\eta = 1/(3\cdot 2^{20}\sqrt{2}) > 0$. Therefore, 
\begin{equation}\label{eqt:Ra_decay_step1}
\frac{b(f)}{c(f)} \geq \frac{1}{3}(1-f_\infty)\geq \frac{\eta}{12}>0.
\end{equation}
Note that the first inequality above is an equality if and only if $f(x) \equiv (1-x^2-f_\infty)_+ + f_\infty\in\mathbb{D}$.

Next, we assume that $a<1$ and $a\neq 0$. Denote $k=b(f)/c(f)$. Note that 
\[r_a(f) = \frac{2(1-a)k}{\big(1-a/3-2a k\big)_+} = \frac{1-a}{a}\left(\frac{1}{g(+\infty)}-1\right).\]
Formally, this formula is valid even when $g(+\infty) = 0$. Since $\sgn(a)g(x)$ non-increasingly converges to $\sgn(a)g(+\infty)$, for any $\delta<r_a(f)$, there is some $\bar{\delta}\in(\delta,r_a(f))$ and some $X_\delta>0$ such that, in spite of the sign of $a$, 
\[\frac{1-a}{a}\left(\frac{1}{g(x)}-1\right) \geq \frac{1-a}{a}\left(\frac{1}{g(X_\delta)}-1\right) \geq \bar\delta > \delta,\quad \text{for $x\geq X_\delta$}.\]
Also note that we always have $g(x)^{1/a}\leq 1$ regardless of the sign of $a$. Then, for $x\geq X_\delta$, we have
\begin{align*}
\mtx{R}_a(f)(x)&\lesssim \exp\left(\frac{1-a}{a}\int_{X_\delta}^x\frac{g(y)-g(0)}{yg(y)}\idiff y\right) \\
&= \exp\left(\frac{1-a}{a}\int_{X_\delta}^x\frac{1}{y}\cdot\left(1-\frac{1}{g(y)}\right)\idiff y\right) \leq  \exp\left(-\int_{X_\delta}^x\frac{\bar\delta}{y}\idiff y\right) \lesssim x^{-\bar\delta}.
\end{align*}
Hence, $\lim_{x\rightarrow+\infty}x^\delta \mtx{R}_a(f)(x)=0$. Moreover, if $k\geq (1-a/3)/2a$, then $r_a(f)=+\infty$. Otherwise, the inequalities in \eqref{eqt:Ra_decay_step1} imply that 
\begin{align*}
r_a(f) &= \frac{2(1-a)k}{1-a/3-2a k} \geq \frac{(1-a)\eta}{6-2a-a\eta} > 0.
\end{align*}
The case $a=0$ can be handled similarly by directly using the special formula \eqref{eqt:R_0} of $\mtx{R}_0$.
\end{proof}

When $\mtx{R}_a(f)$ is not compactly supported, the next lemma provides a point-wise lower bound of $\mtx{R}_a(f)$ in terms of the ratio $b(f)/c(f)$.

\begin{lemma}\label{lem:Ra_lower_bound}
Given $f\in \mathbb{D}$, let $k = b(f)/c(f)$ and suppose $2ak< 1-a/3$. Then, for any $x\geq 1$, 
\[\mtx{R}_a(f)(x)\geq \left(1-\frac{2ak}{1-a/3}\right)^{1/a}\econst^{-1/3}\cdot x^{-r_a},\]
where 
\begin{equation}
r_a = \frac{2(1-a)k}{1-a/3-2a k}.
\end{equation}
\end{lemma}

\begin{proof}
Let $g = \mtx{T}_a(f)$. By the monotonicity of $\sgn(a)g(x)$, we have 
\[\sgn(a)g(x)\geq \sgn(a)g(+\infty) = \sgn(a)\cdot \left(1-\frac{2ab(f)}{(1-a/3)c(f)}\right) =  \sgn(a)\cdot\left(1-\frac{2ak}{1-a/3}\right).\]
This implies that, in spite of the sign of $a$,
\[g(x)^{1/a}\geq g(+\infty)^{1/a} = \left(1-\frac{2ak}{1-a/3}\right)^{1/a},\]
and
\[\frac{1-a}{a}\left(\frac{1}{g(x)}-1\right) \leq \frac{1-a}{a}\left(\frac{1}{g(+\infty)}-1\right) = r_a. \]
For $x\in[0,1]$, we can use the first inequality in \eqref{eqt:Ra_close_midstep2} when $a\geq 0$ or the second inequality in \eqref{eqt:Ra_close_midstep2} when $a<0$ to obtain 
\[\frac{1-a}{a}\left(\frac{1}{g(x)}-1\right)\leq \frac{1-a}{a}\cdot\frac{\frac{2a}{3-a}x^2}{1 - \frac{2a}{3-a}x^2}= \frac{2(1-a)x^2}{3-a-2ax^2}\leq  \frac{2}{3}x^2.\]
Therefore, for $x\geq 1$,
\begin{align*}
\frac{1-a}{a}\int_0^x\frac{g(0)-g(y)}{yg(y)}\idiff y &= \frac{1-a}{a}\int_0^1\frac{1}{y}\left(\frac{g(0)}{g(y)}-1\right)\idiff y + \frac{1-a}{a}\int_1^x\frac{1}{y}\left(\frac{g(0)}{g(y)}-1\right)\idiff y\\
&\leq \frac{2}{3}\int_0^1y\idiff y + r_a\int_1^x\frac{1}{y}\idiff y = \frac{1}{3} + r_a\cdot \ln x.
\end{align*}
Finally, we have
\[\mtx{R}_a(f)(x) = g(x)^{1/a}\exp\left(\frac{1-a}{a}\int_0^x\frac{g(y)-g(0)}{yg(y)}\idiff y\right)\geq \left(1-\frac{2ak}{1-a/3}\right)^{1/a}\econst^{-1/3}\cdot x^{-r_a},\]
which is the desired lower bound.
\end{proof}

Roughly speaking, Lemmas \ref{lem:Ra_decay} and \ref{lem:Ra_lower_bound} together imply that $\mtx{R}_a(f)$ is either compactly supported or decaying like $|x|^{-r_a(f)}$ in the far field. This will be made more precise in the next Section when $f$ is a fixed point of $\mtx{R}_a$.

\subsection{Existence of solutions} One last ingredient for establishing existence of fixed points of $\mtx{R}_a$ is the compactness of the set $\mathbb{D}$.

\begin{lemma}\label{lem:compactness} 
The set $\mathbb{D}$ is compact with respect to the $L_\rho^\infty$-norm.
\end{lemma}

\begin{proof} For any $f\in \mathbb{D}$, we use convexity and monotonicity to obtain
\[-\frac{f'(x)}{2x}\leq \frac{f(0)-f(x)}{x^2} \leq \min\{ 1\ , \frac{1}{x^2}\} ,\quad x>0.\]
implying that $|f'(x)|\leq \min\{2x,2x^{-1}\}\leq 2$. Based on this, we show that $\mathbb{D}$ is sequentially compact. 

Let $\{f_n\}_{n=1}^{+\infty}$ be an arbitrary sequence in $\mathbb{D}$. Initialize $n_{0,k}=k$, $k\geq 1$. For each integer $m\geq 1$, let $\epsilon_m = 2^{-m}$ and $L_m = \epsilon_m^{-2}$. It follows that $\rho(x)f_n(x)\leq \rho(x)\leq \epsilon_m$ for all $x\geq L_m$. Furthermore, since $|f_n'(x)|\leq 2$ on $[0,L_m]$, we can apply Ascoli's theorem to select a sub-sequence $\{f_{n_{m,k}}\}_{k=1}^{+\infty}$ of $\{f_{n_{m-1,k}}\}_{k=1}^{+\infty}$ such that $\|\rho(f_{n_{m,i}}-f_{n_{m,j}})\|_{L^\infty}\leq 2\epsilon_m$ for any $i,j\geq 1$. Then the diagonal sub-sequence $\{f_{n_{m,m}}\}_{m=1}^{+\infty}$ is a Cauchy sequence in the $L_\rho^\infty$-norm. This proves that $\mathbb{D}$ is sequentially compact.
\end{proof}

We are now ready to prove the existence of fixed points of $\mtx{R}_a$ for any $a\leq 1$ using the Schauder fixed-point theorem.

\begin{theorem}\label{thm:existence_fixed_point}
For each $a\leq 1$, the map $\mtx{R}_a$ has a fixed point $f_a\in\mathbb{D}$, i.e. $\mtx{R}_a(f_a)=f_a$. As a corollary, for each $a\leq 1$, \eqref{eqt:main_equation} admits a solution $(\om, c_l, c_\om)$ with $f = -\om/x\in \mathbb{D}$ and $c_l, c_\om$ given in Proposition \ref{prop:fixed_point_solution}.
\end{theorem}

\begin{proof}
By Theorem \ref{thm:Ra_continuous} and Lemma \ref{lem:compactness}, $\mathbb{D}$ is convex, closed and compact in the $L_\rho^\infty$-norm, and $\mtx{R}_a$ continuously maps $\mathbb{D}$ into itself. The Schauder fixed-point theorem implies that $\mtx{R}_a$ has a fixed point in $\mathbb{D}$. The second part of the theorem then follows from Proposition \ref{prop:fixed_point_solution}.
\end{proof}

We remark that we are not able to prove the uniqueness of fixed points of $\mtx{R}_a$ in $\mathbb{D}$ for general values of $a$, except for the special cases $a=0$ and $a=1$ (see Section \ref{sec:review}). However, based on our numerical observations in Section \ref{sec:numerical}, we conjecture the following $a$-monotone and $a$-continuous properties of the fixed points: 

\begin{conjecture}
For each $a\leq 1$, let $f_a\in \mathbb{D}$ be a fixed point of $\mtx{R}_a$. Then, for any $a_1\leq a_2\leq 1$, 
\[f_{a_1}(x)\geq f_{a_2}(x),\quad x\in \mathbb{R}.\]
Moreover, there is a family of fixed points $\{f_a: f_a = \mtx{R}_a(f_a)\}_{a\leq1}\subset \mathbb{D}$  such that $f_a$ depends continuously on $a$ in the $L_\rho^\infty$-norm. 
\end{conjecture}

Let us explain what we can obtain if this conjecture is true. Firstly, it immediately implies the uniqueness of the fixed points of $\mtx{R}_a$. In fact, if $f_a$ and $\tilde{f}_a$ are two fixed points of $\mtx{R}_a$, then $f_a(x)\leq \tilde{f}_a(x)$ and $f_a(x)\geq \tilde{f}_a(x)$ are both true for all $x$, implying that $f_a\equiv \tilde{f}_a$.

Secondly, it implies the existence of a critical value $a_c$ (as predicted in \cite{lushnikov2021collapse}) such that $f_a$ must be compactly supported if $a>a_c$ and $f_a$ must be strictly positive on $\mathbb{R}$ if $a\leq a_c$. To see this, we note that if $f_{\tilde{a}}$ is compactly supported, then $f_a$ must also be compactly supported for any $a\geq \tilde{a}$ provided that the conjecture is true. That is, the set 
\[\mathbb{S}_a = \{a\leq 1:\ \text{the fixed point $f_a=\mtx{R}_a(f_a)$ is compactly supported}\}\]
is a continuous interval. We will show that, for example, the unique fixed point $f_0$ of $\mtx{R}_0$ is strictly positive on $\mathbb{R}$ (see Section \ref{sec:review}). Hence, the value $a_c:= \inf \mathbb{S}_a$ is lower bounded by $0$ and thus is finite. 

Moreover, we can show that the $a$-continuity of $f_a$ in the $L_\rho^\infty$-norm implies that $c(f_a)$ and $b(f_a)$ are both continuous in $a$ (see Lemma \ref{lem:cf_continuous} and the proof of Lemma \ref{lem:uniform_small}), and so is the quantity $b(f_a)/c(f_a) - (1-a/3)/2a$. Therefore, by Corollary \ref{cor:Ta_property}, we have $b(f_{a_c})/c(f_{a_c}) - (1-a_c/3)/2a_c=0$ provided that the second claim of our conjecture is true, and thus $f_{a_c}$ is strictly positive on $\mathbb{R}$.

\section{General properties of solutions}\label{sec:general_properties}
In this section, we study general properties of fixed-point solutions 
\[f=\mtx{R}_a(f) = \big(\mtx{T}_a(f)(x)\big)^{1/a}\cdot\exp\left(\frac{1-a}{a}\int_0^x\frac{\mtx{T}_a(f)(y)-\mtx{T}_a(f)(0)}{y\mtx{T}_a(f)(y)}\idiff y\right)\]
for $a\leq 1$. With the fixed-point relation in hand, we are able to refine some of the estimates in the previous section and obtain more accurate characterizations of these fixed points, which then transfer to characterizations of the corresponding solutions of \eqref{eqt:main_equation}.

In what follows, we will always denote by $f$ (or $f_a$ when we emphasize its dependence on $a$) a fixed point of $\mtx{R}_a$ for some $a\leq 1$. Recall the definitions
\[b(f) := \frac{2}{\pi}\int_0^{+\infty}\big(f(y)-f(+\infty)\big)\idiff y,\]
\[c(f) := \frac{2}{\pi}\int_0^{+\infty}\frac{f(0)-f(y)}{y^2}\idiff y,\]
\begin{equation}\label{eqt:ra}
r_a(f) := \frac{2(1-a)b(f)}{\big((1-a/3)c(f)-2a b(f)\big)_+}. 
\end{equation}
As we will see, the ratio $b(f)/c(f)$ determines the asymptotic behavior of $f$ as $x\rightarrow +\infty$. In fact, according to Corollary \ref{cor:Ta_property} we already know that $f=\mtx{R}_a(f)$ is compactly supported if and only if $a>0$ and $b(f)/c(f) > (1-a/3)/2a$. Moreover, in view of Lemmas \ref{lem:Ra_decay} and \ref{lem:Ra_lower_bound}, the number $r_a(f)$ characterizes the decay rate of $f$ when $f$ is strictly positive on $\mathbb{R}$.\\

Let us outline what we are going to accomplish in this section. First, we will derive reasonable bounds on the ratio $b(f)/c(f)$ that depend on the parameter $a$. This is done by establishing an integral identity (see \eqref{eqt:Qf_bf} below) from the fundamental equation \eqref{eqt:main_equation} that involves $b(f)$, $c(f)$, $a$, and some bilinear form of $f$. The $a$-dependent estimates of $b(f)/c(f)$ will then tell us for what values of $a$ the corresponding fixed point $f_a$ is compactly supported and for what values of $a$ it is not. 

Next, we will establish some uniform (in $a$) far-field decay bounds and moment bounds for fixed points of $\mtx{R}_a$, which is also accomplished based on the integral identity mentioned above and its variants. We then refine on these uniform estimates of fixed points to further obtain finer characterizations of their asymptotic behavior in the far field. 

After that, we will prove the smoothness of the fixed-point solutions either on $\R$ or in the interior of their supports. This is done by exploiting the gaining of regularity through the linear map $\mtx{T}$.

Finally, we collect all established results and give a proof of our main theorem. 

\subsection{Estimates of $b(f)/c(f)$}
We first show that a fixed-point solution $f=\mtx{R}_a(f)$ must decay to $0$ at the infinity. 

\begin{lemma}\label{lem:f_infinity}
Let $f\in \mathbb{D}$ be a fixed point of $\mtx{R}_a$ for some $a\leq 1$. Then, $f(+\infty) = \lim_{x\rightarrow+\infty}f(x) = 0$.
\end{lemma}

\begin{proof}
Lemma \ref{lem:Ra_decay} implies that $f(+\infty) = \mtx{R}_a(f)(+\infty) =0$ for any $a<1$. Hence, we only need to prove the lemma for $a=1$, in which case $f = \mtx{R}_1(f) = \mtx{T}_1(f)$. Suppose that $f_\infty := f(+\infty) >0$. Then, we have
\[f_\infty = \mtx{T}_1(f)(+\infty) = 1 - \frac{3b(f)}{c(f)},\]
that is,
\[\frac{b(f)}{c(f)} = \frac{1}{3}(1-f_\infty).\]
In view of the first inequality in \eqref{eqt:Ra_decay_step1} and the arguments right before it, we derive that
\[f(x) \equiv (1-x^2 -f_\infty)_+ + f_\infty.\]
However, it is argued in Appendix \ref{sec:f_m} that a function of the form $(1-x^2-p)_+ + p$ with $p\in[0,1)$ cannot be a fixed point of $\mtx{R}_1$. This contradiction implies that $f(+\infty) = 0$. 
\end{proof}

Lemma \ref{lem:f_infinity} implies that the expression of $b(f)$ for a fixed point $f=\mtx{R}_a(f)$ can be simplified as 
\[b(f) = \frac{2}{\pi}\int_0^{+\infty}f(y)\idiff y.\]
Based on this we can further derive a lower bound of the decay rate of a fixed-point solution $f$, showing that $f$ must decay faster that $x^{-1}$ for any $a\leq 1$.

\begin{lemma}\label{lem:finite_bf}
Let $f\in \mathbb{D}$ be a fixed point of $\mtx{R}_a$ for some $a\leq 1$, and let $c_\om$ be given by \eqref{eqt:f_to_cw}. Then $b(f)<+\infty$. As a corollary, $r_a(f) > 1$ and $c_\om<0$.
\end{lemma}

\begin{proof} We first prove $b(f)<+\infty$ by contradiction. Suppose that $b(f) = +\infty$. If $a>0$, we must have $b(f)/c(f)> (1-a/3)/2a$, and thus $f = \mtx{R}_a(f)$ is compactly supported. However, this implies $b(f)\lesssim \int_0^{+\infty}f(y)\idiff y<+\infty$, which is a contraction. If $a\leq 0$, we have 
\[r_a(f) = \frac{2(1-a)b(f)}{(1-a/3)c(f)-2a b(f)} = \frac{1+|a|}{|a|}> 1.\]
By Lemma \ref{lem:Ra_decay}, there is some $\delta >1$ such that $f\lesssim \min\{1,x^{-\delta}\}$. This again leads to the contradiction that $b(f)\lesssim \int_0^{+\infty}f(y)\idiff y<+\infty$.

Next, we argue that $r_a(f)>1$ and $c_\om<0$. If $2ab(f)\geq (1-a/3)c(f)$, then $r_a(f)=+\infty>1$ by the definition \eqref{eqt:ra}. In addition, if $a<1$, we have
\[c_\om = \frac{1-a/3}{2}c(f) - b(f) < \frac{1-a/3}{2}c(f) - ab(f)\leq0.\]
As for $a=1$, we have argued in the proof of Lemma \ref{lem:Ra_decay} that the inequality $b(f)\geq c(f)/3$ is an equality if and only if $f \equiv f_m := (1-x^2)_+$. However, it is easy to check that $f_m$ cannot be a fixed point of $\mtx{R}_1$ (since $\mtx{T}_1(f_m)(x)\gtrsim x^{-2}$ for $x\geq1$; see Appendix \ref{sec:f_m}). Thus, we must have $b(f)>c(f)/3$ when $a=1$, which again implies that $c_\om<0$.

If $2ab(f)< (1-a/3)c(f)$, Lemma \ref{lem:Ra_lower_bound} states that $f(x)\gtrsim x^{-r_a(f)}$ for $x\geq 1$, which implies $r_a(f)>1$ since $b(f)<+\infty$. Moreover, we can use \eqref{eqt:f_to_cl} and \eqref{eqt:f_to_cw} to compute that
\begin{equation}\label{eqt:cw_cl}
\frac{c_\om}{c_l} = 1 - \frac{2(1-a)b(f)}{(1-a/3)c(f)-2ab(f)} = 1-r_a(f).
\end{equation}
Note that $c_l>0$ in this case. Therefore, $r_a(f)>1$ implies $c_\om<0$.
\end{proof}

Lemma \ref{lem:finite_bf} states that each fixed-point solution $f$ corresponds to a negative $c_\om$, implying that the profile $\om = -xf$ corresponds to a self-similar finite-time blowup of the gCLM model \eqref{eqt:gCLM} in the form \eqref{eqt:self-similar_solution} (recall that now $\om$ stands for the profile $\Omega$ in \eqref{eqt:self-similar_solution}).\\

Now, based on the uniform estimate $r_a(f)>1$, we can establish a finer estimate on the ratio $b(f)/c(f)$ for a fixed point $f\in \mathbb{D}$.

\begin{theorem}\label{thm:k_bound}
Given any $a\leq 1$, let $f\in \mathbb{D}$ be a fixed point of $\mtx{R}_a$. Then, 
\begin{equation}\label{eqt:k_formula}
\frac{b(f)}{c(f)} = \frac{1-a/3}{1+a\mu(f)},
\end{equation}
for some $f$-dependent constant $\mu(f)$ such that
\[0 \leq \mu(f) \leq \overline{\mu} := \frac{9\pi^2}{64}-\frac{3}{4} <1.\]
In particular, when $a\in[0,1]$, 
\[\mu(f) \geq \underline{\mu} := \frac{\overline{\mu}}{9(1-a/3)^2} > 0.\]
\end{theorem}

\begin{proof}
Let $\om = -xf$, $u = -(-\Delta)^{-1/2}\om$, and let $c_l,c_\om$ be given by \eqref{eqt:f_to_cl} and \eqref{eqt:f_to_cw}, so that the tuple $(\om,u,c_l,c_\om)$ satisfies equation \eqref{eqt:main_equation}. We then handle each term in \eqref{eqt:main_equation} separately:
\begin{itemize}[leftmargin=*]
\item Owing to Lemma \ref{lem:finite_bf}, 
\[\int_0^{+\infty}\om'(x)\idiff x = \om(+\infty) = -\lim_{x\rightarrow+\infty} xf(x) = 0.\]
\item By a property of the Hilbert transform (Lemma \ref{lem:Hilbert_property1}),  
\[\frac{2}{\pi}\int_0^{+\infty}\frac{u'(x)\om(x)}{x}\idiff x = \frac{2}{\pi}\int_0^{+\infty}\frac{\mtx{H}(\om)(x)\cdot \om(x)}{x}\idiff x = -\frac{1}{2}\big(\mtx{H}(\om)(0)\big)^2 = -\frac{b(f)^2}{2}.\]
\item Since $u(x)/x = (-\Delta)^{-1/2}(xf)/x = \mtx{T}(f)(x)+b(f)$ (see \eqref{eqt:T_to_laplacian}), 
\[-\frac{2}{\pi}\int_0^{+\infty}\frac{u(x)\om'(x)}{x}\idiff x = \frac{2}{\pi}\int_0^{+\infty}\left(\frac{u(x)}{x}\right)'\om(x)\idiff x = -\frac{2}{\pi}\int_0^{+\infty}\mtx{T}(f)'(x)\cdot xf(x)\idiff x =: Q(f).\]
\end{itemize}
Putting these together, we can multiply both sides of equation \eqref{eqt:main_equation} (or \eqref{eqt:main_equation_modified}) by $-(2/\pi)x^{-1}$ and then integrate them over $[0,+\infty)$ to get 
\begin{equation}\label{eqt:Qf_bf}
aQ(f) = c_\om b(f) + \frac{b(f)^2}{2} = \frac{1-a/3}{2}c(f)b(f) - \frac{b(f)^2}{2},
\end{equation}
that is, 
\[\frac{b(f)}{c(f)} = \frac{1-a/3}{1+a\mu},\quad \text{where}\quad \mu = \mu(f) := \frac{2Q(f)}{b(f)^2}.\]

We now only need to estimate $\mu$. Using formula \eqref{eqt:T_derivative} and integration by parts, we can compute that 
\begin{align*}
Q(f) &= -\frac{2}{\pi^2}\int_0^{+\infty}\int_0^{+\infty}f(x)f'(y)xF'(x/y)\idiff x\idiff y\\
&= \frac{2}{\pi^2}\int_0^{+\infty}\int_0^{+\infty}f'(x)f'(y)\left(y^2F_1(x/y)\right)\idiff x\idiff y\\
&= \frac{1}{\pi^2}\int_0^{+\infty}\int_0^{+\infty}f'(x)f'(y)xy\left(\frac{y}{x}F_1(x/y) + \frac{x}{y}F_1(y/x)\right)\idiff x\idiff y\\
&= \frac{1}{\pi^2}\int_0^{+\infty}\int_0^{+\infty}f'(x)f'(y)xyF_2(x/y)\idiff x\idiff y,
\end{align*}
where 
\[F_1(t) := \int_0^tsF'(s)\idiff s,\quad F_2(t) := t^{-1}F_1(t) + tF_1(t^{-1}),\quad t\geq 0.\]
Carrying on the calculations above, we get
\begin{align*}
Q(f) &= \frac{1}{\pi^2}\int_0^{+\infty}\int_0^{+\infty}\frac{f'(x)}{x}\frac{f'(y)}{y}x^2y^2F_2(x/y)\idiff x\idiff y\\
&=-\frac{1}{\pi^2}\int_0^{+\infty}\int_0^{+\infty}\left(\frac{f'(x)}{x}\right)'\frac{f'(y)}{y}\left(y^5F_3(x/y)\right)\idiff x\idiff y\\
&= \frac{1}{\pi^2}\int_0^{+\infty}\int_0^{+\infty}\left(\frac{f'(x)}{x}\right)'\left(\frac{f'(y)}{y}\right)'x^3y^3F_4(x/y)\idiff x\idiff y,
\end{align*}
where 
\[F_3(t) := \int_0^ts^2F_2(s)\idiff s,\quad F_4(t) := t^3\int_0^{1/t}s^5F_3(1/s)\idiff s,\quad t\geq 0.\]
Taylor expansions of these special functions $F_i, i =1,2,3,4$, are presented in Appendix \ref{sec:F_i}, with which we find that that $0\leq F_4(t)\leq F_4(1) = \pi^2/32-1/6$ for all $t\geq 0$. Thus, using that fact that $\left(f'(x)/x\right)'\geq 0$ for $f\in \mathbb{D}$, we have 
\[0\leq Q(f) \leq F_4(1)\left(\frac{1}{\pi}\int_0^{+\infty}\left(\frac{f'(y)}{y}\right)'y^3\idiff y\right)^2.\]
As for $b(f)$, integration by parts gives 
\[b(f) = \frac{2}{\pi}\int_0^{+\infty}f(y)\idiff y = -\frac{2}{\pi}\int_0^{+\infty}f'(y)y\idiff y = \frac{2}{3\pi}\int_0^{+\infty}\left(\frac{f'(y)}{y}\right)'y^3\idiff y.\]
We have used that $\lim_{x\rightarrow+\infty}x^2f'(x) = 0$. Therefore, 
\[\mu = \frac{2Q(f)}{b(f)^2}\leq \frac{9}{2}F_4(1) = \frac{9\pi^2}{64}-\frac{3}{4} =:\overline{\mu}<1.\]

To obtain the claimed lower bound of $\mu$ for $a\geq0$, we first rewrite $Q(f)$ as
\begin{align*}
Q(f) &= \frac{2}{\pi}\int_0^{+\infty}\left(\frac{u(x)}{x}\right)'\om(x)\idiff x \\
&= \frac{2}{\pi}\int_0^{+\infty}\frac{u'(x)\om(x)}{x}\idiff x - \frac{2}{\pi}\int_0^{+\infty}\frac{u(x)\om(x)}{x^2}\idiff x\\
&= -\frac{b(f)^2}{2} + \frac{2}{\pi}\int_0^{+\infty}\big(\mtx{T}(f)(x)+b(f)\big)f(x)\idiff x\\
&= -\frac{b(f)^2}{2} + \frac{2}{\pi^2}\int_0^{+\infty}\int_0^{+\infty}f(x)f(y)\frac{y}{x}\ln\left|\frac{x+y}{x-y}\right|\idiff x\idiff y\\
&= -\frac{b(f)^2}{2} + \frac{1}{\pi^2}\int_0^{+\infty}\int_0^{+\infty}f(x)f(y)\left(\frac{x}{y}+\frac{y}{x}\right)\ln\left|\frac{x+y}{x-y}\right|\idiff x\idiff y\\
&= \frac{1}{\pi^2}\int_0^{+\infty}\int_0^{+\infty}f(x)f(y)\left(\left(\frac{x}{y}+\frac{y}{x}\right)\ln\left|\frac{x+y}{x-y}\right|-2\right)\idiff x\idiff y.
\end{align*}
Note that 
\begin{equation}\label{eqt:k_bound_step}
\left(\frac{x}{y}+\frac{y}{x}\right)\ln\left|\frac{x+y}{x-y}\right|-2\geq 0,\quad \text{for all $x,y\geq 0$}.
\end{equation}
Hence, $Q(f)\geq Q(f_m)$ for any $f\in \mathbb{D}$, where $f_m(x) := (1-x^2)_+$. In fact, according to Appendix \ref{sec:f_m}, we have exactly $2Q(f_m)/b(f_m)^2 = \overline{\mu}$, and thus
\[Q(f) \geq Q(f_m) = \frac{b(f_m)^2}{2}\cdot \overline{\mu} = \frac{8\overline{\mu}}{9\pi^2}.\]
In the case $a\geq 0$, we use the crude lower bound $\mu \geq 0$ to get 
\[b(f) = \frac{1-a/3}{1+a\mu}c(f)\leq (1-a/3)\cdot\frac{4}{\pi}.\]
It then follows that
\[\mu = \frac{2Q(f)}{b(f)^2} \geq \frac{\overline{\mu}}{9(1-a/3)^2} =: \underline{\mu}.\] 
This completes the proof.
\end{proof}

We can derive from Theorem \ref{thm:k_bound} a series of insightful results about fixed-point solutions of $f=\mtx{R}_a(f)$. First of all, since $b(f)$, $c(f)$, and $\mu(f)$ are positive by definition, the formula \eqref{eqt:k_formula} implies that $1+a\mu(f)>0$ for any $a\leq 1$. This is particularly meaningful when $a<0$, that is, 
\begin{equation}\label{eqt:mu_bound_by_a}
\mu(f) \leq -\frac{1}{a} = \frac{1}{|a|}, \quad a<0.
\end{equation}

As an improvement of the uniform bound in Lemma \ref{lem:finite_bf}, we can derive finer bounds of $r_a(f)$ from Theorem \ref{thm:k_bound}, therefore providing estimates of decay rates of the fixed-point solutions in view of Lemma \ref{lem:Ra_decay} and Lemma \ref{lem:Ra_lower_bound}. In particular, the inequality \eqref{eqt:mu_bound_by_a} helps us determine the limit of $r_a(f)$ as $a\rightarrow -\infty$.

\begin{corollary}\label{cor:ra_bound}
Let $f_a\in \mathbb{D}$ be a fixed point of $\mtx{R}_a$ for some $a\leq 1$, and let $r_a(f_a)$ be given by \eqref{eqt:ra}. Then, 
\[r_a(f_a) \geq \frac{2(1-a)}{\big(1+a\overline{\mu}-2a\big)_+}>2,\quad \text{for $a\in(0,1)$},\]
and  
\[1 < \frac{2(1+|a|)}{1+2|a|} \leq r_a(f_a) \leq \frac{2(1+|a|)}{1+|a|(2-\overline{\mu})} < 2,\quad \text{for $a<0$},\]
where $\overline{\mu}$ is given in Theorem \ref{thm:k_bound}. Moreover, $r_0 = \lim_{a\rightarrow 0}r_a = 2$, and $\lim_{a\rightarrow -\infty}r_a = 1$.
\end{corollary}

\begin{proof}
Write $k_a = b(f_a)/c(f_a)$, $\mu_a=\mu(f_a)$, and $r_a=r_a(f_a)$. Note that $r_1=+\infty$. For $a\in(0,1)$, we have by Theorem \ref{thm:k_bound} that
\[r_a = \frac{2(1-a)}{\big((1-a/3)/k_a-2a\big)_+} = \frac{2(1-a)}{\big(1+a\mu_a-2a\big)_+} \geq \frac{2(1-a)}{\big(1+a\overline{\mu}-2a\big)_+}>2.\]
As for $a<0$, we similarly compute that 
\begin{equation*}
r_a = \frac{2(1-a)}{\big(1+a\mu_a-2a\big)_+} = \frac{2(1+|a|)}{1+|a|(2-\mu_a)}
\begin{cases}
\displaystyle\ \leq \frac{2(1+|a|)}{1+|a|(2-\overline{\mu})} < \frac{2(1+|a|)}{1+|a|} < 2\ , \vspace{4mm} \\ 
\displaystyle\ \geq \frac{2(1+|a|)}{1+|a|(2-0)} = \frac{2(1+|a|)}{1+2|a|} > 1\ .
\end{cases}
\end{equation*}
From the formula of $r_a$ in terms of $\mu_a$, it is easy to see that $\lim_{a\rightarrow 0}r_a = r_0 = 2$. Furthermore, \eqref{eqt:mu_bound_by_a} implies that $0< 1-|a|\mu_a < 1$ for all $a<0$, and thus
\[\lim_{a\rightarrow -\infty}r_a = \lim_{a\rightarrow -\infty}\frac{2(1+|a|)}{1-|a|\mu_a + 2|a|} = 1.\]
This completes the proof.
\end{proof}

We can similarly derive estimates of the ratio $c_l/c_\om$ from Theorem \ref{thm:k_bound}.

\begin{corollary}\label{cor:clcw_ratio}
Let $f_a\in \mathbb{D}$ be a fixed point of $\mtx{R}_a$ for some $a\leq 1$, and let $\gamma_a = -c_l/c_\om$ with $c_l,c_\om$ given by \eqref{eqt:f_to_cl} and \eqref{eqt:f_to_cw}, respectively. Then, 
\begin{equation}\label{eqt:clcw_ratio_formula}
\gamma_a = \frac{1-a(2-\mu(f_a))}{1-a\mu(f_a)}
\begin{cases}
\displaystyle\ \in (-1,1), & a\in(0,1),  \vspace{2mm}\\ 
\displaystyle\ \in (\max\{1,|a|\},1+2|a|), & a <0, 
\end{cases}
\end{equation}
In particular, $\gamma_1 = \lim_{a\rightarrow 1}\gamma_a = -1$, $\gamma_0=\lim_{a\rightarrow 0}\gamma_a = 1$, and $\lim_{a\rightarrow -\infty}\gamma_a = +\infty$.
\end{corollary}

\begin{proof} We use \eqref{eqt:f_to_cl} and \eqref{eqt:f_to_cw} to find that 
\[\gamma_a = -\frac{c_l}{c_\om} = -\frac{1-a/3-2ak}{1-a/3-2k},\]
where $k = b(f)/c(f)$. Formula \eqref{eqt:clcw_ratio_formula} is then a direct result of Theorem \ref{thm:k_bound} and the inequality in \eqref{eqt:mu_bound_by_a}, and the three limits easily follow.
\end{proof}

The next corollary provides intervals of the parameter $a$ where we can determine for sure whether a fixed point $f=\mtx{R}_a(f)$ is compactly supported or strictly positive on $\mathbb{R}$.

\begin{corollary}\label{cor:ac_estimate} 
Let $f_a\in \mathbb{D}$ be a fixed point of $\mtx{R}_a$ for some $a\leq 1$. Then, there are some constants $0<\underline{a}<\overline{a}<1$ such that, for $a > \overline{a}$, $f_a$ must be compactly supported; for $a < \underline{a}$, $f_a$ must be strictly positive on $\mathbb{R}$. More precisely, 
\[\underline{a} = \frac{400}{848 - 9\pi^2} \approx 0.5269,\quad \overline{a} = \frac{64}{176-9\pi^2}\approx0.7342.\]
\end{corollary}

\begin{proof}
Denote $k = b(f_a)/c(f_a)$ and $\mu=\mu(f_a)$. By Corollary \ref{cor:Ta_property}, $f$ must be strictly positive on $\mathbb{R}$ if $a\leq 0$. Hence, we only need to consider $a\in(0,1]$, in which case we have by Theorem \ref{thm:k_bound} that 
\[\frac{1-a/3}{1+a\overline{\mu}} \leq k = \frac{1-a/3}{1+a\mu} \leq \frac{1-a/3}{1+a\underline{\mu}},\]
where $\overline{\mu} = 9\pi^2/64 - 3/4\geq \mu$ and $\underline{\mu} = \overline{\mu}/(9(1-a/3)^2)\leq \mu$.

According to Corollary \ref{cor:Ta_property}, $f_a = \mtx{R}_a(f_a)$ is compactly supported if and only if $2ak>(1-a/3)$. Therefore, for $f$ to be compactly supported, it suffices for $a>0$ to satisfy
\[ \frac{1-a/3}{2a} < \frac{1-a/3}{1+a\overline{\mu}},\]
that is, 
\[a>  \frac{1}{2-\overline{\mu}} = \frac{64}{176-9\pi^2}=:\overline{a}\approx0.7342.\]
On the other hand, for $f_a$ to be strictly positive on $\mathbb{R}$, it suffices for $a>0$ to satisfy
\[ \frac{1-a/3}{2a} > \frac{1-a/3}{1+a\underline{\mu}},\]
that is, 
\begin{equation}\label{eqt:ac_estimate_midstep}
a<\frac{1}{2-\underline{\mu}} = \frac{1}{2-\frac{\overline{\mu}}{9(1-a/3)^2}}.
\end{equation}
We give a rough estimate on $a$ for \eqref{eqt:ac_estimate_midstep} to hold. Note that any $a\in (0,1/2]$ must satisfy \eqref{eqt:ac_estimate_midstep} since the right-hand side is apparently greater than $1/2$. Thus, we only need to consider $a>1/2$, in which case it suffices for $a$ to satisfy 
\[a<\frac{1}{2-\frac{\overline{\mu}}{9(1-1/6)^2}} = \frac{400}{848 - 9\pi^2}=: \underline{a}\approx 0.5269.\]
The claim is thus proved.
\end{proof}

One can see that if we iterate the argument above (by plugging $\underline{a}$ into the right-hand side of \eqref{eqt:ac_estimate_midstep}), we can obtain a larger value of $\underline{a}$ and thus shorten the uncertain interval $[\underline{a},\overline{a}]$. However, this only improves the value of $\underline{a}$ very slightly, so we omit the effort here.

\subsection{Uniform decay bounds} Corollary \ref{cor:ra_bound} provides estimates of the asymptotic decay rate of a fixed-point solution $f_a = \mtx{R}_a(f_a)$. However, it does not tell whether $f_a(x)$ can be uniformly bounded by $Cx^{-\delta}$ for a range of $a$ and for some uniform constants $C,\delta>0$ (that only depend on the range boundary). One way to achieve this is by uniformly controlling polynomial moments of the form $\int_0^{+\infty}x^pf(x)\idiff x$. We establish this kind of estimates in this subsection, which will be useful when we estimate the support size of a compactly supported $f_a$.\\

We start with a uniform decay bound of the form $Cx^{-1}$. Recall that $b(f_a)<+\infty$ for all $a\leq 1$. More precisely, we have by Theorem \ref{thm:k_bound} that 
\[\int_0^{+\infty}f_a(x)\idiff x  = \frac{\pi}{2}b(f_a) = \frac{\pi}{2}\cdot \frac{1-a/3}{1+a\mu(f_a)}c(f_a)\leq \frac{2(1-a/3)}{1+a\mu(f_a)}\leq \frac{2(1-a/3)}{1+\min\{a\underline{\mu},a\overline{\mu}\}}=:C_0(a).\]
We have used the uniform upper bound of $c(f)$ in \eqref{eqt:cf_bound}, i.e. $c(f)\leq \pi/4$. This implies, for any $a_0\leq 1$ and for all $a\leq [a_0,1]$, 
\begin{equation}\label{eqt:uniform_decay_x}
f_a(x)\leq \frac{1}{x}\int_0^xf_a(y)\idiff y \leq \frac{C_0(a)}{x}\leq \frac{C_0(a_0)}{x}, \quad x\geq 0.\end{equation}
We have used that $f_a(x)$ is decreasing in $x$. In fact, we can do a little better than this for $a\in[0,1]$.

\begin{lemma}\label{lem:uniform_small}
Let $f_a\in \mathbb{D}$ denote a fixed point of $\mtx{R}_a$. Then, for any $\epsilon>0$, there is some constant $X=X_{\epsilon}>0$ (only depending on $\epsilon$) such that $f_a(x)\leq \epsilon x^{-1}$ for all $x\geq X$ and for all $a\in [0,1]$.
\end{lemma}

\begin{proof} We prove this lemma by contradiction. Suppose that the claim is not true. Then, there exists some $\epsilon>0$ and some sequence $\{(a_n,X_n)\}_{n=1}^{+\infty}\subset[0,1]\times [0,+\infty)$ such that 
\[X_nf_n(X_n)> \epsilon\quad \text{for all $n\geq1$,} \quad \text{and}\quad \lim_{n\rightarrow+\infty}X_n=+\infty.\]
Here $f_n = f_{a_n}\in \mathbb{D}$ is a fixed point of $\mtx{R}_{a_n}$. By the closedness and compactness of $D$ in the $L_\rho^\infty$-norm (Lemma \ref{lem:compactness}), there is a sub-sequence of $\{(a_n,X_n)\}_{n=1}^{+\infty}$, still denoted by $\{(a_n,X_n)\}_{n=1}^{+\infty}$, such that 
\[a_n\rightarrow a_*\in[0,1]\quad \text{and}\quad f_n\xrightarrow{L_\rho^\infty} f_*\in \mathbb{D} \quad \text{as $n\rightarrow +\infty$}.\]
Moreover, one can easily modify the proof of Theorem \ref{thm:Ra_continuous} to show that 
\[\mtx{R}_{a_n}(f_n)\xrightarrow{L_\rho^\infty} \mtx{R}_{a_*}(f_*) \quad \text{as $n\rightarrow +\infty$}.\]
Beware that one needs to use the continuity of the function $(1-t)^{1/t}$ around $t=0$ to show the continuity of $(\mtx{T}_a(f))^{1/a}$ in $a$ around $a=0$. We then immediately have that $f_* = \mtx{R}_{a_*}(f_*)$, that is, $f_*$ is a fixed point of $\mtx{R}_{a_*}$ in $\mathbb{D}$. Moreover, we have $c(f_*) = \lim_{n\rightarrow +\infty}c(f_n)$ by Lemma \ref{lem:cf_continuous} and $b(f_*) \leq \liminf_{n\rightarrow +\infty} b(f_n)$ by Fatou's lemma. Let $Q(f)$ be defined as in the proof of Theorem \ref{thm:k_bound}:
\[Q(f) = \frac{1}{\pi^2}\int_0^{+\infty}\int_0^{+\infty}f(x)f(y)\left(\left(\frac{x}{y}+\frac{y}{x}\right)\ln\left|\frac{x+y}{x-y}\right|-2\right)\idiff x\idiff y.\]
Owing to \eqref{eqt:k_bound_step}, we can also use Fatou's lemma to get $Q(f_*) \leq \liminf_{n\rightarrow +\infty} Q(f_n)$. Writing \eqref{eqt:Qf_bf} as 
\[aQ(f_a) = \frac{1}{2}\left(\frac{1-a/3}{2}c(f_a)\right)^2 - \frac{1}{2}\left(b(f_a)-\frac{1-a/3}{2}c(f_a)\right)^2,\]
we then obtain that, for $a\in[0,1]$, 
\[\left(b(f_*)-\frac{1-a/3}{2}c(f_*)\right)^2 \geq \limsup_{n\rightarrow +\infty} \left(b(f_n)-\frac{1-a/3}{2}c(f_n)\right)^2.\]
Also note that $b(f_a) - (1-a/3)c(f_a)/2 >0$ for all $a\leq 1$ (see Lemma \ref{lem:finite_bf}). The inequality above implies $b(f_*)\geq \limsup_{n\rightarrow +\infty} b(f_n)$. Hence, $b(f_*) = \lim_{n\rightarrow +\infty} b(f_n)$.

Since $b(f_*)<+\infty$, there is some $X_*>0$ such that 
\[\int_{X_*}^{+\infty} f_*(x)\idiff x <\frac{\epsilon}{4}.\] 
Since all $f_n$ are bounded by $1$ on $\mathbb{R}$, by the dominated convergence theorem we have 
\begin{align*}
\lim_{n\rightarrow +\infty}\int_{X_*}^{+\infty} f_n(x)\idiff x &= \lim_{n\rightarrow +\infty}\left( \frac{\pi}{2}b(f_n) - \int_0^{X_*}f_n(x)\idiff x\right) \\
&= \frac{\pi}{2}b(f_*) - \int_0^{X_*}f_*(x)\idiff x = \int_{X_*}^{+\infty} f_*(x)\idiff x < \frac{\epsilon}{4}.
\end{align*} 
Therefore, there is some $N\geq 1$ such that, for $n\geq N$, $X_n\geq 2X_*$ and 
\[\int_{X_*}^{+\infty} f_n(x)\idiff x  < \frac{\epsilon}{2},\]
which implies that 
\[X_nf_n(X_n) \leq 2(X_n-X_*)f_n(X_n)\leq 2\int_{X_*}^{X_n} f_n(x)\idiff x \leq 2\int_{X_*}^{+\infty} f_n(x)\idiff x  < \epsilon,\quad n\geq N.\]
However, this contradicts the assumption that $X_nf(X_n)>\epsilon$ for all $n$. The lemma is thus proved.
\end{proof}

Next, we explain how to obtain stronger moment bounds by generalizing the technique used in the proof of \eqref{thm:k_bound}. Let $f\in\mathbb{D}$ be a fixed point of $\mtx{R}_a$ for some $a\leq 1$. For $0\leq p<r_a(f) - 1$, we define
\begin{equation}\label{eqt:bp}
b_p(f) = \frac{2}{\pi}\int_0^{+\infty}x^pf(x)\idiff x.
\end{equation}
Note that $b_0(f) = b(f)$. We then multiply both sides of equation \eqref{eqt:main_equation} by $2x^{p-1}/\pi$ and integrate them over $[0,+\infty)$ (using integration by parts when necessary) to obtain
\begin{equation}\label{eqt:Up_Vp_bp}
c_lpb_p(f) + a\left(U_p(f) + (p-1)V_p(f)\right) = -c_\om b_p(f) - U_p(f),
\end{equation}
where
\[U_p(f):= -\frac{2}{\pi}\int_0^{+\infty}x^{p-1} u'(x)\om(x)\idiff x,\]
and 
\[V_p(f):= - \frac{2}{\pi}\int_0^{+\infty}x^{p-2}u(x)\om(x)\idiff x.\]
Recall that that $\om = -xf$ and $u= -(-\Delta)^{-1}\om$.
Substituting \eqref{eqt:f_to_cl}, \eqref{eqt:f_to_cw} in to \eqref{eqt:Up_Vp_bp} and rearranging the equation yields
\begin{equation}\label{eqt:Up_Vp_bp_simple}
\left((1+ap)b(f)-(1+p)\frac{1-a/3}{2}c(f)\right)b_p(f) = (1+a)U_p(f) + a(p-1)V_p(f).
\end{equation}
Note that $U_0(f) = b(f)^2/2$ and $V_0(f)-U_0(f) = Q(f)$, so \eqref{eqt:Up_Vp_bp_simple} becomes \eqref{eqt:Qf_bf} when $p=0$. Moreover, we can use \eqref{eqt:Up_Vp_bp_simple} to obtain uniform bounds of $b_1(f)$ and $b_2(f)$ as follows.

\begin{lemma}\label{lem:uniform_moment}
Let $f_a\in\mathbb{D}$ be a fixed point of $\mtx{R}_a$ for some $a\leq 1$, and let $b_p(f_a)$ be defined as in \eqref{eqt:bp}. Then, given any $a_1\in(0,1]$, there is some uniform constant $C_1$ only depending on $a_1$ such that, for $a\in[a_1,1]$,  
\[b_1(f_a)\leq C_1.\]
Furthermore, given any $a_2\in[1/2\, ,1]$, there is some uniform constant $C_2$ only depending on $a_2$ such that, for $a\in[a_2,1]$,  
\[b_2(f_a)\leq C_2.\]
\end{lemma}

\begin{proof}
Write $f= f_a$. For $p=1$, we can compute that 
\[U_1 = -\frac{2}{\pi}\int_0^{+\infty}u'(x)\om(x)\idiff x \leq \frac{1}{\pi}\int_0^{+\infty}\left((u'(x))^2 + \om(x)^2\right)\idiff x = \frac{2}{\pi}\int_0^{+\infty}(xf(x))^2\idiff x.\]
We have used that $u' = \mtx{H}(\om)$. By Lemma \ref{lem:uniform_small}, for any $\epsilon>0$, there is some uniform constant $X_{\epsilon}>0$ such that $xf(x) \leq \epsilon$ for $x\geq X_{\epsilon}$. Also, in view of \eqref{eqt:uniform_decay_x}, there is some uniform constant $C_0$ such that $xf(x)\leq C_0$ for all $x\geq 0$ and $a\in [0,1]$. Thus
\[\int_0^{+\infty}(xf(x))^2\idiff x \leq \int_0^{X_\epsilon}x^2f(x)^2\idiff x + \epsilon\int_{X_\epsilon}^{+\infty}xf(x)\idiff x \leq C_0^2X_{\epsilon} + \epsilon\cdot \frac{\pi}{2}b_1(f). \]
We then use \eqref{eqt:k_formula} and \eqref{eqt:Up_Vp_bp_simple} with $p=1$ to obtain 
\[a(1-\mu(f))b(f)b_1(f)= (1+a)U_1\leq (1+a)\left(2C_0^2X_\epsilon/\pi + \epsilon b_1(f)\right).\]
Note that $\mu(f)\leq \overline{\mu}$ and $b(f)\geq b((1-x^2)_+) = 4/3\pi$. It follows that
\[b_1(f)\leq \frac{3\pi(1+a)}{4a(1-\overline{\mu})}\left(2C_0^2X_\epsilon/\pi + \epsilon b_1(f)\right)\leq \frac{3\pi}{2a(1-\overline{\mu})}\left(2C_0^2X_\epsilon/\pi + \epsilon b_1(f)\right).\]
Now, given any $a_1\in(0,1]$, if we choose $\epsilon = a_1(1-\overline{\mu})/3\pi$, then $\epsilon\leq a(1-\overline{\mu})/3\pi$ for all $a\in[a_1,1]$. Hence, we obtain 
\[b_1(f) \leq \frac{6C_0^2X_\epsilon}{a(1-\overline{\mu})} \leq \frac{6C_0^2X_\epsilon}{a_1(1-\overline{\mu})} =: C_1,\]
where $C_1$ only depends on $a_1$.

Next, for $p=2$, we again use \eqref{eqt:k_formula} and \eqref{eqt:Up_Vp_bp_simple} to obtain 
\begin{equation}\label{eqt:U2_V2_b2}
\frac{(4-3\mu(f))a-1}{2}b(f)b_2(f) = (1+a)U_2(f) + aV_2(f).
\end{equation}
Since $x\om(x)$ is an even function of $x$, $\mtx{H}(x\om)(0)=0$. Thus, we can use Lemma \ref{lem:Hilbert_property1} to find that 
\begin{align*}
U_2(f) &= -\frac{2}{\pi}\int_0^{+\infty}xu'(x)\om(x)\idiff x = -\frac{1}{\pi}\int_{\mathbb{R}}x\mtx{H}(\om)(x)\om(x)\idiff x\\
&= -\frac{1}{\pi}\int_{\mathbb{R}}\frac{\mtx{H}(x\om)(x)\cdot (x\om(x))}{x}\idiff x = \frac{1}{2}\left(\mtx{H}(x\om)(0)\right)^2 = 0.
\end{align*}
As for $V_2(f)$, we first use \eqref{eqt:T_to_laplacian} and \eqref{eqt:T_integrate_by_part} to derive that 
\begin{align*}
u(x) = x(\mtx{T}(f)(x) + b(f)) &= \frac{x}{\pi}\int_0^{+\infty}f'(y)y(F(x/y)-2)\idiff y= -\frac{x}{\pi}\int_0^{+\infty}f'(y)yF(y/x)\idiff y\\
&\leq -\frac{2}{\pi}\int_0^{+\infty}f'(y)y^2\idiff y = \frac{4}{\pi}\int_0^{+\infty}f(y)y\idiff y = 2b_1(f).
\end{align*}
We have used that $f'(x)\leq 0$ for all $x\geq0$, $F(t) = 2-F(1/t)$, and $F(t)\leq 2t$ for all $t\geq 0$ ($F(t)\leq t^2\leq t$ for $t\in[0,1]$ and $F(t)\leq 2$ for $t>1$; see Appendix \ref{sec:F}). Hence, we have
\[V_2(f) = - \frac{2}{\pi}\int_0^{+\infty}u(x)\om(x)\idiff x \leq 2b_1(f) \cdot \frac{2}{\pi}\int_0^{+\infty}xf(x)\idiff x = 2b_1(f)^2.\]
Then, substituting these estimates into \eqref{eqt:U2_V2_b2} and using $\mu(f)\leq \overline{\mu}$ and $b(f)\geq 4/3\pi$ yields
\[b_2(f) \leq \frac{3\pi}{2\left((4-3\overline{\mu})a-1\right)}\cdot aV_2(f) \leq \frac{3\pi b_1(f)^2}{(4-3\overline{\mu})a-1},\quad \text{for $a>1/(4-3\overline{\mu})$}.\]
Note that $\overline{\mu} = 9\pi^2/64-3/4<2/3$, and thus the inequalities above hold for $a\geq1/2 > 1/(4-3\overline{\mu})$. By the first statement of Lemma \ref{lem:uniform_moment}, for $a\in [a_2,1]$, there is some uniform constant $C_{1,2}$ depending only on $a_2$ such that $b_1(f)\leq C_{1,2}$. Hence, we further obtain 
\[b_2(f) \leq \frac{3\pi C_{1,2}^2}{(4-3\overline{\mu})a_2-1}=:C_2,\]
where $C_2$ only depends on $a_2$.
\end{proof}

For a compactly supported fixed-point solution $f_a$, we will use the uniform bound of $b_2(f_a)$ in Lemma \ref{lem:uniform_moment} to derive estimates of the support size of $f_a$ in the next subsection. 

\subsection{Asymptotic behavior}
Our next goal is to give accurate characterization of the asymptotic behavior of a fixed point $f=\mtx{R}_a(f)$ as $x\rightarrow +\infty$. Below, we first control the tail behavior of $\mtx{T}(f)$ based on the decay rate of $f$.

\begin{lemma}\label{lem:critical_decay}
Given $f\in \mathbb{D}$, if $C_\delta := \sup_{x\in \mathbb{R}}\ x^{1+\delta}f(x) <+\infty$ for some $\delta\in(0,2)$, then 
\[\sup_{x\in \mathbb{R}}\ x^{\delta}(\mtx{T}(f)(x)-\mtx{T}(f)(+\infty))\lesssim C_\delta.\]
If $\lim_{x\rightarrow+\infty}x^3f(x)=0$, then 
\[\lim_{x\rightarrow+\infty}x^2\big(\mtx{T}(f)(x)-\mtx{T}(f)(+\infty)\big) = \frac{2}{\pi}\int_0^{+\infty}y^2f(y)\idiff y =: b_2(f).\] 
Moreover, if $b_2(f) <+\infty$, then
\[\mtx{T}(f)(x)-\mtx{T}(f)(+\infty)\leq \frac{3b_2(f)}{x^2}\quad \text{and}\quad |\mtx{T}(f)'(x)| \leq \frac{10b_2(f)}{|x|^3}.\]
\end{lemma}

\begin{proof}
For any $x>0$, we calculate that 
\begin{align*}
\mtx{T}(f)(x) - \mtx{T}(f)(+\infty) &= \mtx{T}(f)(x) + b(f) = \frac{1}{\pi}\int_0^{+\infty}f(y)\cdot \frac{y}{x}\ln\left|\frac{x+y}{x-y}\right|\idiff y\\
&\leq \frac{C_\delta}{\pi}\cdot\frac{1}{x}\int_0^{+\infty}\frac{1}{y^\delta}\ln\left|\frac{x+y}{x-y}\right|\idiff y = \frac{C_\delta}{\pi}\cdot\frac{1}{x^\delta}\int_0^{+\infty}\frac{1}{t^\delta}\ln\left|\frac{t+1}{t-1}\right|\idiff t\lesssim \frac{C_\delta}{x^\delta}.
\end{align*}
We have used the fact that the non-negative function $\frac{1}{t^\delta}\ln\left|\frac{t+1}{t-1}\right|$ is integrable on $[0,+\infty)$ for any $\delta\in(0,2)$. This proves the first claim.

As for the second claim, we compute that 
\begin{align*}
x^2\big(\mtx{T}(f)(x)-\mtx{T}(f)(+\infty)\big) &= \frac{1}{\pi}\int_0^{+\infty}f(y)xy\ln\left|\frac{x+y}{x-y}\right|\idiff y \\
&= \frac{2}{\pi}\int_0^xf(y)y^2\idiff y + \frac{1}{\pi}\int_0^{+\infty}f(y)\left(xy\ln\left|\frac{x+y}{x-y}\right|-2y^2\chi_{\{y\leq x\}}\right)\idiff y \\
&= \frac{2}{\pi}\int_0^xf(y)y^2\idiff y + \frac{1}{\pi}\int_0^{+\infty}f(tx)(tx)^3\left(\frac{1}{t^2}\ln\left|\frac{t+1}{t-1}\right|-\frac{2}{t}\chi_{\{t\leq1\}}\right)\idiff t.
\end{align*}
Since $\lim_{x\rightarrow+\infty}x^3f(x)=0$, the function $f(tx)(tx)^3$ is uniformly bounded for all $x,t\geq0$, and $\lim_{x\rightarrow +\infty}f(tx)(tx)^3 = 0$ for any $t>0$. We also note that the function $\frac{1}{t^2}\ln\left|\frac{t+1}{t-1}\right|-\frac{2}{t}\chi_{\{t\leq1\}}$ is absolutely integrable on $[0,+\infty)$. By the dominated convergence theorem, we have
\[\lim_{x\rightarrow+\infty}\frac{1}{\pi}\int_0^{+\infty}f(tx)(tx)^3\left(\frac{1}{t^2}\ln\left|\frac{t+1}{t-1}\right|-\frac{2}{t}\chi_{\{t\leq1\}}\right)\idiff t = 0.\]
Therefore, 
\[\lim_{x\rightarrow+\infty}x^2\big(\mtx{T}(f)(x)-\mtx{T}(f)(+\infty)\big) = \lim_{x\rightarrow+\infty}\frac{2}{\pi}\int_0^xf(y)y^2\idiff y = \frac{2}{\pi}\int_0^{+\infty}f(y)y^2\idiff y,\]
which is also valid even when the last integral is infinite.

Finally, we prove the third claim. We use \eqref{eqt:T_integrate_by_part} to derive that 
\begin{align*}
\mtx{T}(f)(x) - \mtx{T}(f)(+\infty) &= \frac{1}{\pi}\int_0^{+\infty}f'(y)y(F(x/y)-2)\idiff y= -\frac{1}{\pi}\int_0^{+\infty}f'(y)yF(y/x)\idiff y\\
&\leq -\frac{2}{\pi}\int_0^{+\infty}f'(y)\frac{y^3}{x^2}\idiff y = \frac{6}{\pi x^2}\int_0^{+\infty}f(y)y^2\idiff y = \frac{3b_2(f)}{x^2}.
\end{align*}
We have used that $f'(x)\leq 0$ for all $x\geq0$, $F(t) = 2-F(1/t)$, and $F(t)\leq 2t^2$ for all $t\geq 0$ (see Appendix \ref{sec:F}). Moreover, we apply \eqref{eqt:T_derivative} to get, for $x>0$,
\begin{align*}
|\mtx{T}(f)'(x)| &= -\frac{1}{\pi}\int_0^{+\infty}f'(y)F'(x/y)\idiff y\\
&=\frac{1}{\pi}\int_0^{+\infty}\left(\frac{f'(y)}{y}\right)'\left(\int_0^ysF'(x/s)\idiff s\right) \idiff y\\
&=\frac{x^2}{\pi}\int_0^{+\infty}\left(\frac{f'(y)}{y}\right)' \cdot \frac{y}{x}\left(\frac{4}{3}-G(x/y)\right) \idiff y\\
&\leq \frac{4}{3\pi x^3}\int_0^{+\infty}\left(\frac{f'(y)}{y}\right)' \cdot y^5 \idiff y\\
&= \frac{20}{\pi x^3}\int_0^{+\infty}f(y)y^2 \idiff y = \frac{10b_2(f)}{x^3}.
\end{align*}
We have used that $(f'(x)/x)'\geq 0$ for all $x\geq 0$, $(4t/3-tG(1/t))'=tF'(1/t)$, and $4/3-G(1/t)\leq 4t^4/3$ for all $t\geq 0$ (see Appendix \ref{sec:G}). This completes the proof.
\end{proof}

We can now classify and characterize the tail behavior of $f$ depending on the relation between the parameter $a$ and the ratio $b(f)/c(f)$.

\begin{theorem}\label{thm:solution_type}
Let $f\in \mathbb{D}$ be a fixed point of $\mtx{R}_a$ for some $a\leq 1$. Denote $k = b(f)/c(f)$. Then, one of the following happens:
\begin{enumerate}
\item $2ak>1-a/3$ (must happen when $a>\overline{a}$): $f$ is compactly supported on $[-L_f,L_f]$, where $L_f := \sup\{x:f(x)>0\}$ satisfies
\[ \bar{C}\left(k - \frac{1-a/3}{2a}\right)^{-1/2}\leq  L_f \leq \tilde{C}\left(k - \frac{1-a/3}{2a}\right)^{-1/2}\]
for some absolute constants $\bar{C},\tilde{C}>0$. Moreover, there is some finite number $C_{a,f}>0$ such that 
\[\lim_{x\rightarrow L_f-}\ \frac{f(x)}{(L_f-x)^{p_a(f)}} = C_{a,f},\]
where 
\[p_a(f) = \frac{1}{a} + \frac{1-a}{a}\cdot\frac{(1-a/3)c(f)}{L_f\cdot 2a|\mtx{T}(f)'(L_f)|} \geq \max\left\{\frac{3-a}{2a}\ , \frac{1}{a} + \frac{1-a}{a}\cdot \hat{C} L_f^2\right\}\]
for some absolute constant $\hat{C}>0$.

\item $2ak=1-a/3$ (can only happen when $\underline{a}\leq a\leq \overline{a}$): $f$ is strictly positive on $\mathbb{R}$, and there is some finite number $C_{a,f}>0$ such that 
\[\lim_{x\rightarrow+\infty}\frac{\ln f(x)}{x^2} = -C_{a,f}.\]
\item $2ak<1-a/3$ (must happen when $a<\underline{a}$): $f$ is strictly positive on $\mathbb{R}$, and there is some finite number $C_{a,f}>0$ such that 
\[\lim_{x\rightarrow +\infty} x^{r_a(f)}f(x) = C_{a,f},\]
where 
\[r_a(f) =  \frac{2(1-a)}{(1-a/3)/k-2a} \begin{cases}
\displaystyle\ >2, & a\in(0,1),  \\ 
\displaystyle\ =2, & a= 0,  \\ 
\displaystyle\ \in(1,2), & a <0.  \\ 
\end{cases}\] 
\end{enumerate}
\end{theorem}

\begin{proof}
(1): Note that this case can only happen for $a>0$. Write $g = \mtx{T}_a(f)$ and $L = L_f$. We note that $f(x)=0$ if and only if $g(x)=0$. By the definition of $L$, 
\[0 = g(L) = 1 + \frac{2a\cdot \mtx{T}(f)(L)}{(1-a/3)c(f)},\]
implying that  
\[\mtx{T}(f)(L)+b(f) = b(f) - \frac{1-a/3}{2a}c(f) = c(f)\left(k - \frac{1-a/3}{2a}\right).\]
We have by Lemma \ref{lem:cf_bound} $c(f)\leq 4/\pi$ and by Theorem \ref{thm:k_bound}
\[c(f) = \frac{1+a\mu(f)}{1-a/3}b(f) \geq b(f) \geq b((1-x^2)_+) = \frac{4}{3\pi}.\] 
We thus need to upper and lower bound $\mtx{T}(f)(L)+b(f)$ in terms of $L$. Recall that this case can only happen for $a>\overline{a}>1/2$ (see Corollary \ref{cor:ac_estimate}). Hence, by Lemma \ref{lem:uniform_moment}, there is some absolute constant $C_0$ such that $b_2(f)\leq C_0$ for all $a>\overline{a}$. We then use the third result in Lemma \ref{lem:critical_decay} to show that 
\[\mtx{T}(f)(L)+b(f) \leq \frac{3b_2(f)}{L^2} \leq \frac{3C_0}{L^2}.\]
On the other hand, for any $f\in \mathbb{D}$, one has
\begin{align*}
\mtx{T}(f)(L)+b(f) &= \frac{1}{\pi}\int_0^{+\infty}f(y)\cdot \frac{y}{L}\ln\left|\frac{L+y}{L-y}\right|\idiff y\\
&\geq \frac{1}{\pi}\int_0^{+\infty}f_m(y)\cdot \frac{y}{L}\ln\left|\frac{L+y}{L-y}\right|\idiff y = \mtx{T}(f_m)(L)+b(f_m)\geq \frac{4}{15\pi}\cdot \frac{1}{L^2},
\end{align*}
where $f_m = (1-x^2)_+$. The last inequality above owes to the estimates of $f_m$ in Appendix \ref{sec:f_m}. Combing these estimates yields
\[\frac{9\pi C_0}{4}\left(k - \frac{1-a/3}{2a}\right)^{-1} \geq L^2\geq \frac{1}{15}\left(k - \frac{1-a/3}{2a}\right)^{-1}.\]

In order to prove the second claim, we first show that $\mtx{T}(f)\in C^{1,1/2}(\mathbb{R})$. In fact, one can easily show that the function $F$ defined in \eqref{eqt:F_definition} satisfies
\[\frac{|F'(t)-F'(s)|}{|t-s|^{1/2}}\leq C\left(\frac{1}{|t-1|^{1/2}}+\frac{1}{|s-1|^{1/2}}\right)\]
for some absolute constant $C>0$. We then use formula \eqref{eqt:T_derivative} to obtain 
\begin{align*}
\frac{|\mtx{T}(f)'(x) - \mtx{T}(f)'(z)|}{|x-z|^{1/2}}&\leq \frac{1}{\pi}\int_0^{+\infty}(-f'(y))\cdot \frac{|F'(x/y)-F'(z/y)|}{|x-z|^{1/2}}\idiff y\\
&\leq  \frac{C}{\pi}\int_0^{+\infty}(-f'(y))\cdot \left(\frac{1}{|x-y|^{1/2}}+\frac{1}{|z-y|^{1/2}}\right)\idiff y\\
&\leq \frac{2C}{\pi}\int_0^{L}\left(\frac{1}{|x-y|^{1/2}}+\frac{1}{|z-y|^{1/2}}\right)\idiff y\\
&\leq \tilde{C}.
\end{align*}
This also implies that $g=\mtx{T}_a(f)\in C^{1,1/2}([-L,L])$. Note that we only need to consider $a\in(0,1]$, in which case the left derivative of $g$ at $L$ satisfies 
\[g'_{-}(L) = \frac{2a}{(1-a/3)c(f)}\cdot \mtx{T}(f)'(L) = \frac{2a}{(1-a/3)c(f)}\cdot \frac{1}{\pi}\int_0^{L}f'(y)\cdot F'(L/y)\idiff y<0. \]
Also, we can use the convexity of $g(\sqrt{s})$ in $s$ to find that, for $x\in[0,L)$, 
\[g(x) = -(L^2-x^2)\cdot \frac{g(x)-g(L)}{x^2-L^2} \geq -(L^2-x^2)\cdot\frac{g'_{-}(L)}{2L} \gtrsim (L-x)|g'_{-}(L)|.\]
It then follows that 
\[\frac{|g(x)-(x-L)g'_{-}(L)|}{(L-x)g(x)}\lesssim \frac{|g(x)-(x-L)g'_{-}(L)|}{(L-x)^2|g'_{-}(L)|} = \frac{|g'(\tilde{x})-g'_{-}(L)|}{(L-x)|g'_{-}(L)|} \lesssim \frac{1}{(L-x)^{1/2}},\quad x< \tilde{x}<L.\]
We have used Lagrange's mean value theorem and the result that $g\in C^{1,1/2}([-L,L])$. 

Now, if we choose 
\[p_a(f) = \frac{1}{a} + \frac{1-a}{a}\cdot\frac{(1-a/3)c(f)}{L\cdot 2a|\mtx{T}(f)'(L)|} = \frac{1}{a} - \frac{1-a}{a}\cdot\frac{1}{L g'_{-}(L)},\]
then we have 
\begin{align*}
\frac{f(x)}{(L-x)^{p_a(f)}} &= \frac{f(x)}{(L-x)^{1/a}}\cdot \exp\left(\frac{1-a}{a}\cdot\frac{1}{Lg'_{-}(L)}\int_{L-1}^x\frac{1}{y-L}\idiff y\right)\\
&= \left(\frac{g(x)}{(L-x)}\right)^{1/a}\cdot\exp\left(\frac{1-a}{a}\int_0^x\frac{g(y)-1}{yg(y)}\idiff y + \frac{1-a}{a}\int_{L-1}^x\frac{1}{L(y-L)g'_{-}(L)}\idiff y\right)\\
&=\left(\frac{g(x)}{(L-x)}\right)^{1/a}\cdot\exp\left(\frac{1-a}{a}\int_0^{L-1}\frac{g(y)-1}{yg(y)}\idiff y + \frac{1-a}{a}I(x)\right),
\end{align*}
where 
\begin{align*}
I(x) := \int_{L-1}^x\frac{1}{y}\idiff y + \int_{L-1}^x\frac{1}{Lyg'_{-}(L)}\idiff y + \int_{L-1}^x\frac{g(y)-(y-L)g'_{-}(L)}{y(y-L)g(y)g'_{-}(L)}\idiff y.
\end{align*}
Note that we must have $L>1$. Hence, by the preceding estimates, the limit $I(L) :=\lim_{x\rightarrow L-}I(x)$ exists and is finite. Therefore, we have the limit 
\begin{align*}
\lim_{x\rightarrow L-} \frac{f(x)}{(L-x)^{p_a(f)}} &= \lim_{x\rightarrow L-}\left(\frac{g(x)}{(L-x)}\right)^{1/a}\cdot\exp\left(\frac{1-a}{a}\int_0^{L-1}\frac{g(y)-1}{yg(y)}\idiff y + \frac{1-a}{a}\cdot I(x)\right)\\
&= |g'_{-}(L)|^{1/a}\cdot\exp\left(\frac{1-a}{a}\int_0^{L-1}\frac{g(y)-1}{yg(y)}\idiff y + \frac{1-a}{a}\cdot I(L)\right)\\
&=: C_{a,f}\,,
\end{align*}
which is finite and strictly positive. Finally, by the convexity of $g(\sqrt{s})$ in $s$, we derive that 
\[\frac{g'_{-}(L)}{2L}\geq \frac{g(L)-g(0)}{L^2} = -\frac{1}{L^2},\]
which implies 
\[p_a(f) = \frac{1}{a} - \frac{1-a}{a}\cdot\frac{1}{L g'_{-}(L)}\geq \frac{1}{a} + \frac{1-a}{2a} = \frac{3-a}{2a}.\]
Moreover, by Lemma \ref{lem:uniform_moment} and the third result of Lemma \ref{lem:critical_decay}, we have
\[|g'_{-}(L)| = \frac{2a}{(1-a/3)c(f)}\cdot |\mtx{T}(f)'(L)| \leq \frac{2a}{(1-a/3)c(f)} \cdot \frac{10b_2(f)}{L^3} \leq \frac{45\pi C_0}{2L^3}.\]
Hence, 
\[p_a(f) = \frac{1}{a} + \frac{1-a}{a}\cdot\frac{1}{L |g'_{-}(L)|} \geq \frac{1}{a} + \frac{1-a}{a}\cdot \frac{2L^2}{45\pi C_0}.\]

(2): Write $g = \mtx{T}_a(f)$. In this case, $g(+\infty) = 1 - 2ak/(1-a/3)=0$, and thus 
\[g(x) = \frac{2a}{(1-a/3)c(f)}\cdot\big(\mtx{T}(f)(x) - \mtx{T}(f)(+\infty)\big) = \frac{\mtx{T}(f)(x) - \mtx{T}(f)(+\infty)}{b(f)}.\]
By Lemma \ref{lem:Ra_decay}, $\lim_{x\rightarrow+\infty}x^\delta f(x)=0$ for any $\delta>0$. In particular, $\lim_{x\rightarrow+\infty}x^3 f(x)=0$. It then follows from the second claim in Lemma \ref{lem:critical_decay} that 
\[b(f)\cdot \lim_{x\rightarrow+\infty}x^2g(x) = \frac{2}{\pi}\int_0^{+\infty}y^2f(y)\idiff y = b_2(f) <+\infty.\]
We can now use L'Hopital's rule to compute that 
\[\lim_{x\rightarrow+\infty}\frac{\int_0^{x}\frac{g(y)-g(0)}{yg(y)}\idiff y}{x^2} = \lim_{x\rightarrow+\infty}\frac{g(x)-1}{2x^2g(x)} = -\frac{b(f)}{2b_2(f)}.\]
Therefore, 
\[\lim_{x\rightarrow+\infty}\frac{\ln f(x)}{x^2} = \lim_{x\rightarrow+\infty}\left(\frac{1}{a}\frac{\ln g(x)}{x^2} + \frac{1-a}{a}\frac{\int_0^{x}\frac{g(y)-g(0)}{yg(y)}\idiff y}{x^2}\right) = -\frac{1-a}{2a}\cdot\frac{b(f)}{b_2(f)} =: -C_{a,f}.\]

(3): Write $g = \mtx{T}_a(f)$. In this case, $g(+\infty) = 1 - 2ak/(1-a/3) > 0$. We can compute that,  
\begin{align*}
\frac{1-a}{a}\int_0^{x}\frac{g(y)-1}{yg(y)}\idiff y &= \frac{1-a}{a}\left(\int_0^1\frac{g(y)-1}{yg(y)}\idiff y + \int_1^x\frac{g(y)-g(+\infty)}{yg(y)g(+\infty)}\idiff y + \int_1^x\frac{g(+\infty)-1}{yg(+\infty)}\idiff y \right)\\
&= \frac{1-a}{a}\left(\int_0^{1}\frac{g(y)-1}{yg(y)}\idiff y + \int_1^{x}\frac{g(y)-g(+\infty)}{yg(y)g(+\infty)}\idiff y\right) - r_a(f)\ln x .
\end{align*}
By Lemma \ref{lem:Ra_decay} and Corollary \ref{cor:ra_bound}, there exists some $\delta\in(0,1)$ such that $\lim_{x\rightarrow +\infty}x^{1+\delta}f(x) = 0$. Using Lemma \ref{lem:critical_decay}, we find that (in spite of the sign of $a$)
\[
0\leq \frac{1-a}{a}\cdot \frac{1}{x}\left(\frac{g(x)-g(+\infty)}{g(x)g(+\infty)}\right) = \frac{1-a}{(1-a/3)c(f)}\cdot\frac{1}{x}\left(\frac{\mtx{T}(f)(x)-\mtx{T}(f)(+\infty)}{g(x)g(+\infty)}\right)\leq \frac{\tilde{C}_{a,f}}{x^{1+\delta}},
\]
for some finite constant $\tilde{C}_{a,f}$ that only depends on $a$ and $f$. This implies that 
\[0\leq \frac{1-a}{a}\int_1^{+\infty}\frac{g(y)-g(+\infty)}{yg(y)g(+\infty)}\idiff y<+\infty.\]
Therefore, 
\begin{align*}
\lim_{x\rightarrow +\infty} x^{r_a(f)}f(x) &= \lim_{x\rightarrow +\infty} x^{r_a(f)}g(x)^{1/a}\exp\left(\frac{1-a}{a}\int_0^x\frac{g(y)-1}{yg(y)}\idiff y \right)\\
&=g(+\infty)^{1/a}\exp\left\{\frac{1-a}{a}\left(\int_0^1\frac{g(y)-1}{yg(y)}\idiff y + \int_1^{+\infty}\frac{g(y)-g(+\infty)}{yg(y)g(+\infty)}\right)\right\}\\
&=:C_{a,f}<+\infty.
\end{align*}

Note that the special case $a=0$ belongs to the scenario $2ak<1-a/3$. In this case, we have $k = 1$, i.e. $-\mtx{T}(f)(+\infty) = b(f)=c(f)$, and $r_0(f) = 2$. Using the special formula \eqref{eqt:R_0} for $\mtx{R}_0$, we again find that
\begin{align*}
\lim_{x\rightarrow +\infty} x^{r_0(f)}f(x) &= \lim_{x\rightarrow +\infty} x^2\exp\left(\frac{2}{c(f)}\left(\mtx{T}(f)(x)+\int_0^x\frac{\mtx{T}(f)(y)}{y}\idiff y\right)\right)\\
&= \lim_{x\rightarrow +\infty} \exp\left(\frac{2b(f)}{c(f)}\int_1^x\frac{1}{y}\idiff y\right)\cdot \exp\left(\frac{2}{c(f)}\left(\mtx{T}(f)(x)+\int_0^x\frac{\mtx{T}(f)(y)}{y}\idiff y\right)\right)\\
&=\lim_{x\rightarrow +\infty} \exp\left(\frac{2}{c(f)}\left(\mtx{T}(f)(x)+\int_0^1\frac{\mtx{T}(f)(y)}{y}\idiff y + \int_1^x\frac{\mtx{T}(f)(y)-\mtx{T}(f)(+\infty)}{y}\right) \right)\\
&=\econst^{-2}\exp\left(\frac{2}{c(f)}\left(\int_0^1\frac{\mtx{T}(f)(y)}{y}\idiff y + \int_1^{+\infty}\frac{\mtx{T}(f)(y)-\mtx{T}(f)(+\infty)}{y}\right) \right)\\
&=:C_{0,f}<+\infty.
\end{align*}
The proof is thus completed.
\end{proof}

As a brief summary, Theorem \ref{thm:solution_type} states that: $(1)$ when $2ak>1-a/3$, $f$ is compactly supported on $[-L_f,L_f]$ and $f(x)\sim (L_f-x)^{p_a(f)}$ as $x\rightarrow L_f-$, where the degeneracy order $p_a(f)\gtrsim 1/a + (1-a)L_f^2/a$; $(2)$ when $2ak=1-a/3$, $f(x)$ decays exponentially fast as $x\rightarrow +\infty$, and so does $\om = -xf$; $(3)$ when $2ak<1-a/3$, $f(x)\sim x^{-r_a(f)}$ as $x\rightarrow +\infty$, which means that $\om(x) = -xf(x)\sim x^{c_\om/c_l}$ in view of \eqref{eqt:cw_cl}.

\subsection{Regularity}
In this subsection, we study the regularity of a solution $\om = -xf$ with $f\in \mathbb{D}$ being a fixed point of $\mtx{R}_a$ for some $a\leq 1$. We shall always denote $g = \mtx{T}_a(f)$ in the sequel. Recall that 
\[ f(x) = \mtx{R}_a(f)(x) = g(x)^{1/a}\exp\left(\frac{1-a}{a}\int_0^{+\infty}\frac{g(y)-1}{yg(y)}\idiff y\right),\]
and 
\begin{equation}\label{eqt:f_derivative}
f'(x) = \left(\frac{1}{a}g'(x)+ \frac{1-a}{a}\frac{g(x)-1}{x}\right)\frac{f(x)}{g(x)}.
\end{equation}
Since $|f'(x)|\leq 2\min\{x,x^{-1}\}$ (see the proof of Lemma \ref{lem:compactness}), it is not hard to check by formula \eqref{eqt:T_derivative} that $\mtx{T}(f)\in C^1(\mathbb{R})$, and thus $g\in C^1([-L_f,L_f])$, where 
\[L_f = \sup\{\, x\, :\, g(x)>0\} = \sup\{\, x\, :\, f(x)>0\}.\]
Clearly $L_f\geq 1$ since $f(x)\geq (1-x^2)_+$, and $L_f=+\infty$ if $f$ is strictly positive on $\mathbb{R}$. Note that $f/g\in C([-L_f,L_f])$. We then obtain from \eqref{eqt:f_derivative} that $f\in C^1([-L_f,L_f])$. Moreover, when $a<1$, $f(x)/g(x)\leq g(x)^{1/a-1}\rightarrow 0$ as $x\rightarrow L$, and thus $f\in C^1(\mathbb{R})$. On the other hand, $f'$ has a step jump at $x=L_f$ when $a=1$, so we only have $f\in C(\mathbb{R})$. These regularity properties all easily pass onto $\om = -xf$, that is, $\om\in C^1([-L_f,L_f])$ for all $a\leq 1$, and $\om\in C^1(\mathbb{R})$ only if $a<1$. 

Moreover, when $2ak=1-a/3$ with $k=b(f)/c(f)$, it is easy to show by Theorem \ref{thm:solution_type} and formula \eqref{eqt:f_derivative} that $f(x)\lesssim \econst^{-\delta x^2}$ and $f'(x)\lesssim \econst^{-\delta x^2}$ for some $\delta >0$. Therefore,
\[\|\om\|_{H^1(\mathbb{R})} \leq \|xf\|_{L^2(\mathbb{R})} + \|f\|_{L^2(\mathbb{R})} + \|xf'\|_{L^2(\mathbb{R})}<+\infty.\]

When $2ak<1-a/3$, we always have $0<\min\{g(0),g(+\infty)\}\leq g(x)\leq \max\{g(0),g(+\infty)\}$. By the convexity of $\sgn(a)\cdot g(\sqrt{s})$ in $s$, 
\[\frac{|g'(x)|}{g(x)}\leq \frac{2|g(0)-g(x)|}{xg(x)}\leq \frac{1}{x}\cdot\frac{2|1-g(+\infty)|}{\min\{1,g(+\infty)\}}\lesssim \frac{a}{x},\quad x\geq 1.\]
It then follows from \eqref{eqt:f_derivative} that $|f'|\lesssim f/x$. By Corollary \ref{cor:ra_bound}, if $a\geq 0$, then $f\lesssim x^{-2}$ and $|f'|\lesssim x^{-3}$, which again yields $\|\om\|_{H^1(\mathbb{R})}<+\infty$. Otherwise, if $a< 0$, then $f\lesssim x^{-1}$ and $|f'|\lesssim x^{-2}$, so it is only guaranteed that $\|\om\|_{\dot{H}^1(\mathbb{R})}<+\infty$.

To obtain higher regularity of $f$ or $\om$, we need to make use of the compactness of the map $\mtx{T}$ as described in the next lemma.

\begin{lemma}\label{lem:regularity_induction}
Given $f\in \mathbb{D}$, suppose that $b(f)<+\infty$. If $f\in H^p_{loc}(-L\, ,L)$ for some $L>0$ and some integer $p\geq 0$, then $\mtx{T}(f)'\in H^p_{loc}(-L\, ,L)$. In particular, if $f\in H^p(\mathbb{R})$ for some integer $p\geq 0$, then $\mtx{T}(f)'\in H^p(\mathbb{R})$.
\end{lemma} 

\begin{proof}
In view of \eqref{eqt:T_to_laplacian}, we have 
\begin{equation}\label{eqt:Tf_Hf}
(x\mtx{T}(f))' = -\mtx{H}(xf) - b(f) = \mtx{H}(xf)(0)-\mtx{H}(xf) = x\cdot \frac{\mtx{H}(xf)(0)-\mtx{H}(xf)}{x} = -x\mtx{H}(f).
\end{equation}
We have used Lemma \ref{lem:Hilbert_property1} for the last identity above. It follows that
\[\mtx{T}(f)(x) = -\frac{1}{x}\int_0^xy\mtx{H}(f)(y)\idiff y,\]
and thus
\[
\mtx{T}(f)'(x) = - \mtx{H}(f)(x) + \frac{1}{x^2}\int_0^xy\mtx{H}(f)(y)\idiff y = - \mtx{H}(f)(x) + \int_0^1t\mtx{H}(f)(tx)\idiff t.
\]
Then, for any integer $p\geq0$, we have
\[\mtx{T}(f)^{(p+1)}(x) = - \mtx{H}(f)^{(p)}(x) + \int_0^1t^{p+1}\mtx{H}(f)^{(p)}(tx)\idiff t,\]
which easily implies that 
\[\|\mtx{T}(f)\|_{\dot{H}^{p+1}([-L,L])} \leq C_p\|\mtx{H}(f)\|_{\dot{H}^p([-L,L])} \]
for any $L>0$ ($L$ can be $+\infty$) and for some constant $C_p$ that only depends on $p$.

By Lemma \ref{lem:Hilbert_property2}, if $f\in H^p_{loc}(-L\, ,L)$ for some $L>0$ and some integer $p\geq 0$, then $\mtx{H}(f)\in H^p_{loc}(-L\, ,L)$, which further implies that $\mtx{T}(f)'\in H^p_{loc}(-L\, ,L)$. Moreover, if $f\in H^p(\mathbb{R})$, then by the well-known identity $\|\mtx{H}(f)\|_{H^p(\mathbb{R})} = \|f\|_{H^p(\mathbb{R})}$ we know $\mtx{H}(f)\in H^p(\mathbb{R})$, and thus $\mtx{T}(f)'\in H^p(\mathbb{R})$.
\end{proof}

We can now use Lemma \ref{lem:regularity_induction} to prove that all fixed-point solutions are interiorly smooth.

\begin{theorem}\label{thm:regularity}
Let $f\in \mathbb{D}$ be a fixed point of $\mtx{R}_a$ for some $a\leq 1$. Denote $k = b(f)/c(f)$ and $L = \sup\{\, x\, :\, f(x)>0\}$. Then, one of the following happens:
\begin{enumerate}
\item $2ak>1-a/3$: $f$ is compactly supported on $[-L, L]$, and $f$ is smooth in the interior of $(-L,L)$. 
\item $2ak\leq 1-a/3$: $f$ is strictly positive on $\mathbb{R}$, and $f, (xf)' \in H^p(\mathbb{R})$ for any $p\geq 0$.
\end{enumerate}
\end{theorem}

\begin{proof}
$(1)$: In view of \eqref{eqt:Tf_Hf}, we can write \eqref{eqt:f_derivative} on $[-L,L]$ as
\begin{equation}\label{eqt:f_derivative_2}
\begin{split}
f'(x) &= \frac{2}{(1-a/3)c(f)}\cdot\left(\mtx{T}(f)'(x) + (1-a)\frac{\mtx{T}(f)(x)}{x}\right)\frac{f(x)}{g(x)}\\
&= \frac{2}{(1-a/3)c(f)}\cdot\big(a\mtx{T}(f)'(x) - (1-a)\mtx{H}(f)(x)\big)\cdot\frac{f(x)}{g(x)}
\end{split}
\end{equation}
For any $0<L'<L$, $g(x)\geq g(L')>0$ (since now $a>0$) for $x\in[-L',L']$. It follows straightforwardly from Lemma \ref{lem:regularity_induction} and Lemma \ref{lem:Hilbert_property2} that if $f\in H^p_{loc}(-L',L')$ for some integer $p\geq 0$, then $f\in H^{p+1}_{loc}(-L',L')$. The proof is routine, so we omit the details here. Since $f \in C^1([-L,L])\subset H^1([-L',L'])$, we immediately obtain by recursion that $f\in H^p_{loc}(-L',L')$ for all $p\geq 0$. This further implies that $f\in H^p_{loc}(-L,L)$ for all $p\geq 0$ since $L'<L$ is arbitrary, and thus $f$ is smooth in the interior of $(-L,L)$. 

$(2)$: If $2ak=1-a/3$, then $g(x)$ is decreasing in $x$ and $g(x)\sim 1/x^2$ as $x\rightarrow +\infty$ (see the proof of part $(2)$ of Theorem \ref{thm:solution_type}). By the exponential decay property of $f$ in this case (part $(2)$ of Theorem \ref{thm:solution_type}), we have that $\|f/g^p\|_{L^{\infty}(\mathbb{R})}<+\infty$ for all $p\geq 0$. Instead, if $2ak<1-a/3$, then $g(x)\geq \min\{g(0),g(+\infty)\}>0$, and we still have $\|f/g^p\|_{L^\infty(\mathbb{R})}<+\infty$ for all $p\geq 0$. In either case, we can use \eqref{eqt:f_derivative_2} and Lemma \ref{lem:regularity_induction} to prove that $f\in H^p(\mathbb{R})$ implies $f\in H^{p+1}(\mathbb{R})$ for all $p\geq 0$. We then use $f\in L^2(\mathbb{R})$ to recursively show that $f \in H^p(\mathbb{R})$ for any $p\geq 0$.

Moreover, since $\|xf\|_{L^\infty(\mathbb{R})}<+\infty$, we have $\|xf/g^p\|_{L^\infty(\mathbb{R})}<+\infty$ for all $p\geq 0$ in all cases. Hence, we can use \eqref{eqt:f_derivative_2} and the fact $f\in H^p(\mathbb{R})$ for any $p\geq 0$ to show that $xf^{(p+1)}\in L^2(\mathbb{R})$ for any $p\geq 0$. Therefore, $(xf)' \in H^p(\mathbb{R})$ for any $p\geq 0$.
\end{proof}

\vspace{2mm}

We finish this section with a proof of Theorem \ref{thm:main} based on the preceding results.

\begin{proof}[Proof of Theorem \ref{thm:main}]
The existence of solutions $(\om,c_l,c_\om)$ of \eqref{eqt:main_equation} for all $a\leq 1$ follows from Theorem \ref{thm:existence_fixed_point}. By Lemma \ref{lem:finite_bf}, we have $c_\om<0$. In view of the scaling property \eqref{eqt:scaling}, we can always re-scale the solution (by only altering $\alpha$) so that $c_\om = -1$. Note that the re-scaling factor $\alpha$ in \eqref{eqt:scaling} can be uniformly bounded for $a$ in a bounded range, and the ratio $b(f)/c(f)$ is invariant under such re-scaling. Then, the estimate \eqref{eqt:cl_formula} of $c_l$ results from Corollary \ref{cor:clcw_ratio}. The regularity properties and the decaying features of $\om$ follow from Theorem \ref{thm:solution_type} and Theorem \ref{thm:regularity}, respectively. Note that $\sgn(c_l) = \sgn(1-a/3-2ak)$ with $k=b(f)/c(f)$, so that the three cases in Theorem \ref{thm:main} one-to-one correspond to the three cases in Theorem \ref{thm:solution_type} in sequence. The algebraic decay of $f$ in Theorem \ref{thm:solution_type} case $(3)$ transfers to the algebraic decay of $\om=-xf$ in Theorem \ref{thm:main} case $(3)$ via the relation \eqref{eqt:cw_cl}. Finally, the values of $\underline{a}$ and $\overline{a}$ are obtained in Corollary \ref{cor:ac_estimate}.
\end{proof}

\section{A review of existing results}\label{sec:review}

As mentioned in the introduction, self-similar finite-time blowup solutions of the gCLM model with interiorly smooth profiles have been found for some particular values of $a$. In particular, these self-similar profiles (i.e., solutions of the self-similar profile equation \eqref{eqt:main_equation}) are all odd functions of $x$ and non-positive on $[0,+\infty)$, so that each of them corresponds to a fixed point of $\mtx{R}_a$. In this section, we will help the reader review these profile solutions $\om$, and we verify that their corresponding fixed-point solutions $f=-\om/x$ all belong to the set $\mathbb{D}$ and satisfy the properties proved in previous sections. We will also discuss some other existing solutions of \eqref{eqt:main_equation} that are beyond our fixed-point family.

\subsection{Solution for $a=1$} When $a=1$, the gCLM model becomes the De Gregorio model \cite{de1996partial}. It is shown in \cite{huang2023self} that the corresponding self-similar profile equation 
\begin{equation}\label{eqt:DG_profile_equation}
(c_lx+u)\om_x  = (c_\om+u_x)\om,\quad u_x = \mtx{H}(\om),\quad u(0)=0,
\end{equation}
admits infinitely many solutions $(\om,c_l,c_\om)$ such that $\om$ is compactly supported on $[-1,1]$ (by re-scaling) and $c_l = c_{\om} = -1$. These solutions are distinct under re-scaling and re-normalization, and they all correspond to eigen-functions of a self-adjoint, compact operator $\mtx{M}$ over a linear space $\mathbb{W}$:
\[\mtx{M}(\om) = \chi_{[-1,1]}\left((-\Delta)^{-1/2}\om - (-\Delta)^{-1/2}\om(1)\cdot x\right),\quad \om\in \mathbb{W},\]
where 
\[\mathbb{W} = \{\om: \om(-x) = -\om(x),\ \om\in H^1_0([-1,1])\}.\] 
One can immediately relate this linear operator $\mtx{M}$ to our map $\mtx{R}_1$. 

In particular, it is proved in \cite{huang2023self} that the leading eigen-function of $\mtx{M}$, denoted by $\om_*$, is the unique (up to a multiplicative constant) solution of \eqref{eqt:DG_profile_equation} that is strictly negative on $(0,1)$. This profile was first found and proved to be non-linearly stable in \cite{chen2021finite}. In this paper, we have shown that \eqref{eqt:DG_profile_equation} admits a solution $(\om_1,c_l,c_\om)$ with $-\om_1/x\in \mathbb{D}$ being a fixed point of $\mathbb{R}_1$. By the uniqueness of $\om_*$, we can conclude that $\om_1$ coincides with $\om_*$ under re-normalization. This means that $f_1=-\om_1/x$ is the unique fixed point of $\mtx{R}_1$ in $\mathbb{D}$. This also means that the function $-\om_*/x$ satisfies all the scaling-invariant properties we have proved for a fixed point of $\mtx{R}_1$. In fact, it is proved in \cite{huang2023self} that $\om_*$ is smooth in the interior of its support and $-\om_*/x$ is decreasing in $x$ on $[0,1]$, which is consistent with our results.

\subsection{Solution for $a=1/2$} An analytic solution of \eqref{eqt:main_equation} for $a=1/2$ was first found by Chen \cite{chen2020singularity} and Lushnikov et al. \cite{lushnikov2021collapse} independently, which is given by the explicit expressions
\[\om_{1/2}(x) = -\frac{4x}{(2+x^2)^2},\quad \mtx{H}(\om_{1/2})(x) = \frac{\sqrt{2}(2-x^2)}{(2+x^2)^2},\quad c_l = \frac{\sqrt{2}}{16},\quad c_\om = -\frac{3\sqrt{2}}{16}.\]
We have normalized $\om_{1/2}$ so that 
\[f_{1/2}(x) = -\frac{\om_{1/2}(x)}{x} = \frac{4}{(2+x^2)^2}\]
satisfies $f_{1/2}(0)=1$ and $f_{1/2}'(x)/2x|_{x=0} = -1$. The corresponding $c_l,c_\om$ are computed using \eqref{eqt:f_to_cl} and \eqref{eqt:f_to_cw}, respectively. Note that the ratio $c_\om/c_l=-3$ is invariant under re-scaling as in \eqref{eqt:scaling}.

It is not difficult to check that $f_{1/2}$ belongs to $\mathbb{D}$, so that $f_{1/2}$ is a fixed point of $\mtx{R}_{1/2}$, i.e. $f_{1/2} = \mtx{R}_{1/2}(f_{1/2})$. Moreover, $f_{1/2}$ is smooth and strictly positive on $\mathbb{R}$, and $f_{1/2}$ decays algebraically as $x\rightarrow +\infty$. Note that $1/2<\underline{a}$, so that $a=1/2$ falls in the case $2ak<1-a/3$ by Corollary \ref{cor:ac_estimate}. Hence, these observations are consistent with Theorem \ref{thm:solution_type} part (3) and Theorem \ref{thm:solution_type} part (2). In particular, the explicit expression of $f_{1/2}$ implies $f_{1/2}(x)\sim x^{-4}$ as $x\rightarrow +\infty$, meaning that $r_{1/2}(f_{1/2}) = 4$. This is consistent with $c_\om/c_l = -3$ in view of \eqref{eqt:cw_cl}.

\subsection{Solution for $a=0$} When $a=0$, the gCLM model reduces to the original Constantin--Lax--Majda model \cite{constantin1985simple}, whose self-similar profile equation writes 
\begin{equation}\label{eqt:CLM_profile_equation}
c_lx\om_x  = (c_\om+u_x)\om,\quad u_x = \mtx{H}(\om),\quad u(0)=0.
\end{equation}
A closed-form solution of \eqref{eqt:CLM_profile_equation} was first given in \cite{elgindi2020effects} (from a different formulation) as
\[\om_0(x) = -\frac{x}{1+x^2},\quad \mtx{H}(\om_0)(x) = \frac{1}{1+x^2},\quad c_l = \frac{1}{2},\quad c_\om=-\frac{1}{2}.\]
Again, we have normalized $\om_0$ so that 
\[f_0(x) = -\frac{\om_0(x)}{x} = \frac{1}{1+x^2}\]
satisfies $f_0(0)=1$ and $f_0'(x)/2x|_{x=0} = -1$. The corresponding $c_l,c_\om$ are computed using \eqref{eqt:f_to_cl} and \eqref{eqt:f_to_cw}, respectively, with the ratio $c_\om/c_l=-1$ invariant under re-scaling of $\om_0$.

Similar to $f_{1/2}$ in the preceding case, this $f_0$ is also verified to be a fixed point of $\mtx{R}_0$ in $\mathbb{D}$. In consistence with $a=0<\underline{a}$, $f_0$ also satisfies all the general properties we have established for the category $2ak<1-a/3$. In particular, $r_0(f_0) = 2$, exactly verifying the claim in Corollary \ref{cor:ra_bound}. \\

We remark that this profile $\om_0$ can be obtained in an elementary way based on the Tricomi identity for the Hilbert transform and a complex argument, as conducted by Elgindi and Jeong in \cite{elgindi2020effects}. In fact, this fashion of complex argument was used in the early work \cite{constantin1985simple} to find explicit finite-time blowup solutions of the Constantin--Lax--Majda model. In addition to the case $a=0$, Lushnikov et al. \cite{lushnikov2021collapse} also employed this complex method to construct $\om_{1/2}$ for $a=1/2$ in a consistent way. Here, we reformulate their methods in terms of $f = -\om/x$ (rather than $\om$) to illustrate the main idea.

Let $(\om,c_l,c_\om)$ be a solution of \eqref{eqt:CLM_profile_equation}. Assume that $\om(x)$ is an odd function of $x$ and that $c_l>0$. Due to the scaling-invariant property \eqref{eqt:scaling}, we may assume that $\om'(0) = -1$ and $c_l=1/2$ (by re-normalization). Define $f = -\om/x$, so that $f(x)$ is even in $x$ and $f(0)=1$. According to \eqref{eqt:cl_relation}, we have $c_l = c_\om + u_x(0)$. Also, by Lemma \ref{lem:Hilbert_property1}, we have 
\[\frac{u'(x) - u'(0)}{x} = \frac{\mtx{H}(\om)(x) - \mtx{H}(\om)(0)}{x} = \mtx{H}\left(\frac{\om}{x}\right)(x) = -\mtx{H}(f)(x).\]
Substituting all these into \eqref{eqt:CLM_profile_equation} yields  
\[f' = -2f\mtx{H}(f).\]
Computing the Hilbert transform of both sides of this identity and using Tricomi's identity, we reach
\[\mtx{H}(f)' = \mtx{H}(f') = -2\mtx{H}\big(f\mtx{H}(f)\big) = f^2 - \mtx{H}(f)^2.\]
Write $g = \mtx{H}(f)$ (not to be confused with the notion $g$ used in the previous sections). Since $f(x)$ is even in $x$, $g(x)$ is odd in $x$ so $g(0)=0$. We then arrive at initial value problem 
\[f'(x) = -2g(x)f(x),\quad g'(x) = f(x)^2-g(x)^2,\quad f(0)=1,\quad g(0)=0.\]
One then finds that 
\[\left(\frac{f(x)}{f(x)^2+g(x)^2}\right)' = 0\quad\Longrightarrow \quad \frac{f(x)}{f(x)^2+g(x)^2} \equiv \frac{f(0)}{f(0)^2+g(0)^2}=1,\]
and that 
\[\left(\frac{g(x)}{f(x)}\right)' = \frac{f(x)^2+g(x)^2}{f(x)} = 1\quad\Longrightarrow \quad \frac{g(x)}{f(x)} = x.\]
It then easily follows that $f(x) = 1/(1+x^2)$, which is exactly equal to $f_0 = -\om_0/x$. An even simpler idea is to consider the complex-valued function $h = f + \iunit g$ and observe that 
\[h'(x) = ih(x)^2,\quad h(0) = 1.\]
This initial value problem of $h$ has a unique solution $h(x) = (1+\iunit x)/(1+x^2)$, which again leads to $f = \text{Re}(h) = 1/(1+x^2)$. 

From the calculations above, we also see that $\om_0$ is the unique solution of \eqref{eqt:CLM_profile_equation} (up to re-scaling) with $c_\om<0$ and $c_l>0$. That is, $f_0 = 1/(1+x^2)$ is the unique fixed point of $\mtx{R}_0$ in $\mathbb{D}$.

\subsection{Solutions beyond the fixed-point family} We have constructed regular solutions of the self-similar profile equation \eqref{eqt:main_equation} from fixed points of $\mtx{R}_a$ for all $a\leq 1$. However, we have not been able to prove uniqueness of fixed points of $\mtx{R}_a$ in $\mathbb{D}$ for general values of $a$ (except for $a=0$ and $a=1$), though we conjecture that such uniqueness is true.

Besides, solutions of \eqref{eqt:main_equation} beyond our fixed-point family have already been found, though with lower regularity. Elgindi and Jeong \cite{elgindi2020effects} have constructed $C^\alpha$ profiles (solutions of \eqref{eqt:main_equation}) for $\alpha \in \{1/n: n\in \mathbb{N}\}$ and $|a|<\epsilon_0/\alpha$ with some small uniform constant $\epsilon_0$. Recently, Zheng \cite{zheng2023exactly} improved on this result by releasing the restriction $\alpha \in \{1/n: n\in \mathbb{N}\}$ but still requiring that $|a|\alpha\ll 1$. Note that $\alpha$ needs to be small if $a$ is not close to $0$. We remark that their solutions are essentially $C^\alpha$ near the point $x=0$ and are smooth away from $x=0$. On the contrary, our solutions constructed from fixed points of $\mtx{R}_a$ are smooth away from their support boundaries (if any). Moreover, their solutions have a heavy tail $x^{-\alpha}$ in the far-field, while our profiles $\om=-xf$ have much faster algebraic decays as described in Theorem \ref{thm:main} or Theorem \ref{thm:solution_type}.

Another family of solutions of \eqref{eqt:main_equation} for a wide range of $a$ was constructed by Castro \cite{martinez2010nonlinear} with the closed-form 
\[
\om_{c}(x) = -\chi_{[-1,1]}\frac{x}{\sqrt{1-x^2}}, \quad 
\mtx{H}(\om_{c})(x) = 
\begin{cases}
\displaystyle\ 1 - \frac{x}{\sqrt{x^2-1}}, & x>1,\\
\displaystyle\ 1, & x\in(-1,1),\\
\displaystyle\ 1 + \frac{x}{\sqrt{x^2-1}}, & x<-1,
\end{cases}
\quad c_l = -a, \quad c_\om=-1.
\]
Interestingly, $(\om,c_l,c_\om) = (\om_{c},-a,-1)$ is a universal solution of \eqref{eqt:main_equation} for all values of $a$ with the same formula for $\om_{c}$ under the normalization conditions $\om_c'(0)=-1, c_\om = -1$. Apparently, $f_c = -\om_c/x = \chi_{[-1,1]}/\sqrt{1-x^2}$ does not belong to $\mathbb{D}$. Though $f_c$ is smooth in the interior of its support, it is unbounded and has an infinite $L^2$-norm due to its singularity at $x=\pm1$.\\

\section{Numerical simulations}\label{sec:numerical}
In \cite{lushnikov2021collapse}, Lushnikov et al. performed direct numerical simulations of the gCLM model \eqref{eqt:gCLM} and found evidence of self-similar finite-time blowup from smooth initial data for a wide range of the parameter $a$. In particular, they dynamically re-scaled the time-dependent solution of \eqref{eqt:gCLM} to obtain numerically convergent self-similar profiles. They observed that there seems to be a critical value $a_c\approx 0.6891$ such that the profile converges to a compactly supported function when $a>a_c$, while it converges to a function strictly negative on $(0,+\infty)$ when $a<a_c$. This observation is consistent with our theoretical results, though we only give a rough estimate of this critical value as $a_c\in(\underline{a},\overline{a})\approx (0.5269, 0.7342)$ (see Corollary \ref{cor:ac_estimate}). 

Moreover, Lushnikov et al. also considered the self-similar profile equation \eqref{eqt:main_equation} and converted it into a nonlinear eigenvalue relation. They then obtained approximate solutions by numerically solving this nonlinear eigenvalue problem. However, they did not know how the ratio $c_l/c_\om$ depends on the solution $\om$ a priori, which brought them additional difficulty as they had to iterate the value of $c_l/c_\om$ while they solved the nonlinear eigenvalue problem.

An alternative way to obtain self-similar profiles of the gCLM model is by introducing time-dependence into the profile equation \eqref{eqt:main_equation} and solving the initial-value problem
\begin{equation}\label{eqt:dynamic_rescaling}
\om_t + (c_l(t)x+au)\om_x  = (c_\om(t)+u_x)\om,\quad u_x = \mtx{H}(\om),\quad u(0)=0,
\end{equation}
with some suitable initial data $\om(x,0)$. One needs to impose two time-independent normalization conditions on $\om(x,t)$ so that $c_l(t),c_\om(t)$ can be uniquely determined by the solution $\om(x,t)$. In fact, equation \eqref{eqt:dynamic_rescaling} (usually referred to as the dynamically re-scaling equation of \eqref{eqt:gCLM}) is equivalent to the gCLM model \eqref{eqt:gCLM} under some dynamic change of variables. See e.g. \cite{chen2021finite,huang2023self} for details of the equivalent transformation between the two equations. Apparently, if the solution $(\om(x,t),c_l(t),c_\om(t)$ of \eqref{eqt:dynamic_rescaling} converges as $t\rightarrow +\infty$, then the equilibrium $(\om(x),c_l,c_\om)$ is a solution of the self-similar profile equation \eqref{eqt:main_equation}. Chen et al. \cite{chen2021finite} obtained an approximate self-similar profile of the De Gregorio model ($a=1$) by numerically solving the dynamically re-scaling equation \eqref{eqt:dynamic_rescaling} with $a=1$. They then used computer-assisted proof based on this approximate self-similar profile to show that the De Gregorio model will blow up in finite time from smooth initial data. \\

Different from the methods in \cite{lushnikov2021collapse} and \cite{chen2021finite}, we obtain approximate solutions of the self-similar profile equation \eqref{eqt:main_equation} for any $a\leq 1$ by numerically solving the fixed-point problem $f = \mtx{R}_a(f)$ using a direct iteration method. That is, starting with some smooth initial function $f^{(0)}\in\mathbb{D}$, we compute 
\begin{equation}\label{eqt:scheme}
f^{(n+1)} = \mtx{R}_a(f^{(n)}),\quad n=0, 1, 2,\dots.
\end{equation}
We have not been able to prove the convergence of this iterative method. Nevertheless, this scheme converges quickly for any $a\leq 1$, with the maximum residual $\|f^{(n)}-\mtx{R}_a(f^{(n)})\|_{L^\infty}$ dropped below a very small tolerance (set to be $10^{-7}$ in our computations) only in a few iterations ($25$ at most). Empirically, our method is much more efficient than numerically solving the time-dependent gCLM model \eqref{eqt:gCLM} or its dynamically re-scaling equation \eqref{eqt:dynamic_rescaling}. 

Note that the scheme \eqref{eqt:scheme} keeps $f^{(n)}\in \mathbb{D}$ theoretically, but we need to re-normalize the solution $f^{(n)}$ in every step so that $f^{(n)}(0)=1$ and $\lim_{x\rightarrow 0}(f^{(n)})'(x)/2x = -1$ due to numerical errors. The initial function $f^{(0)}$ does not need to be chosen carefully or specifically for each value of $a$. In fact, we can simply use $f^{(0)}=1/(1+x^2)$ (which is the unique fixed point of $\mtx{R}_0$) for all values of $a$. Even if the initial function $f^{(0)}$ is strictly positive on $\mathbb{R}$, the solution will converge to some compactly supported fixed point for $a>a_c\approx 0.6891$. A more efficient way to obtain fixed-point solutions of $f_a=\mtx{R}_a(f_a)$ for a series of values of $a$ is by employing the idea of the continuation method. That is, after we obtain a numerically convergent fixed point $f_a$ for $\mtx{R}_a$, we use it as the initial guess $f^{(0)}$ in the scheme \eqref{eqt:scheme} for $a\pm \varepsilon$ with some small step size $\varepsilon>0$. 

We present below some numerical results that verify and visualize the preceding theoretical results on the fixed point solutions $f_a=\mtx{R}_a(f_a)$ for $a\leq 1$. We also provide some numerical evidence to support our conjecture on the behavior of $f_a$ as $a$ changes.

Figure \ref{fig:f_a}(a) plots the numerically obtained fixed points $f_a=\mtx{R}_a(f_a)$ for a series of values of $a$. As we can see, for each value of $a$, $f_a(x)$ is decreasing in $x$ and is lower bounded by $(1-x^2)_+$. It is also visually verified in Figure \ref{fig:f_a}(b) that each $f_a(\sqrt{s})$ is convex in $s$. Moreover, these plots support our conjecture that, for any $a_1\leq a_2\leq 1$, $f_{a_1}(x)\geq f_{a_2}(x)$ for all $x$. Figure \ref{fig:T_a} plots the corresponding numerically computed $\mtx{T}_a(f_a)$. We can see that for each value of $a$, $\sgn(a)\cdot\mtx{T}_a(f_a)(x)$ is decreasing in $x$ and $\sgn(a)\cdot\mtx{T}_a(f_a)(\sqrt{s})$ is convex in $s$, which again is consistent with our analysis results.

Figure \ref{fig:bf_cf}(a) plots $b(f_a)$ and $c(f_a)$ as functions of $a$ by connecting numerical data points obtained for different value of $a$. It suggests that both $b(f_a)$ and $c(f_a)$ are continuous in $a$, supporting our conjecture that $f_a$ depends continuously on $a$. Figure \ref{fig:bf_cf}(b) compares $2ak$ against $1-a/3$ where $k=b(f_a)/c(f_a)$, showing that the two solid lines cross at a unique critical value $a=a_c$. To the right of $a_c$, $2ak>1-a/3$ and $f_a$ is compactly supported; to the left of $a_c$, $2ak<1-a/3$ and $f_a$ is strictly positive on $\mathbb{R}$. 

Figure \ref{fig:r_a}(a) plots $r_a(f_a)$ as a function of $a$ for $a< a_c$. We observe that $r_a(f_a)$ is increasing in $a$, with $r_0(f_0) = 2$ and $\lim_{a\rightarrow -\infty}r_a(f_a) = 1$. As conjectured, the numerical fitting of $r_a(f_a)$ climbs to $+\infty$ as $a\rightarrow a_c-$, implying the transition from non-compactly supported solutions to compactly supported ones when $a$ crosses $a_c$. Figure \ref{fig:r_a}(b) plots the curves of $x^{r_a(f_a)}f_a(x)$ for a few values of $a<a_c$, demonstrating that they all converge to some constants as $x\rightarrow +\infty$, which is consistent with Theorem \ref{thm:solution_type} part $(3)$.

Figure \ref{fig:mu_a}(a) plots $c_l/|c_\om| = -c_l/c_\om$ as a function of $a$, which is consistent the estimates in Corollary \ref{cor:clcw_ratio}. One can see in this figure how the decay rate of $\om = -xf_a$ (as given in Theorem \ref{thm:main} part $(3)$) continuously depends on $a$. Figure \ref{fig:mu_a}(b) plots $\mu(f_a)$ as a function of $a$, visualizing the estimates of $\mu(f_a)$ in Theorem \ref{thm:k_bound}.  

\vspace{10mm} 

\begin{figure}[!ht]
\centering
    \begin{subfigure}[b]{0.49\textwidth}
        \includegraphics[width=1\textwidth]{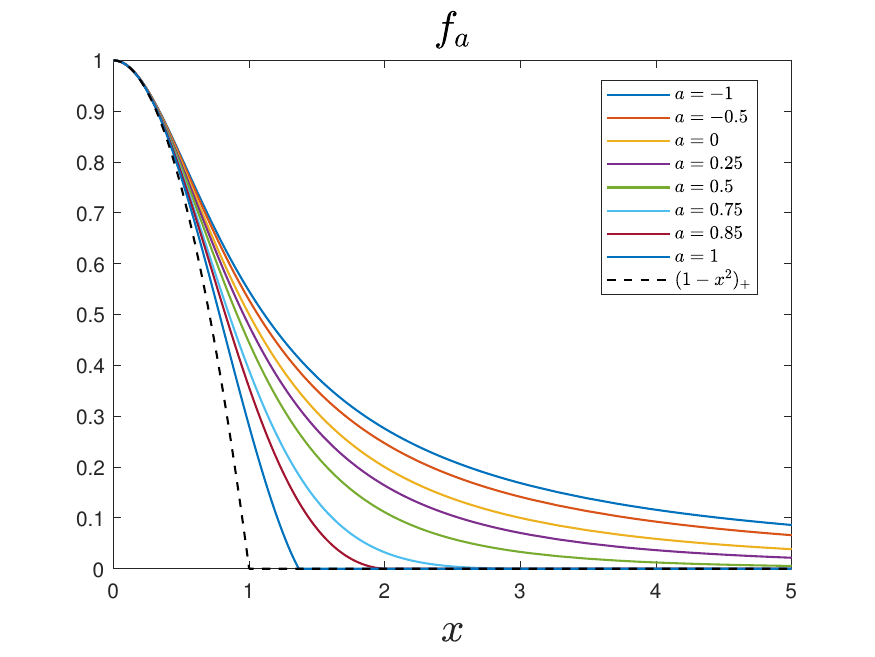}
        \caption{$f_a(x)$}
    \end{subfigure}
    \begin{subfigure}[b]{0.49\textwidth}
        \includegraphics[width=1\textwidth]{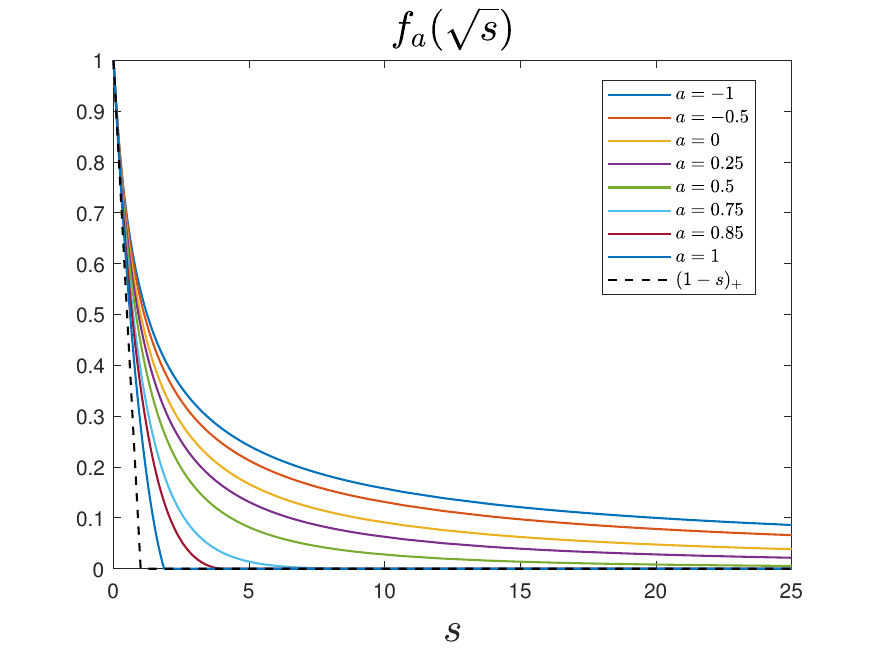}
        \caption{$f_a(\sqrt{s})$}
    \end{subfigure}
    \caption[f_a]{Numerically obtained fixed-point solutions $f_a(x)$ plotted in (a) coordinate $x$ and (b) coordinate $s=x^2$. The dashed line represents the lower bound $(1-x^2)_+=(1-s)_+$ for functions in $\mathbb{D}$.}
    \label{fig:f_a}
\end{figure}

\newpage

\begin{figure}[!ht]
\centering
    \begin{subfigure}[b]{0.49\textwidth}
        \includegraphics[width=1\textwidth]{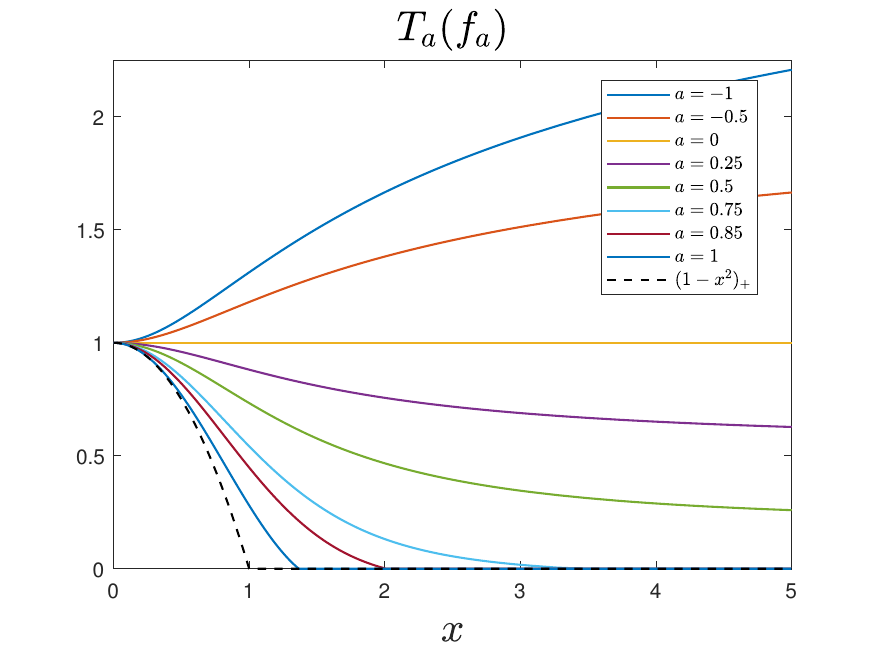}
        \caption{$\mtx{T}_a(f_a)(x)$}
    \end{subfigure}
    \begin{subfigure}[b]{0.49\textwidth}
        \includegraphics[width=1\textwidth]{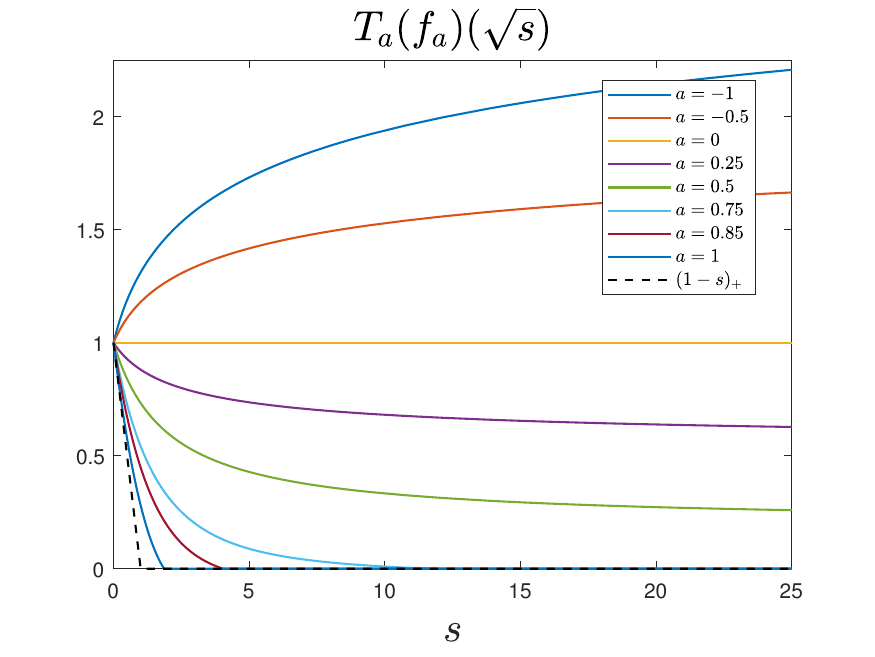}
        \caption{$\mtx{T}_a(f_a)(\sqrt{s})$}
    \end{subfigure}
    \caption[T_a]{$\mtx{T}_a(f_a)(x)$ plotted in (a) coordinate $x$ and (b) coordinate $s=x^2$. The dashed line represents the lower bound $(1-x^2)_+=(1-s)_+$.}
    \label{fig:T_a}
\end{figure}

\vspace{10mm}

\begin{figure}[!ht]
\centering
    \begin{subfigure}[b]{0.49\textwidth}
        \includegraphics[width=1\textwidth]{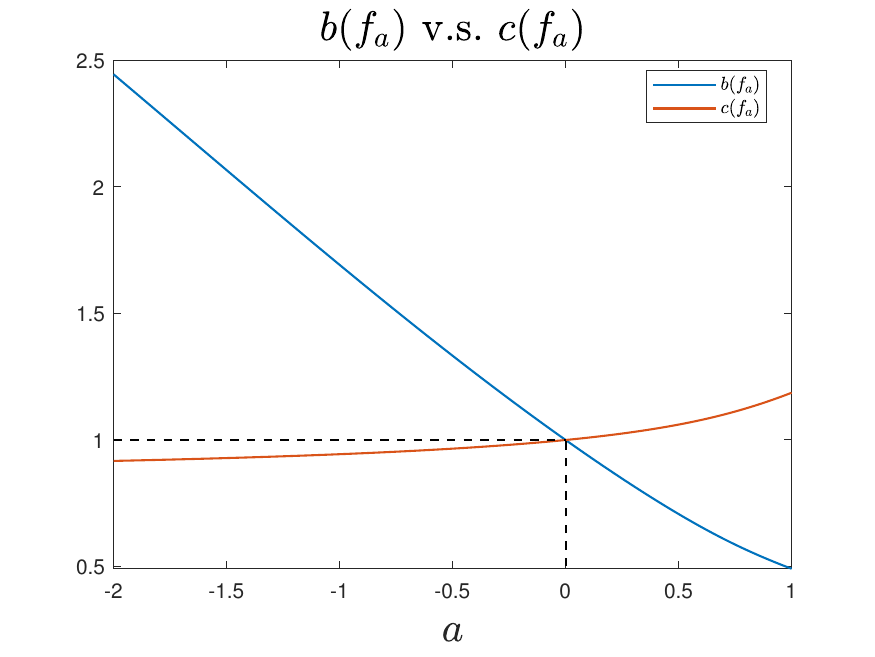}
        \caption{$b(f_a)$ v.s. $c(f_a)$}
    \end{subfigure}
    \begin{subfigure}[b]{0.49\textwidth}
        \includegraphics[width=1\textwidth]{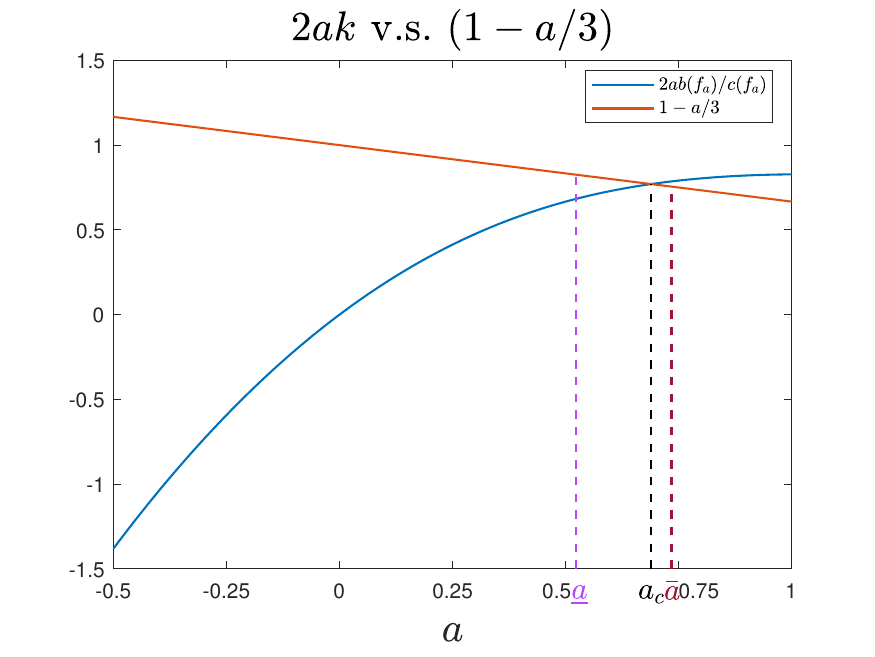}
        \caption{$2ak$ v.s. $1-a/3$}
    \end{subfigure}
    \caption[bf_cf]{(a) $b(f_a)$ and $c(f_a)$ as functions of $a$. (b) Comparison between $2ak$ and $1-a/3$ where $k=b(f_a)/c(f_a)$; the three vertical dashed lines represent the numerical estimate of the critical value $a_c$ and its theoretical upper and lower bounds, respectively. All solid lines are plotted by connecting data points obtained for different values of $a$.}
    \label{fig:bf_cf}
\end{figure}

\newpage

\begin{figure}[!ht]
\centering
    \begin{subfigure}[b]{0.49\textwidth}
        \includegraphics[width=1\textwidth]{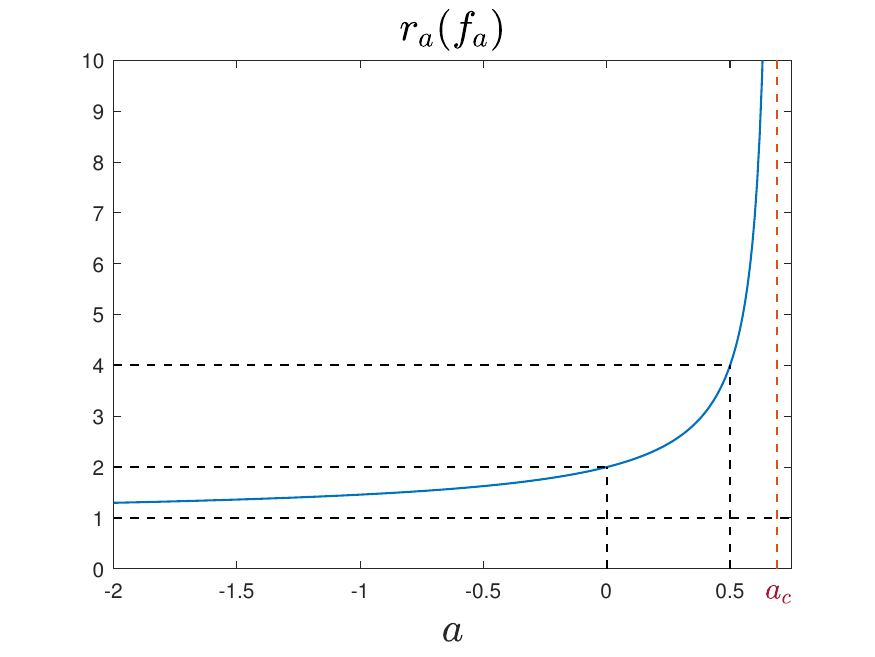}
        \caption{$r_a(f_a)$}
    \end{subfigure}
    \begin{subfigure}[b]{0.49\textwidth}
        \includegraphics[width=1\textwidth]{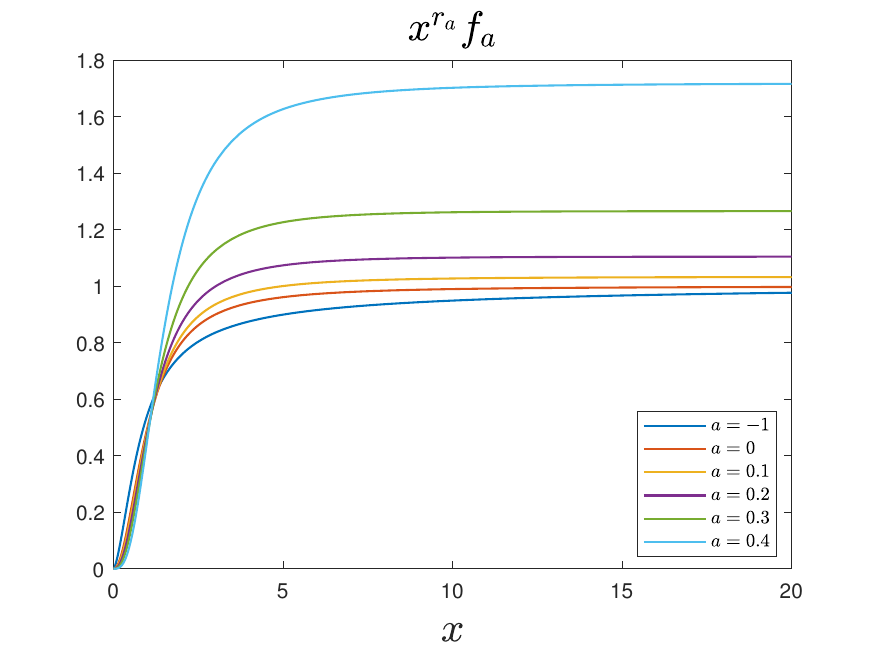}
        \caption{$x^{r_a(f_a)}f_a(x)$}
    \end{subfigure}
    \caption[r_a]{(a) $r_a(f_a)$ as a function of $a$ for $a<a_c$. (b) Curves of $x^{r_a}f_a(x)$ for some values of $a<a_c$. All solid lines are plotted by connecting data points obtained for different values of $a$.}
    \label{fig:r_a}
\end{figure}

\vspace{10mm}

\begin{figure}[!ht]
\centering
    \begin{subfigure}[b]{0.49\textwidth}
        \includegraphics[width=1\textwidth]{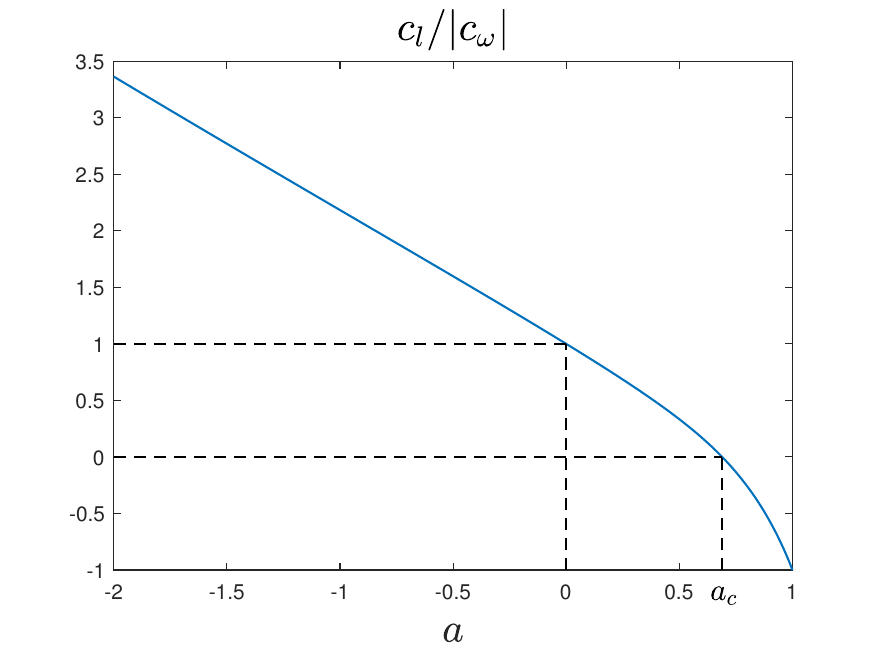}
        \caption{$c_l/|c_\om|$}
    \end{subfigure}
    \begin{subfigure}[b]{0.49\textwidth}
        \includegraphics[width=1\textwidth]{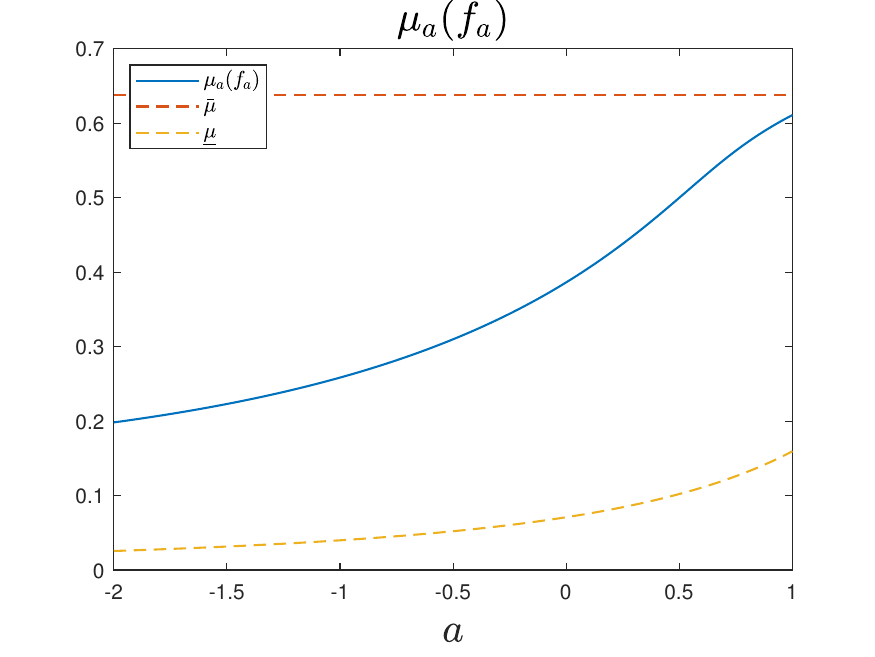}
        \caption{$\mu(f_a)$}
    \end{subfigure}
    \caption[mu_a]{(a) $c_l/|c_\om|$ as a function of $a$. (b) $\mu(f_a)$ as a function of $a$. All solid lines are plotted by connecting data points obtained for different values of $a$.}
    \label{fig:mu_a}
\end{figure}

\newpage

\appendix

\section{Special functions}\label{sec:special_functions}

\subsection{Special function $F$}\label{sec:F}

We define
\begin{equation}\label{eqt:F_definition}
F(t):= \frac{t^2-1}{2t}\ln\left|\frac{t+1}{t-1}\right| + 1, \quad t\geq 0.
\end{equation}
The derivative of $F$ reads
\[F'(t) = \frac{t^2+1}{2t^2}\ln\left|\frac{t+1}{t-1}\right| - \frac{1}{t}.\]
For $t\in[0,1)$, $F(t)$ and $F'(t)$ have the Taylor expansions 
\[F(t) = \suml_{n=1}^\infty\frac{2t^{2n}}{4n^2-1},\quad F'(t) = \suml_{n=1}^\infty\frac{4nt^{2n-1}}{4n^2-1}.\]
For $t\in[0,1)$, $F(1/t)$ and $F'(1/t)$ have the Taylor expansions 
\[F(1/t) = 2 - \suml_{n=1}^\infty\frac{2t^{2n}}{4n^2-1},\quad F'(1/t) = \suml_{n=1}^\infty\frac{4nt^{2n+1}}{4n^2-1}.\]

\begin{lemma}\label{lem:F_property} The function $F$ defined in \eqref{eqt:F_definition} satisfies
\begin{enumerate}
\item $F(1/t) = 2-F(t)$, $F'(1/t) = t^2F'(t)$; 
\item $F\in C([0,+\infty))$, $F(0) = 0$, $F(1) = 1$, $\lim_{t\rightarrow+\infty}F(t)=2$, $\lim_{t\rightarrow0}F(t)/t = 0$;
\item $F'(0)=0$ and $F'(t)>0$ for $t> 0$.
\end{enumerate}
\end{lemma}

\begin{proof} Property ($1$) is straightforward to check. ($2$) follows from the Taylor expansion of $F(t)$ and property ($1$). ($3$) follows from the Taylor expansion of $F'(t)$ and property ($1$).
\end{proof}

\subsection{Special function $G$}\label{sec:G}

We define
\begin{equation}\label{eqt:G_definition}
G(t):= \frac{3t^4 - 2t^2 - 1}{8t^3}\ln\left|\frac{t+1}{t-1}\right| + \frac{1}{4t^2} + \frac{7}{12}, \quad t\geq 0.
\end{equation}
The derivative of $G$ reads
\[G'(t) = \frac{3t^4 + 2t^2 + 3}{8t^4}\ln\left|\frac{t+1}{t-1}\right| - \frac{3t^2+3}{4t^3}.\]
For $t\in[0,1)$, $G(t)$ and $G'(t)$ have the Taylor expansions 
\[G(t) = \suml_{n=1}^{+\infty}\frac{4(n+1)t^{2n}}{(2n-1)(2n+1)(2n+3)},\quad G'(t) = \suml_{n=1}^{+\infty}\frac{8n(n+1)t^{2n-1}}{(2n-1)(2n+1)(2n+3)}.\]
For $t\in[0,1)$, $G(1/t)$ and $G'(1/t)$ have the Taylor expansions 
\[G(1/t) = \frac{4}{3}-\suml_{n=1}^{+\infty}\frac{4nt^{2n+2}}{(2n-1)(2n+1)(2n+3)},\quad G'(1/t) = \suml_{n=1}^{+\infty}\frac{8n(n+1)t^{2n+3}}{(2n-1)(2n+1)(2n+3)}.\]

\begin{lemma}\label{lem:G_property} The function $G$ defined in \eqref{eqt:G_definition} satisfies
\begin{enumerate}
\item $G'(1/t) = t^4G'(t)$;
\item $G\in C([0,+\infty))$, $G(0) = 0$, $G(1) = 5/6$, $\lim_{t\rightarrow+\infty}G(t)=4/3$, $\lim_{t\rightarrow0}G(t)/t=0$;
\item $G'(t)\geq 0$ for $t\geq 0$.
\item $(4t/3-tG(1/t))' = tF'(1/t)$ for $t\geq 0$.
\end{enumerate}
\end{lemma}

\begin{proof} Properties $(1)$ is straightforward to check. $(2)$ follows from the Taylor expansions of $G(t)$ and $G(1/t)$. $(3)$ follows from the Taylor expansion of $F'(t)$ and property $(1)$. $(4)$ can be checked straightforwardly by the definitions of $G(t)$ and $F(t)$.
\end{proof}

\subsection{Special functions $F_i$}\label{sec:F_i}
Based on the special function $F$ in Appendix \ref{sec:F}, we introduce a series of functions $F_i(t), t\geq0, i=1,2,3,4$ that appear in the proof of Theorem \ref{thm:k_bound}:
\begin{equation}\label{eqt:F_i}
\begin{split}
F_1(t) &:= \int_0^tsF'(s)\idiff s,\\
F_2(t) &:= t^{-1}F_1(t) + tF_1(1/t),\\
F_3(t) &:= \int_0^ts^2F_2(s)\idiff s,\\
F_4(t) &:= t^3\int_0^{1/t}s^5F_3(1/s)\idiff s.
\end{split}
\end{equation}
It is not hard to check that, for $t>0$,  
\[F_3'(1/t) = t^{-4}F_3'(t),\]
which immediately leads to 
\[F_4(t) = F_4(1/t) = \frac{1}{6}\left(t^{-3}F_3(t) + t^3F_3(1/t)\right).\]
Using the Taylor expansions of $F$ and properties of $F$ in Lemma \ref{lem:F_property}, we can obtain the Taylor expansions of each $F_i$: for $t\in[0,1]$,
\[F_1(t) = \suml_{n=1}^\infty\frac{4nt^{2n+1}}{(2n-1)(2n+1)^2},\quad F_1(1/t) = \frac{\pi^2}{4} - \suml_{n=1}^\infty\frac{4nt^{2n-1}}{(2n-1)^2(2n+1)};\]
\[F_2(t) = F_2(1/t) = \frac{\pi^2}{4}t - \suml_{n=1}^\infty\frac{8nt^{2n}}{(2n-1)^2(2n+1)^2};\]
\[F_3(t) = \frac{\pi^2}{16}t^4 - \suml_{n=1}^\infty\frac{8nt^{2n+3}}{(2n-1)^2(2n+1)^2(2n+3)},\quad F_3(1/t) = \frac{\pi^2}{8}t^{-2} + \suml_{n=1}^\infty\frac{8nt^{2n-3}}{(2n-3)(2n-1)^2(2n+1)^2};\]
\[F_4(t) = F_4(1/t) = \frac{\pi^2}{32}t + \suml_{n=1}^\infty\frac{8nt^{2n}}{(2n-3)(2n-1)^2(2n+1)^2(2n+3)}. \]
An elementary calculation shows that $F_4'(t)\geq F_4'(1) = 0$ for $t\in [0,1]$. Hence, the maximum of $F_4(t)$ is achieved at $t=1$ with
\[F_4(1) = \frac{\pi^2}{32} + \suml_{n=1}^\infty\frac{8n}{(2n-3)(2n-1)^2(2n+1)^2(2n+3)} = \frac{\pi^2}{32} - \frac{1}{6},\]
which is used in the proof of Theorem \ref{thm:k_bound}.

\subsection{Special functions $(1-x^2 - p)_+ + p$}\label{sec:f_m}
The functions $f_{m,p} := (1-x^2 - p)_+ + p$ with $p\in[0,1-\eta/4)$ are a special family in the function set $\mathbb{D}$ that satisfy $b(f_{m,p})/c(f_{m,p}) = (1-p)/3$. In particular, $f_m := f_{m,0} = (1-x^2)_+$ is the minimal function in $\mathbb{D}$ in the sense that $f(x)\geq f_m(x)$ for all $x$ and all $f\in \mathbb{D}$. We have used the following properties of $f_{m,p}$ in our preceding arguments.  

First, we can compute that, for $x\geq 1$,
\begin{align*}
\mtx{T}(f_{m,p})(x) + b(f_{m,p}) &= \frac{1}{\pi}\int_0^{+\infty}f_{m,p}'(y)\cdot y(F(x/y)-2)\idiff y\\
&= \frac{2}{\pi}\int_0^{\sqrt{1-p}}y^2F(y/x)\idiff y \\
&= \frac{2}{\pi}x^3\cdot \int_0^{\sqrt{1-p}/x}t^2F(t)\idiff t \\
&\geq \frac{4}{15\pi}\cdot\frac{(1-p)^{5/2}}{x^2}.
\end{align*}
In particular, for $x\geq 1$,
\[\mtx{T}(f_m)(x) + b(f_m) \geq \frac{4}{15\pi}x^{-2},\]
which has been used in the proof of Theorem \ref{thm:solution_type} part $(1)$.

Based on the estimates above, we find that for $x\geq 1$,
\begin{align*}
\mtx{R}_1(f_{m,p})(x) &= \mtx{T}_1(f_{m,p})(x) \\
&= \left(1 + \frac{3\mtx{T}(f_{m,p})(x)}{c(f_{m,p})}\right)_+ \\
&\geq \left(1 - \frac{3b(f_{m,p})}{c(f_{m,p})} + \frac{4}{5\pi}\cdot \frac{(1-p)^{5/2}}{c(f_{m,p})x^2}\right)_+ \\
&= \left(p + \frac{4}{5\pi}\cdot \frac{(1-p)^{5/2}}{c(f_{m,p})x^2}\right)_+ \\
&= p + \frac{4}{5\pi}\cdot \frac{(1-p)^{5/2}}{c(f_{m,p})x^2} \\
&> p.
\end{align*}
However, $f_{m,p}(x) \equiv p$ for $x\geq 1$, which means that $f_{m,p}$ cannot be a fixed point of $\mtx{R}_1$. We have used this fact in the proofs of Lemma \ref{lem:f_infinity} and Lemma \ref{lem:finite_bf}.\\

Next, we show that $\mu(f_m) = 2Q(f_m)/b(f_m)^2 = \overline{\mu}$, where $\overline{\mu}$ is given in Theorem \ref{thm:k_bound}, and $\mu(f),Q(f)$ are defined in the proof of this theorem. Owing to the calculations in the proof of Theorem \ref{thm:k_bound}, for $f\in \mathbb{D}$, we have
\[Q(f) = \frac{1}{\pi^2}\int_0^{+\infty}\int_0^{+\infty}\left(\frac{f'(x)}{x}\right)'\left(\frac{f'(y)}{y}\right)'x^3y^3F_4(x/y)\idiff x\idiff y,\]
and
\[b(f) =  \frac{2}{3\pi}\int_0^{+\infty}\left(\frac{f'(y)}{y}\right)'y^3\idiff y.\]
Note that $f'_m(x) = -2x, x\in [0,1)$ and $f'_m(x) = 0, x>1$. We thus have
\[\left(\frac{f'(x)}{x}\right)' = 2\cdot \delta(x-1),\quad x\geq 0,\]
where $\delta(x)$ is the Dirac function centered at $0$. It then follows that 
\[Q(f_m) = \frac{4}{\pi^2}F_4(1),\quad b(f_m) = \frac{4}{3\pi},\]
and thus 
\[\mu(f_m) = \frac{2Q(f_m)}{b(f_m)^2} = \frac{9}{2}F_4(1) =: \overline{\mu}.\]
Recall we have shown in the proof of Theorem \ref{thm:k_bound} that $Q(f)\leq \overline{\mu}$ for all suitable $f$. This means that $f_m$ is the maximizer of $\mu(f)$ over the set $\mathbb{D}$.

\section{On the Hilbert transform}\label{sec:Hilbert_transform}

We prove two useful lemmas that exploit properties the Hilbert transform.

\begin{lemma}\label{lem:Hilbert_property1}
For any suitable function $\om$ on $\mathbb{R}$,
\[\frac{\mtx{H}(\om)(x)-\mtx{H}(\om)(0)}{x} = \mtx{H}\left(\frac{\om - \om(0)}{x}\right)(x).\]
As a result, 
\[\frac{1}{\pi}\int_{\mathbb{R}}\frac{\mtx{H}(\om)(x)\cdot \om(x)}{x}\idiff x = \frac{1}{2}\om(0)^2 -\frac{1}{2}\big(\mtx{H}(\om)(0)\big)^2.\]
\end{lemma}

\begin{proof}
The first equation follows directly from the definition of the Hilbert transform on the real line. The second equation is derived from the first one as follows:
\begin{align*}
\frac{1}{\pi}\int_{\mathbb{R}}\frac{\mtx{H}(\om)\cdot \om}{x}\idiff x &= \frac{1}{\pi}\int_{\mathbb{R}}\frac{(\mtx{H}(\om)-\mtx{H}(\om)(0))}{x}\cdot \om\idiff x + \mtx{H}(\om)(0)\cdot \frac{1}{\pi}\int_{\mathbb{R}}\frac{\om}{x}\idiff x\\
&= \frac{1}{\pi}\int_{\mathbb{R}}\mtx{H}\left(\frac{\om-\om(0)}{x}\right)\cdot \om\idiff x - \big(\mtx{H}(\om)(0)\big)^2\\
&= -\frac{1}{\pi}\int_{\mathbb{R}}\frac{\om-\om(0)}{x}\cdot \mtx{H}(\om)\idiff x - \big(\mtx{H}(\om)(0)\big)^2\\
&= -\frac{1}{\pi}\int_{\mathbb{R}}\frac{\om\cdot \mtx{H}(\om)}{x}\idiff x + \om(0)\cdot \frac{1}{\pi}\int_{\mathbb{R}}\frac{\mtx{H}(\om)}{x}\idiff x - \big(\mtx{H}(\om)(0)\big)^2\\
&= -\frac{1}{\pi}\int_{\mathbb{R}}\frac{\mtx{H}(\om)\cdot \om}{x}\idiff x + \om(0)^2 - \big(\mtx{H}(\om)(0)\big)^2.
\end{align*}
Rearranging the equation above yields the desired result.
\end{proof}

\begin{lemma}\label{lem:Hilbert_property2}
Given a function $\om$, suppose that $\|x^\delta\om\|_{L^{+\infty}(\mathbb{R})} = \sup_{x\in \mathbb{R}}|x|^{\delta}|\om(x)|<+\infty$ for some $\delta>0$. If $\om\in H^k_{loc}(A,B)$ for some $A<B$ and some integer $k\geq 0$, then $\mtx{H}(\om)\in H^k_{loc}(A,B)$. 
\end{lemma}

\begin{proof}
We first prove a formula for the $k$-th derivative of $\mtx{H}(\om)$: if $\om\in H^k_{loc}(A,B)$, then for any $A<a<b<B$ and any $x\in (a,b)$, 
\begin{equation}\label{eqt:Hilbert_induction}
\mtx{H}(\om)^{(k)}(x) =  \mtx{H}(\chi_{[a,b]}\om^{(k)})(x) + g_{a,b}^{(k)}(x) + \suml_{j=0}^{k-1}f_{a,b,j}^{(k-j)}(x),
\end{equation}
where the summation is $0$ if $k=0$, and 
\[\mtx{H}(\chi_{[a,b]}\om^{(k)})(x) = \frac{1}{\pi}P.V.\int_a^b\frac{\om^{(k)}(y)}{x-y}\idiff y,\]
\[g_{a,b}(x) := \frac{1}{\pi}\int_{-\infty}^a\frac{\om(y)}{x-y}\idiff y + \int_b^{+\infty}\frac{\om(y)}{x-y}\idiff y,\]
\[f_{a,b,j}(x) := \frac{1}{\pi}\left(\om^{(j)}(a)\ln|x-a| - \om^{(j)}(b)\ln|x-b|\right),\quad j=0,1,2,\dots.\]
We prove this formula with induction. The base case $k=0$ is trivial: 
\[\mtx{H}(\om)(x) = \frac{1}{\pi}P.V.\int_{-\infty}^{+\infty}\frac{\om(y)}{x-y}\idiff y = \frac{1}{\pi}P.V.\int_a^b\frac{\om(y)}{x-y}\idiff y + g_{a,b}(x).\]
Now suppose that \eqref{eqt:Hilbert_induction} is true for some integer $k\geq 0$, we need to show that it is then also true for $k+1$. Under the assumption that $\om\in H^{k+1}_{loc}(A,B)$, we can use integration by parts to rewrite the first term on the right-hand side of \eqref{eqt:Hilbert_induction} as
\begin{align*}
\mtx{H}(\chi_{[a,b]}\om^{(k)})(x) &= \frac{1}{\pi}P.V.\int_a^b\frac{\om^{(k)}(y)}{x-y}\idiff y \\
&= -\frac{1}{\pi}\om^{(k)}(y)\ln|x-y|\Big|_{y=a}^{y=b} + \frac{1}{\pi}\int_a^b\om^{(k+1)}\ln|x-y|\idiff y\\
& = f_{a,b,k}(x) + \frac{1}{\pi}\int_a^b\om^{(k+1)}\ln|x-y|\idiff y.
\end{align*}
Note that $\om^{(k)}(a)$ and $\om^{(k)}(b)$ are finite because $\om\in H^{k+1}_{loc}(A,B)$. It then follows from the inductive assumption that, for $x\in (a,b)$,
\begin{align*}
\mtx{H}(\om)^{(k+1)}(x) &= \left(\mtx{H}(\chi_{[a,b]}\om^{(k)})(x)\right)' + \left(g_{a,b}^{(k)}(x) + \suml_{j=0}^{k-1}f_{a,b,j}^{(k-j)}(x)\right)'\\
&= f_{a,b,k}'(x) + \left(\frac{1}{\pi}\int_a^b\om^{(k+1)}\ln|x-y|\idiff y\right)' + \left(g_{a,b}^{(k)}(x) + \suml_{j=0}^{k-1}f_{a,b,j}^{(k-j)}(x)\right)'\\
&= f_{a,b,k}'(x) + \frac{1}{\pi}P.V.\int_a^b\frac{\om^{(k+1)}}{x-y}\idiff y + g_{a,b}^{(k+1)}(x) + \suml_{j=0}^{k-1}f_{a,b,j}^{(k+1-j)}(x)\\
&= \frac{1}{\pi}P.V.\int_a^b\frac{\om^{(k+1)}}{x-y}\idiff y + g_{a,b}^{(k+1)}(x) + \suml_{j=0}^kf_{a,b,j}^{(k+1-j)}(x).
\end{align*}
Hence, \eqref{eqt:Hilbert_induction} is also true for $k+1$. This completes the induction.

We then use \eqref{eqt:Hilbert_induction} to prove the lemma. Note that under the assumptions of the lemma, it is easy to see that $g_{a,b}(x)$ and $f_{a,b,j}(x), j=0,1,\dots, k-1,$ are infinitely smooth in the interior of $(a,b)$, and thus $g_{a,b},f_{a,b,j}\in H^k_{loc}(a,b)$. As for the first term on the right-hand side of \eqref{eqt:Hilbert_induction}, we have
\[\left\|\mtx{H}(\chi_{[a,b]}\om^{(k)})\right\|_{L^2([a,b])} \leq \left\|\mtx{H}(\chi_{[a,b]}\om^{(k)})\right\|_{L^2(\mathbb{R})} = \left\|\chi_{[a,b]}\om^{(k)}\right\|_{L^2(\mathbb{R})} = \|\om\|_{\dot{H}^k([a,b])}<+\infty.\]
Therefore, \eqref{eqt:Hilbert_induction} implies that $\mtx{H}(\om)\in H^k_{loc}(a,b)$. Since this is true for any $A<a<b<B$, we immediately have $\mtx{H}(\om)\in H^k_{loc}(A,B)$.
\end{proof}

\subsection*{Acknowledgement} The authors are supported by the National Key R\&D Program of China under the grant 2021YFA1001500.

\bibliographystyle{myalpha}
\bibliography{reference}

\end{document}